\documentclass[11pt]{article}
\usepackage{bbm}
\usepackage{amssymb}
\usepackage{amssymb, amsthm, amsmath, amscd}
\usepackage[usenames,dvipsnames]{color}
\usepackage{framed}

\newtheorem{thm}{Theorem}[section]
\newtheorem{lem}[thm]{Lemma}
\newtheorem{prop}[thm]{Proposition}
\newtheorem{cor}[thm]{Corollary}
\theoremstyle{definition}

\theoremstyle{definition}
\newtheorem{df}[thm]{Definition}
\theoremstyle{definition}
\newtheorem{rem}[thm]{Remark}
\newtheorem{nota}[thm]{Notation}
\theoremstyle{definition}
\newtheorem{exm}[thm]{Example}
\newcommand{\red}{\textcolor{red}}
\newcommand{\blue}{}
\newcommand{\green}{}

\renewcommand{\phi}{\varphi}

\definecolor{purple}{RGB}{150,10,200} 

\newcommand{\N}{\mathbb{N}}
\newcommand{\Z}{\mathbb{Z}}

\newcommand{\R}{\mathbb{R}}
\newcommand{\C}{\mathbb{C}}

\numberwithin{equation}{section}

\newcommand{\id}{\operatorname{id}}

\newcommand{\hm}{homomorphism}
\newcommand{\dt}{\delta}
\newcommand{\ep}{\varepsilon}

\newcommand{\1}{\mbox{\large \bf 1}}

\newcommand{\gm}{\gamma}

\newcommand{\la}{\langle}
\newcommand{\ra}{\rangle}
\newcommand{\andeqn}{\,\,\,{\rm and}\,\,\,}
\newcommand{\rforal}{\,\,\,{\rm for\,\,\,all}\,\,\,}
\newcommand{\CA}{$C^*$-algebra}
\newcommand{\SCA}{$C^*$-subalgebra}

\newcommand{\af}{{\alpha}}
\newcommand{\bt}{{\beta}}

\newcommand{\wilog}{without loss of generality}
\newcommand{\Wlog}{Without loss of generality}

\newcommand{\beq}{\begin{eqnarray}}
\newcommand{\eneq}{\end{eqnarray}}

\usepackage{amsfonts}
\usepackage{mathrsfs}
\usepackage{textcomp}
\usepackage[all]{xy}

\newcommand{\pinr}{\in \mathbb{R}_+\backslash \{0\}}
\newcommand{\dimnuc}{\dim_{\mathrm{nuc}}}

\newcommand{\trdimnuc}{\mathrm{Trdim_{nuc}}}

\newcommand{\Her}{\mathrm{Her}}

\newcommand{\calF}{\mathcal{F} }
\newcommand{\bbN}{\mathbb{N} }
\newcommand{\calZ}{\mathcal{Z}}
\newcommand{\JS}{$\mathcal{Z}$}

\newcommand{\xrule}{\bigskip \hrule \bigskip}

\newcommand{\CAs}{$C^*$-algebras}

\title{
Tracial approximation in simple \CA s}
\author{Xuanlong Fu\footnote{Email: xlf@fudan.edu.cn}\, and
Huaxin Lin\footnote{Email: hlin@uoregon.edu}
 }
\date{
}

\begin{document}

\maketitle


\begin{abstract}

%
%
We revisit the notion of tracial approximation for unital simple \CA s.
We show that a unital
simple separable $C^*$-algebra $A$  is
asymptotically
tracially in the class of \CA s with
{{finite}} nuclear dimension
if  and only if {\blue{$A$}} is
asymptotically tracially in the class of simple nuclear ${\cal Z}$-stable \CA s.
\end{abstract}

\section{Introduction}

Nuclear dimension for \CA s was first introduced in \cite{WZ2010}.
Over the time, this notion becomes increasingly important
in the study of \CA s in the connection with the Elliott program
\cite{Ell94},
the program of {\blue{classification of}}
separable simple 
amenable
\CA s by the Elliott invariant,  a set of $K$-theory related invariant.  The part of  the Toms-Winter conjecture
(see \cite[Conjecture 9.3]{WZ2010})
states
that a unital simple {\blue{nuclear}} separable \CA\ $A$ has
finite nuclear dimension if and only if $A$ is ${\cal Z}$-stable, i.e.,
$A\otimes {\cal Z}\cong A,$ where ${\cal Z}$ is the Jiang-Su algebra, 
{{a}} unital  separable and {\blue{infinite}} dimensional simple \CA\,
which has $K_0({\cal Z})=\Z$ (as an ordered group),
$K_1({\green{{\cal Z}}})=\{0\}$
and a unique tracial state
{\green{(see \cite{JS1999}).}}
This part of
the Toms-Winter conjecture is now a theorem
(see \cite{W-2012-pure algebras}, \cite{CETWW-2019}{\green{, see also}}
\cite{MS2014}).

On the other hand tracial rank was introduced in
\cite{Lin01-2} (and see also \cite{Lin 2001-1}).
\CA s with tracial rank zero are also called
\CA s which are tracially AF.
 Amenable tracially AF-algebras and \CA s of tracial rank one were classified in  \cite{LinDuke} and \cite{Lin-TR-One}
with the
presence of UCT
({\blue{These classification}}   results  {\blue{were preceded}} by
\cite{EG96} and
\cite{EGLi}, respectively{{).}}
{\blue{These}} had been generalized to the classification
of  the class of amenable simple \CA s which have  rationally generalized tracial rank at most one  satisfying the UCT
{\blue{(see \cite{GLN2015}, \cite{GLN2019}, and \cite{GLN2019-2}, see also \cite{Lin-Invent-2011}, \cite{Winter2014}, \cite{L2014localizing}{\green{,}}
and  \cite{LN-Adv}).}}
In {{\cite{EGLN2015}}},
it is proved
that all unital separable simple \CA s with finite
nuclear dimension
in the UCT class
in fact  have rationally generalized
tracial rank at most one
{\blue{(using \cite{TWW2015})}}.
In other words, all unital separable simple \CA s with finite nuclear dimension satisfying the UCT
are classified (up to isomorphism) by their Elliott invariant.  This can also be restated, by the proof of Toms-Winter conjecture as mentioned above,
that all
{\green{unital}}
amenable simple ${\cal Z}$-stable \CA s satisfying the UCT are classified.

The beginning point of this paper is to search a tracial version of Toms-Winter conjecture (which is now a theorem).
We {\blue{revisit}} a version of tracial approximation  (see {{Definition}} \ref{def-tracial_approximation} and
{{Proposition}} \ref{equivalent-definition-of-tracial-approximation} below).
The main results include the following statement: A unital {\blue{separable}} simple \CA\ $A$ which is
asymptotically tracially in ${\cal N}_{\cal Z}$
 (the class of
all nuclear ${\cal Z}$-stable \CA s) if and only if $A$ is
asymptotically tracially
in ${\cal N}_n$ (the class of {\green{all}} 
\CA s with nuclear dimension
at most $n$) for some integer $n\ge {{0}}$ (see Theorem \ref{tdimnuc-finite-equal-t-nuclear-Z-stable} below).
It is also shown that a unital separable simple \CA\ $A$ which is
asymptotically tracially in {{${\cal C}_{{\cal Z},s}$}}
(the class of all
{{separable}}
${\cal Z}$-stable \CA s), 
{{is}} either purely infinite, or
has stable rank one (see {{Theorem}} \ref{dichotomy-tracial-finite-dimnuc-finite-infinite}). Moreover,  $A$ has strict comparison
{\blue{(for positive elements).}}   Furthermore, it is shown that if $A$ is a unital {\blue{separable}} simple \CA\
which is
asymptotically tracially in ${\cal N}$
(the class of all nuclear \CA s) and $A$ is
asymptotically tracially in ${\cal C}_{{\cal Z},s},$ then $A$ is asymptotically tracially
in ${\cal N}_{\cal Z}$
{\blue{(see {{Theorem}}
\ref{Thm-tracial-simple-nuclear-Z-stable}).}}
{\blue{As one expects, in the case that $A$ is a unital separable nuclear simple \CA, then $A$ is asymptotically tracially in ${\cal N}_n$
if and only if it is
asymptotically tracially in ${\cal C}_{{\cal Z},s},$ and, if and only if $A$ has finite nuclear dimension and ${\cal Z}$-stable.}}
A number of other related results are also obtained.
In {\blue{Example}} \ref{EXLast}, a large number of unital non-exact separable simple \CA s which are
asymptotically {{tracially}}
in ${\cal N}_n$ are presented.
{\blue{It should be mentioned that {\green{if}} a unital simple \CA\ {\green{$A$}} is asymptotically tracially in the class of finite dimensional
\CA s then $A$ has tracial rank zero, and {\green{if}} $A$ is asymptotically tracially in the class of \CA s which are 
1-dimensional {\green{NCCW}} complexes {\green{then}} $A$ has generalized tracial rank at most one.}}


The organization of this paper is as follows.
Section 2 serves as a preliminary. We fix some frequently used notations and concepts {\blue{there.}}
Section 3 {\blue{studies}} some basic properties of
{\green{asymptotical}} tracial approximation.
Section 4 gives {\blue{some}} useful properties that  are preserved by
{\green{asymptotical}} tracial approximation.
{\blue{One of the
{\green{results}}
is that, if $A$ is a unital separable simple \CA\  which is
asymptotically tracially in the class of exact \CA s, then
every $2$-{\green{quasitrace}} of $A$ is a trace {{(see
{{Corollary}} \ref{CQT=T}).}}}}
{\blue{Section 5 is a preparation for Section 6 which}}
gives a sufficient and necessary condition
{\blue{for}} a {{c.p.c.~}}generalized inductive limit
{\blue{to have}} finite nuclear dimension
(Theorem \ref{gene-lim-fin-nucdim}).
{\blue{Section 7 shows that every unital infinite dimensional separable simple
\CA\ which is asymptotically tracially in ${\cal N}_n$
is asymptotically tracially in ${\cal N}_{\cal Z}$}}
(see Theorem
\ref{tracial-finite-dimnuc-implies-tracial-Z-stable}).
%
{\blue{In
Section 8, we show that a {\blue{separable simple}}
 unital infinite dimensional \CA\ which is
asymptotically tracially in ${\cal N}$
and is also
asymptotically tracially in  ${\cal C}_{{\cal Z},s},$
then it is
asymptotically tracially in ${\cal N}_{\cal Z}$}}
(Theorem \ref{Thm-tracial-simple-nuclear-Z-stable}).
{\blue{In Section 9, we summarize and combine some of the results.}}
%
Theorem \ref{tensor-of-tracial-approximation}
{\blue{shows}} that
{\green{asymptotical}} tracial approximation behaves well
under  {\blue{the}} spatial
{{tensor products.}}  {{As a consequence, a variety of examples can be produced.}}
{\blue{For example, if $A$ is any unital separable simple \CA\ and $B$ is a unital infinite dimensional separable simple \CA\
which is
asymptotically tracially in the class of ${\cal Z}$-stable {{\CAs}},
then the spatial  tensor product $A\otimes B$ is
asymptotically tracially
in the class of ${\cal Z}$-stable \CA s. If both $A$ and $B$ are
asymptotically tracially in
{\green{${\cal N}_n,$}}
then the spatial  tensor  product $A\otimes B$ is {{also}}
asymptotically tracially in
{\green{${\cal N}_n.$}}}}

\vspace{0.2in}

{\bf Acknowledgements:}   The second named author was partially supported by
an NSF grant (DMS-1665183). Both authors acknowledge the support from the Research Center for Operator Algebras, 
East China Normal University.

\section{Preliminary}


\begin{nota}\label{DB1}
Let $X$ be a normed space {{and}}  $0\leq r\leq s$ {{be real numbers.}}
{{Set}} $B_{r,s}(X):=\{x\in X: r\leq \|x\|\leq s\}.$
Denote  {{by $X^1$}} the closed unit ball  {{$B_{0,1}(X)$}} of $X.$

Let $a,b\in X$ {\blue{and let}}  $\epsilon>0$,
we  {{write}} $a\approx_{\epsilon}b$ if
$\|a-b\|< \epsilon$.
%
{{Let}} $Y,Z\subset X$ {\blue{and let}} $\epsilon>0$, we say $Y$ is  {\blue{an}}  $\epsilon$-net of $Z$,
and denoted by $Z\subset_\epsilon Y$,
if, {\blue{for all}} $z\in Z$, {\blue{there is}} $y\in Y$ such that $z\approx_{\epsilon}y$.

\end{nota}

\begin{nota}

{{Let  $A$ and $B$ be \CA s,}} let $\phi: A\rightarrow B$ be a map,
let $\calF\subset A$,  and let $\epsilon> 0$.
{{The map}} $\phi$ is called $(\calF,\epsilon)$-
multiplicative,
or called $\epsilon$-multiplicative on $\calF$,
if for any $x,y\in \calF$,
$\phi(xy)\approx_{\epsilon}\phi(x)\phi(y)$.
If, in addition,
for any $x\in \calF$, $\|\phi(x)\|\approx_{\epsilon}\|x\|$,
then we say $\phi$ is {{an}}
$(\calF,\epsilon)$-{\blue{approximate}} embedding.

\end{nota}

\begin{nota}

Let $A$ {\blue{and}} $B$ be \CAs.  {\blue{The}} spatial tensor product of $A$ and $B$
is denoted by $A\otimes B$.
\end{nota}

\begin{nota}\label{NoGF}
{{Let $\N=\{1,2{{,\cdots}}\}$ be the set of natural numbers.
Let $M_k$ denote the algebra of $k$ by $k$
{{complex}} matrices ($k\in\N$).}}
Let $\{e_{i,j}^{(k)}\}$ denote
{{the canonical}}
matrix units of
$M_k$
{\blue{($1\leq i,j\leq k$ and  $k\in \N$).}}
If $F$ is a finite dimensional  \CA,
then
{\blue{${\cal G}^F$ denotes}} the standard generating set of $F$,
i.e., the union of
{{canonical}} matrix units of each direct summand of $F$.
Note that the standard generating set of $F$ is in the unit ball of $F$ and
is also a linear generating set .

\end{nota}

\begin{nota}
Let $A$ be a $C^*$-algebra {\blue{and}}
let $S,T\subset A$ be subsets of {\blue{$A.$}}
{{Set}} $S\cdot T:=\{st:s\in S, t\in T\}$ and set
$S^{\bot}:=\{a\in A: as=0=sa, \forall s\in S\}.$
Let $\overline{S}^{\|\cdot\|}$ be the norm closure of $S$.
%
Denote by
$\Her_A(S)$ {\blue{(or just $\Her(S)$)}}  the hereditary $C^*$-subalgebra of $A$ generated by $S$.
Let $C^*(S)$ be the $C^*$-subalgebra of $A$ generated by $S$.
{\blue{Denote by}} $A_+$ the set of all positive elements in $A,$
{{by $A_+^1:=A_+\cap A^1,$}}
and by
$A_{sa}$ the set of all self-adjoint elements in $A.$  Denote by
$\mathcal{M}(A)$  the multiplier algebra of $A$.
For $x\in A$, the spectrum of $x$ is denoted by $\mathrm{sp}_A(x),$ or just $\text{sp}(x).$

\end{nota}

\begin{nota}
\label{notation-2}

Let $A_i$ be \CAs\ ($i\in\mathbb{N}$).
{\blue{Set}}
$\prod_{i=1}^{\infty}A_i:=\{\{a_1,{\blue{a_2,}}\cdots\}:
a_i\in A_i$, $\sup_{i\in\mathbb{N}}\|a_i\|<\infty\},$
and set
$\bigoplus_{i=1}^{\infty}A_i:=\{\{a_1,{\blue{a_2,}}\cdots\}:
a_i\in A_i$, $\lim_{i\rightarrow \infty}\|a_i\|=0\}$.
Denote by $\pi_\infty:
\prod_{i=1}^{\infty}A_i\to {{\prod_{i=1}^{\infty}A_i/\bigoplus_{i=1}^{\infty}A_i}}$ the quotient map.
We also use the notation
$l^{\infty}(A):=\prod_{i=1}^{\infty}A$ and
$c_0(A):=\bigoplus_{i=1}^{\infty}A$.
Define $\iota: A\to l^\infty (A)$ by $\iota(a)=\{a,a, {\blue{\cdots}}\},$
the constant sequence, for all $a\in A.$
Define $\iota_A=\pi_\infty\circ \iota.$

{
Let $h: B\rightarrow \prod_{i=1}^{\infty}A_i/ \bigoplus_{i=1}^{\infty}A_i$
be a *-homomorphism. {\blue{The map}} $h$ is called a {\blue{\it strict embedding,}}
if for any $b\in B$, there exists $\{b_1,{\blue{b_2,}}\cdots\}\in\prod_{i=1}^{\infty}A_i$
such that $h(b)={\blue{\pi_\infty(\{b_1,{\blue{b_2,}}\cdots\})}}$ and
$\|b\|=\liminf_{i\rightarrow \infty}\|b_i\|$.
If $C\subset \prod_{i=1}^{\infty}A_i/ \bigoplus_{i=1}^{\infty}A_i$
is a $C^*$-subalgebra  {\blue{and}} the embedding map
${\blue{\iota}}: C\hookrightarrow \prod_{i=1}^{\infty}A_i/ \bigoplus_{i=1}^{\infty}A_i$
is a {\blue{strict}} embedding, then we say $C$ is strictly embedded.


(1) Note {\blue{that,}} if $C\subset l^\infty(A)/c_0(A)$ is full in
$l^\infty(A)/c_0(A),$ then $C$ is strictly embedded {\blue{(see also
Proposition \ref{strictly-embedding-simple-case}).}}

(2) For a \CA\ $A$, the map $\iota_A$ defined above is a strict embedding,
and the map $\hat\iota: A\rightarrow l^\infty(A)/c_0(A)$,
$a\mapsto \pi_\infty(\{a,0,a,0,a,0,\cdots\})$ is not.
}
\end{nota}

\begin{prop}
\label{strictly-embedding-simple-case}
Let $A_1,$ $A_2,\cdots$ be \CAs\, {\blue{and}} let  $A$ be a simple \CA.
Let $h: A\rightarrow \prod_{i=1}^{\infty}A_i/ \bigoplus_{i=1}^{\infty}A_i$
be an embedding.
If for some nonzero element $a\in A\backslash\{0\}$,
there exists $a_i\in A_i$ ($i\in \N$) such that
$h(a)={\blue{\pi_\infty(\{a_1,a_2,\cdots\})}}$
and $\liminf_{i\rightarrow \infty}\|a_i\|>0$ hold,
then $h$ is a strict embedding.

\end{prop}

\begin{proof}

If $h$ is not a {\blue{strict}} embedding, then {{we can choose}}
$c\in A$ and natural numbers $i_1<i_2<\cdots$,
{\blue{and}} $c_i\in A_i$ ($i\in\N$) such that
$h(c)={\blue{\pi_\infty(\{c_1,{\blue{c_2,}}\cdots\})}}$ and
$\lim_{n\rightarrow\infty}\|c_{i_n}\|<\|c\|.$

{\blue{Let $\pi_1: \prod_{i=1}^{\infty}A_i/ \bigoplus_{i=1}^{\infty}A_i
\rightarrow
\prod_{n=1}^{\infty}A_{i_n}/ \bigoplus_{n=1}^{\infty}A_{i_n}$ be the quotient map}} 
induced by the
{\blue{quotient}} map
$\pi_0:\prod_{i=1}^{\infty}A_i\to \prod_{n=1}^{\infty}A_{i_n}.$
By the assumption of this proposition,
$\|\pi_1\circ h(a)\|=\|{\blue{\pi_\infty(\{a_{i_1},a_{i_2},\cdots\})}}\|
={\blue{\lim\inf}}_{n\rightarrow\infty}
\|a_{i_n}\|>0.
$
{\blue{It follows that}} ${\blue{\pi_1}}\circ h$ is a nonzero *-homomorphism.
Since $A$ is simple,
$\pi_1\circ h$ is an embedding.
However, {{by the choice of $c,$}}
we have
$\|\pi_1\circ h(c)\|=\|{\blue{\pi_\infty(\{c_{i_1},c_{i_2},\cdots\})}}\|
=\lim_{n\rightarrow\infty}
\|c_{i_n}\|<\|c\|,$
which is {\blue{contradicted to that}} $\pi_1\circ h$ is an embedding.
Thus $h$ is a strict embedding.

\end{proof}

\begin{nota}
Let $\epsilon>0.$ {\blue{Define}} a continuous function
$f_{\epsilon}: {{[0,+\infty)}}
\rightarrow [0,1]$ {\blue{by}}
$$
f_{\epsilon}(t)=
\left\{\begin{array}{ll}
0  &t\in [0,\epsilon],\\
1  &t\in [2\epsilon,\infty),\\
\mathrm{linear } &{t\in[\epsilon,2\epsilon].}
\end{array}\right.
$$

\end{nota}

\begin{nota}
Let $\phi: A\rightarrow B$ be a  linear map.
{{The map}} $\phi$ is positive, if $\phi(A_+)\subset B_+$ {\blue{and}}
$\phi$ is completely positive, abbreviated as c.p.,
if $\phi\otimes \mathrm{id}: A\otimes M_n\rightarrow B\otimes M_n$
are positive for all $n\in\mathbb{N}.$
{\blue{{{If}} $\phi$ is positive linear and $\|\phi\|\le 1,$ then it is called positive 
{{contractive,}}
abbreviated as p.c.,}}
if $\phi$ is c.p.~and $\|\phi\|\leq 1$,
then $\phi$ is completely positive contractive,
abbreviated as c.p.c..
If $\phi$ is c.p.c. and $\phi(1_A)=1_B$,
then $\phi$ is call unital completely positive,
abbreviated as u.c.p..

\end{nota}

The following lemma is a {well known corollary
of Stinespring's theorem (cf. \cite[Lemma 7.11]{KR2002}):

\begin{lem}
\label{c.p.c.-almost-multiplicative-easy}

Let $\phi: A\rightarrow B$ be a {{c.p.c.~map}}  {\blue{from \CA\, $A$ to \CA\, $B.$}}
Then
$\|\phi(xy)-\phi(x)\phi(y)\|\leq \|\phi(xx^*)-\phi(x)\phi(x^*)\|^{1/2}\|y\|$ {\blue{for all}} $x,y\in A.$

\end{lem}

The following lemma is {\blue{taken from}}
\cite[Lemma 3.5]{KW-Covering-dimension}.

\begin{lem}
\label{c.p.c.-almost-multiplicative}

Let $A,B,C$ be $C^*$-algebras, {{let}} $a\in A_{sa}$, {\blue{and}}
{\blue{let}} $\epsilon>0.$  {\blue{Suppose that}}
{{$\psi: A\to B$ and $\phi: B\to C$}}
are c.p.c.~maps
and $\|\phi\circ \psi(a^2)-\phi\circ \psi(a)^2\|\leq \epsilon.$
{\blue{Then,}} for all $b\in B$,
$$
\|\phi(\psi(a)b)-\phi(\psi(a))\phi(b)\|
\leq \epsilon^{1/2}\|b\| {\blue{\andeqn}}
\|\phi(b\psi(a))-\phi(b)\phi(\psi(a))\|
\leq \epsilon^{1/2}\|b\|.
$$

\end{lem}

\begin{proof}
We will only show the first  {\blue{inequality.}}
The proof of the {\blue{second}} is  similar.
We have
$$0\leq \phi(\psi(a)^2)-\phi(\psi(a))^2 
\leq \phi(\psi(a^2))-\phi(\psi(a))^2 \leq\epsilon.
$$
{\blue{Thus}} $\| \phi(\psi(a)^2)-\phi(\psi(a))^2\|\le\ep.$
{\blue{By}} Lemma \ref{c.p.c.-almost-multiplicative-easy} we have
$\|\phi(\psi(a)b)-{{\phi(\psi(a))}}\phi(b)\|
\leq \epsilon^{1/2}\|b\|$.

\end{proof}

Some versions of the following statements {\blue{are}}  well known
(which can also be derived
{{by}}
using Lemma \ref{c.p.c.-almost-multiplicative}
{\blue{in the case of c.p.c.~maps}}).

\begin{lem}
\label{perturbation-c.p.c.-to-p.c.p}
For any \CAs\ $A$ and $B$, any p.c.~map (resp.~c.p.c.~map) $\phi: A\rightarrow B,$
any projection $p\in A,$
any $\dt\in(0,{{1/8}}),$
if $\|\phi(p)-\phi(p)^2\|\leq \delta$,
then there exists  {{a}}  {{p.c.~map}}
(resp. c.p.c.~map) $\psi: A\rightarrow B$
satisfying

(1) $\psi(p)$ is a projection in $C^*(\phi(p))$, {{and}}

(2) $\|(\phi-\psi)|_{pAp}\|<5\dt^{1/2}.$

\end{lem}

\begin{proof}
If $\|\phi(p)-\phi(p)^2\|\leq \delta<{{1/8}}$,
one has  ${\rm sp}(\phi(p))\subset [0,\eta]\cup [1-\eta, 1],$
where $\eta={\frac{2\dt}{1+\sqrt{1-4\dt}}}<
{{\frac{4\dt}{2+\sqrt 2}
}}.$
Then
%
$$
h(t)=
\left\{\begin{array}{ll}
0,  &{\rm for\ }t\in [0,\eta],\\
1/t^{1/2},  &{\rm for\ } t\in [1-\eta,1]
\end{array}\right.
$$
is a continuous function on ${\rm sp}(\phi(p)).$
Let $c:=h(\phi(p)).$
Define a  {\blue{positive linear map (resp. c.p.~map)}}
{{$\psi: A\rightarrow B$ by}}
$x\mapsto c\phi(pxp)c$ for all ${{x}}\in A.$
Then $e:=\psi(p)=h(\phi(p))^2\phi(p)$ is a projection in $C^*(\phi(p)).$
{{It follows from \cite[Corollary 1]{RD66}
%
that $\psi$ is a p.c.~map (resp. c.p.c~map).}}
{\blue{For $x\in (pAp)_{sa}^1$,
by Kadison's generalized Schwarz inequality (\cite[Theorem 1]{Kad52}),}}
$$
\|(1-c)\phi(x)\|^2=
\|(1-c)\phi(x)^2(1-c)\|
\le\|(1-c)\phi(x^2)(1-c)\|
\le \|(1-c)\phi(p)(1-c)\|<\eta.
$$
Then, for $x\in (pAp)_{sa}^1,$ one estimates
$$
\|\phi(x)-\psi(x)\|
=\|\phi(x)-c\phi(x)c\|\le \|(1-c)\phi(x)\|+\|c\|\|\phi(x)(1-c)\|
<\eta^{1/2}(1+\frac{1}{\sqrt{1-\eta}}).
$$

\noindent
{{Therefore,}}
$\|(\phi-\psi)|_{pAp}\|
< 2\eta^{1/2}(1+\frac{1}{\sqrt{1-\eta}}) {\blue{={\frac{4}{\sqrt{{2+\sqrt{2}}}}}(1+\frac{1}{\sqrt{1-\eta}})\dt^{1/2}}}
< 5\dt^{1/2}.$

\end{proof}

\begin{df}\label{Dcuntz}
Let $A$ be a \CA\
{\blue{and}}  let $M_{\infty}(A)_+:=\bigcup_{n\in\mathbb{N}}M_n(A)_+$.
{\green{For $x\in M_n(A),$
we identify $x$ with ${\rm diag}(x,0)\in M_{n+m}(A)$
for all $m\in \N.$}}
{\blue{Let $a\in M_n(A)_+$ and $b\in M_m(A)_+$.
Define $a\oplus b:=\mathrm{diag}(a,b)\in M_{n+m}(A)_+$.
If $a, b\in M_n(A),$
we write $a\lesssim_A b$ if {\blue{there are}}
$x_i\in M_n(A)$
such that
$\lim_{i\rightarrow\infty}\|a-x_i^*bx_i\|=0$.
}}
If such $\{x_i\}$ {\blue{does}}  not  exist,  then we write
$a\lnsim_A b$.
We write $a\sim b$ if $a\lesssim_A b$ and $b\lesssim_A a$ hold.
{\blue{The {\blue{Cuntz}} relation}} $\sim$ is an equivalence relation.
We also write $a\lesssim b$ and $a\sim b,$ when $A$ is given and there is no confusion.
Set $W(A):=M_{\infty}(A)_+/\sim_A$.
Let $\la a\ra$ denote the equivalence class of $a$.
We write $\la a\ra\leq \la b\ra $ if $a\lesssim_A b$.
$(W(A),\leq)$ is a partially ordered abelian semigroup.
$W(A)$ is called almost unperforated,
if for any $\la a \ra, \la b\ra\in W(A)$,
{\blue{and for any}} $k\in\N$,
if $(k+1)\la a\ra \leq k\la b\ra$,
then $\la a \ra \leq \la b\ra$
(see \cite{Rordam-1992-UHF2}).

{\blue{Let $k\in\N$
be an integer.}}
We write ${\blue{k\la a\ra\, {\overset{_\approx}{_<}}\, \la b\ra}}$
if ${\rm Her}(b)$ contains $k$ mutually orthogonal elements
$b_1, {\green{b_2,}}
{{\cdots}},b_k$
such that $a\lesssim b_i,$ $i=1,{\green{2,}}
{{\cdots}},k.$

\end{df}


If $B\subset A$ is a hereditary $C^*$-subalgebra,
$a,b\in B_+$, then $a\lesssim_A b \Leftrightarrow a\lesssim_B b$.



\begin{df}
\label{2-quasitrace}
Denote by $QT(A)$ the set of  2-quasitraces of $A$ with $\|\tau\|=\tau(1_A)=1$ 
(see {{\cite[II 1.1, II 2.3]{BH-trace1982}}}) {\blue{and by}}
$T(A)$ the set of all tracial {\blue{states}} on $A.$
{\blue{We will also use $T(A)$ as well as $QT(A)$ for the extensions on $M_k(A)$ for each $k.$}}
For {\blue{$\tau\in QT(A)$,}} define a lower semi-continuous function
$d_\tau: {\blue{M_k(A)_+}}\rightarrow \C$,
$a\mapsto \lim_{n\rightarrow \infty}\tau(f_{1/n}(a))$.
{\blue{The function}} $d_{\tau}$ is called
{the
dimension function
induced by $\tau$}.
\end{df}

\begin{df}
\label{Dstrict}
Let $A$ be a unital \CA. {\blue{We}} say {\blue{that}} $A$ has
strict comparison (for positive elements),
if, for all $a,b\in M_k(A)_+$, {\blue{$a\lesssim b,$ whenever
$d_{\tau}(a)<d_{\tau}(b)$ holds for all $\tau\in QT(A).$}}

\end{df}



\section{{{Asymptotical tracial approximation}}}


\begin{df}[{{Asymptotical tracial}} approximation]
\label{def-tracial_approximation}
Let $A$ be a unital simple $C^*$-algebra,
let $\mathcal{P}$ be a class of \CAs.
We say $A$ is {\blue{asymptotically}} tracially in $\mathcal{P}$,
if for any finite subset $\calF\subset A$,
{\blue{any}}
$\epsilon>0$, and any $a\in A_+\backslash\{0\}$,
there exist a $C^*$-algebra $B$ in $\mathcal{P}$,
{\blue{c.p.c.~maps}}
$\alpha: A\rightarrow B,$
$\beta_n: B\rightarrow A$, {{and}}
$\gamma_n: A\rightarrow A$
{{($n\in\mathbb{N}$),}} such that

(1) $x\approx_{\epsilon}\gamma_n(x)+\beta_n\circ\alpha(x)$
{\blue{for all}} $x\in\calF$ {\blue{and for all}} $n\in\mathbb{N},$

(2) $\alpha$ is  {\blue{an}} $(\calF, \epsilon)$-{{approximate embedding}},

(3) $\lim_{n\rightarrow\infty}\|\beta_n(xy)-\beta_n(x)\beta_n(y)\|=0$ {\blue{and}}
{$\lim_{n\rightarrow \infty}\|\beta_n(x)\|= \|x\|$ }
{\blue{for all}} ${\blue{x,y\in B,}}$ {{and}}

(4) $\gamma_n(1_A)\lesssim_A a$ {\blue{for all}} $ n\in\mathbb{N}$.

\end{df}

\begin{rem}
\label{Rm32}
{{Let us}}  point out that
in {{the}} definition {{above}},
we {\blue{may}} assume that {{$\calF$ is a finite subset of $A^1_+$,}}
$\epsilon\in (0,1)$,
and $\|a\|=1,$
without loss of generality.

{{Asymptotical tracial}} approximation may also be defined for non-unital \CA s as well as for non-simple \CA s. These will be discussed
in a subsequent paper.

Suppose that ${\cal P}$ has the property that, if $A\in {\cal P},$ then $M_n(A)\in {\cal P}$ for all integer $n\ge 1.$
Then,  it is easy to see that, if $A$ is
asymptotically tracially in ${\cal P},$
then $M_n(A)$ is also
asymptotically tracially in ${\cal P}$
{{(cf. \cite[Theorem 3.7.3]{Lnbook}).}}
Also see {{Theorem}} \ref{tensor-of-tracial-approximation}.
\end{rem}

\begin{rem}
Let ${\cal P}_0$  be the class of finite dimensional \CA s
and let ${\cal P}_1$ be
the class of \CA s of 1-dimensional NCCW complexes
(see \cite{ELT1998} for definition of 1-dimensional NCCW complexes)
respectively.
Since \CA s in ${\cal P}_0$ as well as in ${\cal P}_1$
are semiprojective (see \cite{ELT1998}),
one easily verifies  that
$A$ is
asymptotically tracially in  ${\cal P}_0$
is equivalent to {\blue{that}}
$A$  has tracial rank zero (or $A$ is tracially AF), and
$A$ is
asymptotically tracially in ${\cal P}_1$ is equivalent
to {\blue{that}} $A$ has generalized tracial rank one
{\blue{(see \cite[Definition 9.2, Remark 9.3, Lemma 3.20]{GLN2019}).}}

\end{rem}

\begin{df}\label{DNn}
Denote by ${\cal E}$ the class of exact \CA s and by ${\cal N}$ the class of  nuclear \CA s.
{\blue{For}} {\blue{each}} $n\in\N\cup\{0\}$,
let
${\cal N}_n$ be the class of \CAs\
with nuclear dimension at most $n$
(see \ref{definition-of-nuclear-dimension} {\blue{below}}).
Let $\mathcal{C}_{\mathcal{Z}}$ be the class of $\mathcal{Z}$-stable \CAs,
{\blue{let $\mathcal{C}_{\mathcal{Z},s}$
(and $\mathcal{C}_{\mathcal{Z}, s,s}$) be the class
of separable (and simple) ${\cal Z}$-stable \CA s,}}  {\blue{let}}
${\cal N}_{\cal Z}$ {\blue{be}}
the class of nuclear ${\cal Z}$-stable \CA s,
let ${\cal N}_{{\cal Z}, s}$ (and ${\cal N}_{{\cal Z}, s,s}$) {\blue{be}} the class of separable (and simple)
nuclear ${\cal Z}$-stable \CA s, {\blue{respectively.}}

\end{df}

\begin{exm}

Let $A$ be a unital separable residually finite dimensional \CA,
i.e. there {\blue{exists}} a sequence of finite dimensional representations
{\blue{$\{\pi_i\}$}} of $A$
{\blue{such}} that {\blue{$\{\pi_i\}$}}
{\blue{separates}} the points in {\blue{$A.$}}
{\blue{Let us recall the construction in \cite{D2000} below.}}
For instance, $A$ can be
the full group \CA\ of the free group {\blue{of}}  two generators.
Let $n_i$ be the {{dimension}} of $\pi_i$ ($i\in\N$),
{{let}}
$m_1=1$ and let $m_i=\prod_{j=1}^{i-1}(n_j+1)$ {{for $i\geq 2$}}.
For each $i\in\N$, define {{an}} injective *-homomorphism
$$
h_i: A\otimes M_{m_i}\rightarrow A\otimes M_{m_{i+1}},
\quad
x\mapsto x\oplus (\pi_i\otimes \mathrm{id}_{M_{m_i}})(x).
$$

\noindent
{{Let}} $B:=\lim_{i\rightarrow\infty}(A\otimes M_{m_i},h_i)$,
then $B$ is simple separable unital with tracial rank zero
(
\cite{D2000},
see also \cite[Example {\blue{3.7.7}}]{Lnbook}).
In particular,
$B$ {{is asymptotically tracially in ${\cal N}_0$.}}
{\blue{In fact, D\u{a}d\u{a}rl\u{a}t {{showed}} 
that,}}
{\blue{for any unital infinite dimensional simple AF-algebra $C,$ one can produce a
unital separable simple non-exact \CA\ $B$ with tracial rank zero such that
$K_0(B)=K_0(C)$ as  ordered groups}}
(see \cite[Proposition 9]{D2000}).
In \cite{NW2019},
Niu and Wang showed that,
for {\blue{some
choices}}
of
$A,$
$B$ 
 can be  {\blue{constructed}} to be a  simple separable unital exact {\blue{\CA}}\,
with tracial rank zero  but not $\mathcal{Z}$-stable
{\blue{(so it is asymptotically tracially in ${\cal N}_0$
but not $\mathcal{Z}$-stable).}}
However, we will see later that $B$ is
asymptotically tracially  in
{${\cal C}_{{\cal Z},s}.$}
Actually, every simple separable unital infinite dimensional
\CA\ {\blue{which is asymptotically tracially in ${\cal N}_n$
is}} asymptotically  tracially   {\blue{in ${\cal N}_{{\green{\cal Z}},s,s}$}}
(see Theorem \ref{tracial-finite-dimnuc-implies-tracial-Z-stable}).

\end{exm}

\begin{df}\label{DH}
A class of \CA s ${\cal P}$ is said to have property (H), if,
for any $B\in {\cal P}$ and any projection $e\in B,$ $eBe\in {\cal P}.$
\end{df}

The following lemma is well known.

\begin{lem}
\label{orthogonal_p.c._map}
For any $\ep>0,$
there {\blue{exists}} $\dt>0$ such that,
for any  unital \CAs\ $A,$ {\blue{and}} $B$,
{{any}} \CA\ $C,$ {\blue{and,}}
any p.c.~maps (resp. c.p.c.~maps) $\phi: A\to C$ and $\psi: B\to C,$
if $\|\phi(1_A)-\phi(1_A)^2\|<\dt$ and
$\|(\phi(1_A)+\psi(1_B))-(\phi(1_A)+\psi(1_B))^2\|<\dt,$
then there exist p.c.~maps (resp. c.p.c.~maps) $\bar\phi: A\to C$ and
$\bar\psi: B\to C,$
satisfying

(1) $\bar{\phi}(1_A)$,
$\bar\psi(1_B)$ are projections and  $\bar{\phi}(1_A)\bot\bar\psi(1_B)$, and

(2) $\|\phi-\bar\phi\|<\ep$ and
$\|\psi-\bar\psi\|<\ep.$

\noindent
{{Moreover,}}
if $\phi(1_A)$ is a projection, one can take ${\blue{\bar\phi=\phi.}}$

\end{lem}

\begin{proof}
{\blue{Let $\ep>0.$
Put $\ep_1=\min\{{{\ep/(144\sqrt2)}}, 1/4\}.$
There}} exists
a universal constant {\blue{$\dt\in(0, \ep_1/16)$}} such that if
$\|\phi(1_A)-\phi(1_A)^2\|<\dt$ and
$\|(\phi(1_A)+\psi(1_B))-(\phi(1_A)+\psi(1_B))^2\|<\dt,$
then
\beq
\label{f3-7-lem-1}
\|\psi(1_B)-\psi(1_B)^2\|<{{(\ep_1/5)^2}}
\mbox{ and }
\|\phi(1_A)\psi(1_B)\|<{{(\ep_1/5)^2}}.
\eneq

\noindent
{{By}} {\blue{Lemma}} \ref{perturbation-c.p.c.-to-p.c.p}
and {\blue{\eqref{f3-7-lem-1},}}
there exist p.c.~maps (resp. c.p.c.~maps) $\bar\phi: A\to C$ and
$\hat\psi: B\to C,$
such that

(i) $\bar{\phi}(1_A)$ {\blue{and}}
$\hat\psi(1_B)$ are projections, and

(ii) $\|\phi-\bar\phi\|<{{\ep_1,}}$ and
$\|\psi-\hat\psi\|<{{\ep_1}}.$

\noindent
{{Note}} that if $\phi(1_A)$ is a projection, then one can simply take $\phi=\bar\phi.$
By \eqref{f3-7-lem-1} and (ii),
{\blue{one has}}
$\|\bar\phi(1_A)\hat\psi(1_B)\|<{{(\ep_1/5)^2+2\ep_1<3\ep_1}}.$
Then
\beq
\label{f3-7-lem-2}
\|\hat{\psi}(1_B)-(1-\bar\phi(1_A))\hat{\psi}(1_B)(1-\bar\phi(1_A))\|
<{{6\ep_1}}.
\eneq
{\blue{
Thus $((1-\bar\phi(1_A))\hat{\psi}(1_B)(1-\bar\phi(1_A)))^2
\approx_{{\green{18\ep_1}}} (1-\bar\phi(1_A))\hat{\psi}(1_B)(1-\bar\phi(1_A)).$
Then {\blue{(see  \cite[Lemma 2.5.5]{Lnbook}, for example)}} there is a projection
$q\in C^*({{(1-\bar\phi(1_A))\hat{\psi}(1_B)(1-\bar\phi(1_A))}})$
such that}}
\beq
\|q-\hat{\psi}(1_B)\|<{{36\ep_1}}.
\eneq
Therefore {\blue{(see \cite[Lemma 2.5.1]{Lnbook}, for example)}}
there exists a unitary $u\in \tilde C$
{\blue{(or in $C,$ when $C$ is unital)}}
such that $\|1_{\tilde C}-u\|<{\green{36\sqrt2\ep_1}}\leq\ep/4$
and
$u^*\hat{\psi}(1_B)u=q.$
Define $\bar\psi: {{B}}\to C$
by $\bar\psi(x):=u^*\hat\psi(x)u$ for all $x\in A.$
{\blue{One then verifies that ${\bar \phi}$ and ${\bar \psi}$ meet the requirements.}}


\end{proof}

\begin{prop}
\label{hereditary-subalgebra-preserves-tracially-approximation}

Let $\mathcal{P}$ be a class of $C^*$-algebras.
Let $A$ be a unital simple $C^*$-algebra
which is asymptotically tracially in $\mathcal{P}$.
{{Then the following  conditions}} hold:
For any unital hereditary $C^*$-subalgebra $B\subset A$,
{{any}} finite subset $\calF\subset B$,
{{any}} $ \epsilon>0$ {{and any}}  $b\in B_+\backslash\{0\}$,
there exist a $C^*$-algebra $\bar C$ in $\mathcal{P}$,
{{a}} unital hereditary $C^*$-subalgebra {{$C$ of $\bar C$}},
%
%
{{c.p.c.~maps}}
$\alpha: B\rightarrow C$,
$\beta_n: C\rightarrow B$, {{and}}
$\gamma_n: B\rightarrow B\cap\beta_n(C)^{\bot}$
($n\in\mathbb{N}$), such that

(1)
{{the map}} $\alpha$ is u.c.p.,
$\beta_n(1_C)$, $\gamma_n(1_B)$ are projections, {\blue{and}}
$1_B=\beta_n(1_C)+\gamma_n(1_B)$
{{for all}} $ n\in \mathbb{N}$,

(2) $x\approx_{\epsilon}\gamma_n(x)+\beta_n\circ\alpha(x)$
{{for all}} $ x\in\calF$ {{and}}
{{for all}}
$ n\in\mathbb{N}$,

(3) $\alpha$ is {{an}} $(\calF, \epsilon)$-{{approximate embedding}},

(4) $\lim_{n\rightarrow\infty}\|\beta_n(xy)-\beta_n(x)\beta_n(y)\|=0$ {{and}}
$\lim_{n\rightarrow \infty}\|\beta_n(x)\|= \|x\|$
{{for all}} $ x,y\in C$, {{and}}

(5) $\gamma_n(1_B)\lesssim_B b$ {{for all}} $ n\in\mathbb{N}$.


\noindent
If, in addition, $\mathcal{P}$ has property (H),
then $C$ is in $\mathcal{P}$, whence
every unital hereditary $C^*$-subalgebra of $A$
is also asymptotically tracially in $\mathcal{P}$.

\end{prop}


\begin{proof}
Without loss of generality, we may assume that $1_B\in \calF\subset B^1$
and $\epsilon<1$.
Let $\dt_1<\ep/64$ be the {\blue{universal constant {\blue{(in place of $\dt$)}}}} in {{Lemma}} \ref{orthogonal_p.c._map}
{\blue{associated with  $\ep/64$ (in the place of $\ep$).}}
{\blue{Let $\dt:=\frac{1}{128}\min\{\ep,(\dt_1/{\blue{5}})^2, {\blue{1}}\}.$}}

Since $A$ is asymptotically tracially in $\mathcal{P}$,
there {{exist}} a $C^*$-algebra $\bar{C}$ in $\mathcal{P}$
and
c.p.c maps
$\bar{\alpha}: A\rightarrow \bar{C}$,
$\bar{\beta}_n: \bar{C}\rightarrow A$, {{and}}
$\bar{\gamma}_n: A\rightarrow A$
($n\in\mathbb{N}$) {\blue{such}} that

({{$1'$}}) $x\approx_{\delta}\bar{\gamma}_n(x)+\bar{\beta}_n\circ\bar{\alpha}(x)$
{{for all}} $ x\in\calF$ {{and}} {{for all}} $ n\in\mathbb{N}$,

({{$2'$}}) $\bar{\alpha}$ is {{an}} $(\calF, \delta)$-{{approximate embedding}},

({{$3'$}}) $\lim_{n\rightarrow\infty}
\|\bar{\beta}_n(xy)-\bar{\beta}_n(x)\bar{\beta}_n(y)\|=0$,
$\lim_{n\rightarrow \infty}\|\bar{\beta}_n(x)\|= \|x\|$
{{for all}} $ x,y\in \bar{C}$, {\blue{and}}

({{$4'$}}) $\bar{\gamma}_n(1_A)\lesssim_A b$, {{for all}} $ n\in\mathbb{N}$.


\noindent
{{Since}}
$\|{\bar\af}(1_B)-{\bar\af}(1_B)^2\|<{\blue{\dt}}$
{{(see ({{$2'$}})),}}
by Lemma \ref{perturbation-c.p.c.-to-p.c.p},
there exists a c.p.c.~map ${\alpha}: A\rightarrow \bar{C}$
{\blue{such}} that ${\alpha}(1_B)$ is a projection and
\beq\label{34-n3}
\|\bar{\alpha}(x)-{\alpha}(x)\|< \frac{\dt_1}{8}\|x\|
\rforal x\in 1_BA1_B=B.
\eneq
Let $C:={\alpha}(1_B)\bar{C}{\alpha}(1_B)$
be a unital hereditary $C^*$-subalgebra of $\bar C.$
We may view $\af$ as a map from  $B$ to $C.$
{\blue{Then,}} by ({{$2'$}}), {\blue{{{\eqref{34-n3},}} and by the choice of $\dt,$}}
 $\alpha$ is {{an}} $(\calF, \ep)$-{{approximate embedding}}.
Thus (3) in the proposition holds.

By {\blue{($3'$),}} we have
$\lim_{n\rightarrow\infty}
\|\bar{\beta}_n(1_C)-\bar{\beta}_n(1_C)^2\|=0.$
{\blue{Then,}} by Lemma \ref{perturbation-c.p.c.-to-p.c.p},
there {{exist}} c.p.c.~maps $\hat{\beta}_n: C\rightarrow A$
such that $\hat{\beta}_n(1_C)$ are projections and
\beq\label{34-n4}
\|{{\bar{\beta}_{n}|_{C}}}-\hat{\beta}_n\|\rightarrow 0\,\,\,{{{\rm{(as}}\,\,n\to
\infty {\rm{).}}}}
\eneq

\noindent
{{By}} \eqref{34-n3} and \eqref{34-n4}, without loss of generality,
we may assume {\blue{that,}} for all $n\in\N,$
\beq
\label{f3-7prop-3}
\bar{\beta}_n\circ\bar{\alpha}(x)\approx_{\dt_1/8}
\hat{\beta}_n\circ\af(x)
\mbox{ for all }
x\in\calF.
\eneq

\noindent
{{Then,}} from \eqref{f3-7prop-3} and ($1'$),
$$
\|(\bar\gamma_n(1_B)+\hat{\beta}_n\circ\af(1_B))-
(\bar\gamma_n(1_B)+\hat{\beta}_n\circ\af(1_B))^2\|<\dt_1.
$$

\noindent
By {{Lemma}} \ref{orthogonal_p.c._map} and the choice of $\dt_1$,
for each $n\in\N$,
there exists a c.p.c.~map $\hat\gamma_n: A\to A$ such that

(i) $\hat\gamma_n(1_B)$ is a projection and
$\hat\gamma_n(1_B)\bot\hat{\beta}_n\circ\af(1_A)$, and

(ii) $\|\hat\gamma_n-\bar\gamma_n\|<\ep/64.$

\noindent
{{By}} (ii), \eqref{f3-7prop-3} and $(1')$, we have
$1_B\approx_{\ep/32} \hat\gamma_n(1_B)+ \hat{\beta}_n\circ\af(1_B).$
Then there exist unitaries $u_n\in A$ ($n\in\N$)
such that
$\|1_A-u_n\|<\sqrt{2}\ep/32$ and
$u_n^*(\hat\gamma_n(1_B)+ \hat{\beta}_n\circ\af(1_B))u_n=1_B$
(see \cite[Lemma 2.5.1]{Lnbook}).
Define c.p.c.~maps $\gamma_n: B\to B$ by $\gamma_n(x):=u_n^*\hat\gamma_n(x)u_n$,
and define c.p.c.~maps $\beta_n: C\to B$ by $\beta_n(x):=u_n^*\hat\beta_n(x)u_n$.
Then (1) in the proposition holds.
By $(3'),$ \eqref{34-n4} and the fact that $u_n$ are unitaries,
{\blue{condition}} (4) in the proposition holds.

By (ii) and the fact that $\|1_A-u_n\|<\sqrt{2}\ep/32$, we have
\beq
\label{f3-7prop-6}
\|\gamma_n-\bar\gamma_n\|<\ep/4
\mbox{ and }
\|\beta_n-\hat\beta_n\|<\ep/4.
\eneq
Then, by \eqref{f3-7prop-6}, \eqref{f3-7prop-3} and ($1'$),
{\blue{condition}} {\blue{(2)}} in the proposition holds.

By
{{the fact that $\gamma_n(1_B)$ is a projection,}}
\eqref{f3-7prop-6}, and \cite[Proposition 2.2]{Rordam-1992-UHF2},
we have $\gamma_n(1_B)\sim f_{1/4}(\gamma_n(1_B))
\lesssim \bar\gamma_n(1_B)\overset{(4')}{\lesssim} b.$
Thus (5) in the proposition holds.  {\blue{The proposition follows.}}

\end{proof}



The following lemma is well known.
\begin{lem}
\label{projection-Cuntz-subequivalent-lem}
Let $A$ be a \CA, {\blue{$a\in A_+$}} and let
$p\in A$ {\blue{be}} a projection with
$p\lesssim_A a$.
Then there exists $s\in A$ such that $p=s^*as$.
Moreover, if $\{a_n\}\in l^\infty(A)_+$
and $\{p_n\}\in l^\infty(A)$
is a projection such that
$\pi_\infty(\{p_n\})\lesssim_{l^\infty(A)/c_0(A)} \pi_\infty(\{a_n\}),$
then there exists  $\{s_n\}\in l^\infty(A)$ such that
{\blue{$s_n^*s_n=p_n$ and $s_ns_n^*\in \Her(a_n)$ for all}}
large $n.$

\end{lem}

\begin{proof}

Since $p\lesssim_A a$, there exists $r\in A$ such that $\|p-r^*ar\|<1/2$.
Then $\|p-pr^*arp\|<1/2.$
Therefore $pr^*arp$ is an invertible positive element in $pAp$.
{\blue{Hence,}} {by functional calculus} there exists $b\in (pAp)_+$ {{with $\|b\|\leq \sqrt2$}}
such that $p=bpr^*arpb$. {\blue{Choose}} $s=rpb.$  {\blue{Then $p=s^*as.$}}

{\blue{For ``Moreover" part, by what has been proved, there is $t\in l^\infty(A)/c_0(A)$
such that
$\pi_\infty(\{p_n\})=t^*\iota_A(a)t.$ Then there exists
$\{t_n\}\in l^\infty(A)$ such that $\|p_n-(t_n)^*at_n\|<1/2$ for all large $n.$
Thus, by what has been proved, there is $r_n\in A$ ($\|r_n\|\le \sqrt{2}\|t_n\|$)
such that $p_n=r_n^*ar_n$ for all large $n.$
Let $s_n=a_n^{1/2}r_n.$
Then $s_n^*s_n=p_n$ and
$q_n:= s_ns_n^*= a_n^{1/2}r_nr_n^*a_n^{1/2}\in \Her(a_n).$}}


\end{proof}

The following proposition
provides another picture of
Definition \ref{def-tracial_approximation}.

\begin{prop}
\label{equivalent-definition-of-tracial-approximation}
Let $A$ be a {{simple}} unital $C^*$-algebra {\blue{and}}
$\mathcal{P}$ be a class of
{\blue{separable}}
\CAs\,\,  with Property (H).
%
{\blue{If $A$ is
asymptotically tracially in ${\cal P},$ then  the following  holds:}}

For any finite subset $\calF\subset A$,  any $\epsilon>0$, and any
$a \in A_+\backslash\{0\}$,
there exists a unital
$C^*$-subalgebra
$B\subset  l^\infty(A)/c_0(A)$ which is strictly embedded
such that $B$ in $\mathcal{P}$, and
(recall notations defined in Notation \ref{notation-2})

(1) $1_B\iota_A(x)\approx_{\epsilon}\iota_A(x)1_B$\, {\blue{for all}} $ x\in\calF$,

(2) $1_B\iota_A(x)1_B\in_{\epsilon} B$ {\blue{and}}  $\|1_B\iota_A(x)1_B\|\geq \|x\|-\epsilon$\,
{\blue{for all}} $ x\in \calF$, {\blue{and}}

(3) $\iota_A(1_A)-1_B\lesssim_{
l^\infty(A)/c_0(A)}\iota_A(a)$.

\noindent
{{If}} ${\cal P}$ is a class of separable nuclear \CA s, then converse also holds.
\end{prop}

\begin{proof}
Assume that $A$ is asymptotically tracially in $\mathcal{P}$.
Let $\calF\subset A$ be a finite subset
{{with $1_A\in\calF$}},
{{let}} $\epsilon\in(0,1)$, let $a\in A_+$ with $\|a\|=1$, and
let $\delta:=\epsilon^2/4$.
By Proposition \ref{hereditary-subalgebra-preserves-tracially-approximation},
there {{exist}} a {{unital}} $C^*$-algebra $B$ in $\mathcal{P}$,
{{c.p.c.~maps}}
$\alpha: A\rightarrow B$,
$\beta_n: B\rightarrow A${\green{, and}}
$\gamma_n: A\rightarrow A\cap\beta_n(B)^{\bot}$
($n\in\mathbb{N}$) {\blue{such}} that

({{$1'$}}) 
{{$\alpha$}} is u.c.p.,
$\beta_n(1_B)$ {\blue{and}} $\gamma_n(1_A)$ are projections, {\blue{and}}
$1_A=\beta_n(1_B)+\gamma_n(1_A)$
{{for all $ n\in \N$,}}

({{$2'$}}) $x\approx_{\delta}\gamma_n(x)+\beta_n\circ\alpha(x)$
{\blue{for all}} $ x\in\calF$ {\blue{and}} {\blue{for all}} $ n\in\mathbb{N},$

({{$3'$}}) $\alpha$ is {\blue{an}} $(\calF, \delta)$-{{approximate embedding}},

{\blue{({{$4'$}})}} $\lim_{n\rightarrow\infty}\|\beta_n(xy)-\beta_n(x)\beta_n(y)\|=0$ {\blue{and}}
$\lim_{n\rightarrow \infty}\|\beta_n(x)\|= \|x\|$
{\blue{for all}} $ x,y\in B,$ {\blue{and}}

({{$5'$}}) $\gamma_n(1_A)\lesssim_A f_{1/2}(a)$ {\blue{for all}} $ n\in\mathbb{N}$.

\noindent
{{Note}} that {\blue{({{$4'$}})}} induces {{a strict}} embedding
$\beta: B\rightarrow l^{\infty}(A)/c_0(A)$,
$x\mapsto \pi_{\infty}(\{\bt_n(x)\}),$ {\blue{and that}}
({{$2'$}}) shows that, for any $x\in \calF$,
\begin{eqnarray*}
&&\hspace{-0.5in}
{\blue{\|\beta(1_B)\iota_A(x)-\iota_A(x)\beta(1_B)\|
=}}\limsup_{n\rightarrow\infty}\|\beta_n\circ\alpha(1_A)x-x\beta_n\circ\alpha(1_A)\|\\
&\leq&
2\delta+
\limsup_{n\rightarrow\infty}\|\beta_n\circ\alpha(1_A)(\gamma_n(x)+\beta_n\circ\alpha(x))
-(\gamma_n(x)+\beta_n\circ\alpha(x))\beta_n\circ\alpha(1_A)\|\\
&=&
2\delta+
\limsup_{n\rightarrow\infty}\|\beta_n\circ\alpha(1_A)\beta_n(\alpha(x))
-\beta_n(\alpha(x))\beta_n\circ\alpha(1_A)\|\\
&=&
2\delta+
\limsup_{n\rightarrow\infty}
\|\beta_n(\alpha(1_A)\alpha(x)-\alpha(x)\alpha(1_A))\|\le  4\delta <\ep.
\end{eqnarray*}
Thus (1) {\blue{of}} the proposition holds.
For any $x\in\calF$,
\begin{eqnarray*}
\|\beta(1_B)\iota_A(x)\beta(1_B)-\beta\circ\alpha(x)\|
&=&
\limsup_{n\rightarrow\infty}
\|\beta_n(1_B)x\beta_n(1_B)-\beta_n\circ\alpha(x)\|\\
&\leq&
\delta+
\limsup_{n\rightarrow\infty}
\|\beta_n(1_B)(\gamma_n(x)+\beta_n\circ\alpha(x))\beta_n(1_B)-\beta_n\circ\alpha(x)\|\\
&=&
\delta+
\limsup_{n\rightarrow\infty}
\|\beta_n(1_B)\beta_n\circ\alpha(x)\beta_n(1_B)-\beta_n\circ\alpha(x)\|
{\blue{=\delta<\ep.}}
\end{eqnarray*}
Thus $\beta(1_B)\iota_A(x)\beta(1_B)\in_{\epsilon}\beta(B)$.
By the estimation above, {\blue{($4'$), and {\blue{by}} ($3'$),}} we also have
$$
\|\beta(1_B)\iota_A(x)\beta(1_B)\|
\geq
\|\beta\circ\alpha(x)\|-\delta
{{\ =\ }}
\|\alpha(x)\|-\delta
\geq \|x\|-2\delta
\geq \|x\|-\epsilon.
$$
Thus (2) {\blue{of}} the proposition holds.
%


{{By}} ($1'$), {\blue{($5'$),}} and Lemma \ref{projection-Cuntz-subequivalent-lem},
there exist partial isometries $s_n\in {{A}}$ such that
$1_A-\beta_n(1_B)=s_n^*s_n$ and $s_ns_n^*\in \Her_A(f_{1/2}(a))$
{{for all large $n$.}}
Let $s=\{s_n\}\in l^\infty(A).$  Then
$\pi_\infty(s)^*\pi_\infty(s)=\iota_A(1_A)-\bt(1_B)$
and, since $f_{1/4}(a)f_{1/2}(a)=f_{1/2}(a),$
\beq
\pi_\infty(s)\pi_\infty(s)^*
&=&\pi_\infty(\{s_ns_n^*\})
=\pi_\infty(\{f_{1/4}(a)s_ns_n^*f_{1/4}(a)\})\\
&=&f_{1/4}(\iota_A(a))\pi_\infty(ss^*)f_{1/4}(\iota_A (a))
\in\Her_{l^\infty(A)/c_0(A)}(\iota_A(a)),
\eneq
which implies that (3) {{of the proposition}} holds.
This proves the first part of the proposition.

 For the second part, let us assume that ${\cal P}$ is a class of separable nuclear \CA s and consider the converse.
{\blue{Let}} $\calF\subset A^1$ be a finite subset,
let $\epsilon>0$, {\blue{and}} let $a\in A_+$ with $\|a\|=1$.
Let $\delta:=\frac{\epsilon}{100}$ {\blue{and let}}
$\bar{\calF}:=\calF\cup(\calF\cdot \calF).$
{\blue{Suppose that (1), (2) and (3) {\blue{hold}} for ${\bar{\cal F}},$
$\dt$, {{$a,$ and some unital separable nuclear \CA\ $B\in{\cal P}.$}}}}
By {\blue{(2)}} {{and}}
{\blue{the fact that}} $B$ is nuclear, {\blue{and}}
by a consequence of Arveson's extension theorem
(see \cite[Theorem 2.3.13]{Lnbook}
), there exists a c.p.c.~map
$\alpha': l^{\infty}(A)/c_0(A)\rightarrow B$ such {{that}}
\beq
\label{f3-10-1}
1_B\iota_A(x)1_B\approx_{2\delta}\alpha'(1_B\iota_A(x)1_B)
\mbox{ for all }x\in \bar{\calF}.
\eneq
Define a c.p.c.~map $\af: A\to B$ {\blue{by}} $x\mapsto \alpha'(1_B\iota_A(x)1_B)$.
For $x,y\in\calF$,
{\green{by \eqref{f3-10-1} and (2),}}
{\blue{we have}}
$\|\alpha(x)\|\geq
\|1_B{{\iota_A(x)}}1_B\|-2\delta
\geq \|x\|-3\delta$,
and
$$
\alpha(x)\alpha(y)
\overset{{\green{\eqref{f3-10-1}}}}{\approx_{4\delta}}
1_B{{\iota_A(x)}}1_B{{\iota_A(y)}}1_B
\overset{{\green{(1)}}}{\approx_{\delta}}
1_B{{\iota_A(xy)}}1_B
\overset{{\green{\eqref{f3-10-1}}}}{\approx_{2\delta}}
\alpha(xy).
$$
Thus (2) in Definition \ref{def-tracial_approximation} holds.
Since $B$ is nuclear and separable, by the Choi-Effors Lifting Theorem
{{(see \cite[Theorem 3.10]{ChoiE76})}},
there exists a c.p.c.~map $\beta: B\rightarrow l^{\infty}(A)$
such that $\pi_\infty\circ \beta=\mathrm{id}_B.$
Let $\beta_n: B\rightarrow A$ be the $n$-th component
of $\beta$.
{{Applying Lemma \ref{perturbation-c.p.c.-to-p.c.p},
we may also assume that $\beta_n(1_B)$ is a projection
{\blue{for all large $n.$}}}}
Since $\beta$ is a strict embedding,
$\{\beta_n\}$ satisfies
(3) in Definition \ref{def-tracial_approximation}.
%

Define a c.p.c.~map $\gamma_n: A\rightarrow A$ {\blue{by}}
$x\mapsto (1_A-\beta_n(1_B))x(1_A-\beta_n(1_B))$.
{{Note that $\gamma_n(1_A)$ is a projection for all large $n,$
and
$\pi_\infty(\{\gamma_n(1_A)\})=\iota(1_A)-1_B.$}}
By {{(3)}}
and Lemma \ref{projection-Cuntz-subequivalent-lem},
{\blue{we may also assume,}}
for {{all}} large $n,$
 $\gamma_n(1_A)\lesssim_A a$.
Hence (4) in Definition \ref{def-tracial_approximation} holds for
{{all}} large $n.$

By {{(1)}},
for all $x\in\calF$,
\begin{eqnarray*}
\iota_A(x)&\approx_{2\delta}&
(\iota_A(1_A)-1_B)\iota_A(x)(\iota_A(1_A)-1_B)+(1_B\iota_A(x)1_B)\\
&\approx_{2\delta}&
(\iota_A(1_A)-1_B)\iota_A(x)(\iota_A(1_A)-1_B)+\alpha(x)\\
&=&
\pi_\infty(\{1_A-\beta_n(1_B)\})\iota_A(x)\pi_\infty(\{1_A-\beta_n(1_B)\})
+\alpha(x)\\
&=&
\pi_\infty(\{(1_A-\beta_n(1_B))x(1_A-\beta_n(1_B))\}+\beta\circ \alpha(x))\\
&=&
\pi_\infty(\{\gamma_n(x)+\beta_n\circ \alpha(x)\}).
\end{eqnarray*}
Therefore
$x\approx_{4\delta}\gamma_n(x)+\beta_n\alpha(x)$
for all large $n.$
{\blue{Hence}}
(1) in Definition \ref{def-tracial_approximation} holds
for all large $n.$
It follows that  $A$ is asymptotically tracially in $\mathcal{P}$.


\end{proof}

{\blue{The proof of the following is standard and can be found in
\cite[Theorem 5.3]{Lin01-2} {\blue{(see also}} \cite[Lemma 3.6.5]{Lnbook}).}}

\begin{prop}\label{310}
{\blue{Let ${\cal P}$ be a class of unital \CA s which satisfy property (H).
Let $A$ be a unital simple \CA\, which satisfies the first part of the conclusion of {{Proposition}} \ref{equivalent-definition-of-tracial-approximation}
(associated with ${\cal P}$).
Then any unital hereditary \SCA\, $C$  of $A$ also satisfies the first part of the conclusion of {{Proposition}} \ref{equivalent-definition-of-tracial-approximation}  associated with ${\cal P},$ i.e.,
for any finite subset $\calF\subset C$,  any $\epsilon>0$, and any
$a \in C_+\backslash\{0\}$,
there exists a unital
$C^*$-subalgebra
$B\subset  l^\infty(C)/c_0(C)$ which is strictly embedded
such that $B$ in $\mathcal{P}$, and}}

{\blue{
(1) $1_B\iota_A(x)\approx_{\epsilon}\iota_A(x)1_B$\, {\blue{for all}} $ x\in\calF$,}}

(2) $1_B\iota_A(x)1_B\in_{\epsilon} B$ {\blue{and}}  $\|1_B\iota_A(x)1_B\|\geq \|x\|-\epsilon$\,
{\blue{for all}} $ x\in \calF$, {\blue{and}}

(3) $\iota_A(1_A)-1_B\lesssim_{
l^\infty(A)/c_0(A)}\iota_A(a)$.

\end{prop}

\section{Properties passing by {{asymptotical}}  tracial approximations}

In this section, it will be shown that, for certain classes of \CA s ${\cal P},$
if a unital simple \CA\, $A$ is
asymptotically tracially in ${\cal P},$ then $A$ is actually in ${\cal P}.$

\begin{df}
{\blue{Recall that a unital \CA}}
$A$ is {\blue{finite,}} if for any {\blue{nonzero}} projection $p\in A$,
{\blue{$1_A\lesssim_A p$}} implies $p= 1_A$.
$A$ is called stably finite, if $A\otimes M_n$ is finite for all $n\in\N$.
\end{df}


\begin{prop}\label{Tstablerank}
{\blue{Let $A$ be a unital separable simple \CA.}}

(a)\, Let  ${\cal P}_{f}$ be  the class of  unial finite \CA s.
If $A$ {{is}} asymptotically tracially in
${\cal P}_{f},$ then $A\in {\cal P}_{f}.$

{\blue{(b)\, Let  ${\cal P}_{sf}$ be  the class of  unial  stably finite \CA s.
If $A$ {{is}} asymptotically tracially in
${\cal P}_{sf},$ then $A\in {\cal P}_{sf}.$}}

{\blue{(c)\, Let ${\cal Q}$ be the class of separable quasidiagonal \CA s. If $A$ is asymptotically tracially in ${\cal Q},$
then $A\in {\cal Q}.$}}

\end{prop}

\begin{proof}
For (a), {\blue{assuming otherwise and}}  that
{\blue{there is a projection $p\in A$ and  there is}} $v\in A$  such that
$v^*v=1_A$ and $vv^*:=p\not=1_A.$
Since $A$ is asymptotically tracially in ${\cal P}_{f},$
and ${\cal P}_{f}$ has property (H),
then by {{Proposition}} \ref{hereditary-subalgebra-preserves-tracially-approximation},
for any $\ep>0,$ with ${\cal F}=\{1_A, p, v, v^*, 1-p\},$
there is {{a}} {{u.c.p.}}~map $\af: A\to B$ for some unital
{{finite}} 
\CA\, $B$
which is  an $({\cal F}, \ep)$-approximate embedding.

With sufficiently small $\ep,$ we may assume that {{there}} is a projection
$e\in B$ such that
\beq\label{sf-3}
&&\|\af(1_A)-\af(p)\|\ge 1-1/64,
\\
\label{sf-2}
&&\af(v)^*\af(v)\approx_{1/64} \alpha(1_A)=1_B
\andeqn
\af(v)\af(v)^*\approx_{1/64} \alpha(p)\approx_{1/64} e.
\eneq
It follows from \eqref{sf-2} that   $1_B$ and $e$ are equivalent in $B,$
and {\blue{from}} \eqref{sf-3} that $\|1_B-e\|\geq 1/2,$
which {\blue{contradicts the assumption that}}
$B$ is {{finite.}}
{\blue{In}} other words,
$A$ is in ${\cal P}_{f}.$

For (b), 
{{note}} that
{\blue{$B$}} in ${\cal P}_{sf}$ implies $M_n(B)$ in ${\cal P}_{sf}$
for all $n\in \N.$
{\blue{Therefore}} (b) follows from (a) and the last part of \ref{Rm32}.

For (c),  let
{{${\cal F}\subset A^1$}}
be a finite subset and let $\ep>0.$
{{By}}  {{Proposition}} \ref{hereditary-subalgebra-preserves-tracially-approximation},
there is a unital quasidiagonal \CA\, $B$ and a c.p.c.~map $\af: A\to B$
such
that
 \beq
\|\af(a)\|\ge (1-\ep/4)\|a\|\andeqn \|\af(ab)-\af(a)\af(b)\|<\ep/4
 \rforal a, b\in {\cal F}.
 \eneq

\noindent
{{Since}} $B$ is quasidiagonal, by  {\blue{\cite[Theorem 1]{Voi91},}}
there is a c.p.c.~map $\bt: B\to F$ (for some finite dimensional \CA\, $F$) such that
$\|\bt(y)\|\ge \|y\|-\ep/16$ and $\|\bt(xy)-\bt(x)\bt(y)\|<\ep/16$ for all $x, y\in \af({\cal F}).$
Let $\phi=\bt\circ \af.$  Then $\phi$ is a c.p.c.~map from $A$ to
{\blue{$F.$}}
For all $a\in {\cal F},$
\beq
\|\phi(a)\|=\|\bt\circ \af(a)\|\ge \|\af(a)\|-\ep/16\ge (1-\ep/4)\|a\|-\ep/16\ge \|a\|-\ep.
\eneq
Moreover, for all $a, b\in {\cal F}.$
\beq
\phi(ab)=\bt(\af(ab))\approx_{\ep/4} \bt(\af(a)\af(b))\approx_{\ep/16} \bt(\af(a))\bt(\af(b))=\phi(a)\phi(b).
\eneq
It follows from \cite[Theorem 1]{Voi91}
that $A$ is quasidiagonal.


\end{proof}

The following is taken from the proof of   \cite[Lemma 2.4]{Lin1991}.

\begin{lem}[{cf. \cite[Lemma 2.4]{Lin1991}}]
\label{L1991}
Let $A$ be a separable non-elementary simple \CA.
Then there exists a sequence {\blue{$\{d_n\}$ in}} $A_+$ such that $\|d_n\|=1,$
${{(n+1)}}\la d_{n+1}\ra
\,{\overset{_\approx}{_<}}\, \la d_n\ra$ (recall the
{{D}}efinition \ref{Dcuntz}) ($n\in \N$), and,
for any $x\in A_+\setminus\{0\},$ there exists {{$N\in\N$}} such that
$\la {{d_N}}\ra \le \la x\ra.$ 
\end{lem}

\begin{proof}
The proof is contained in the proof of  \cite[Lemma 2.4]{Lin1991}.
Let 
$\{x_n\}$
be a dense sequence of the
{\blue{unit sphere}}
of $A,$
let $z_n=(x_n^*x_n)^{1/2}$ {\blue{and}}
$y_n=f_{1/2}(z_n),$
{{$n\in\N.$}}
The proof of \cite[Lemma 2.4]{Lin1991} shows that,
for any $x\in A_+\setminus \{0\},$
(we may assume that $\|x\|=1$)
there exists $N$ such that $y_N\lesssim x.$
Indeed, as exactly in the proof of \cite[Lemma 2.4]{Lin1991},
there is an integer $N$ such that $\|x-{\blue{z_N}}\|$ is sufficiently small, {\blue{and,}}
with $1/{\blue{8}}>\ep>0,$
$$
\|f_\ep(x)-f_\ep({\blue{z_N}})\|<\ep/4.
$$
By  {\blue{\cite[Proposition 2.2]{Rordam-1992-UHF2},}}
\beq\label{L1991-1}
y_N\lesssim f_{1/4}(z_N)\lesssim f_{\ep/2}(f_{\ep}(z_N))\lesssim f_{\ep}(x)\lesssim x.
\eneq

\noindent
{{Now}} let $d_1=y_1/\|y_1\|.$
There are $2$ mutually orthogonal
nonzero elements $z_{1,1},z_{1,2}\in {\rm Her}(d_1)_+$
(as in the proof of \cite[Lemma 2.4]{Lin1991}).
By  \cite[Lemma 2.3]{Lin1991}, for example,
there is $d_{2}\in {\rm Her}(d_1)_+$ such that $\|d_{2}\|=1$ and
$d_{2}\lesssim y_{2},\,z_{1,1},z_{1,2}.$
{{It}} follows that $2\la d_{2}\ra \,\overset{_\approx}{_<}\, \la d_1\ra.$

Suppose $d_1,d_2,
{{\cdots}},d_n$
have been chosen so that  $\|d_j\|=1,$
$d_j\lesssim y_j$
{\blue{($j=1,2,
{{\cdots}},n$),}} and
{\blue{$(j+1)\la d_{j+1}\ra\, {{\overset{_\approx}{_<}}} \, \la d_{j} \ra$}}
{\blue{($j=1, 2, {{\cdots}},n-1$).}}
There are $n+1$ mutually orthogonal
nonzero elements $z_{n,1}, z_{n,2},
{{\cdots}},z_{n,n+1}\in {\blue{{\rm Her}(d_n)_+}}$ (as in the proof of  \cite[Lemma 2.4]{Lin1991}).
By  \cite[Lemma 2.3]{Lin1991}, for example,
there is $d_{n+1}\in {\blue{{\rm Her}(d_n)_+}}$ such that $\|d_{n+1}\|=1$ and
$d_{n+1}\lesssim y_{n+1},\,z_{n,i},$  $ i=1,2,
{{\cdots}},n+1.$  It follows that $(n+1)\la d_{n+1}\ra \,\overset{_\approx}{_<}\, \la d_n\ra.$

By the induction, 
{{we}} obtain a sequence $\{d_n\}$ such that $\|d_n\|=1,$
$d_{{n}}\lesssim y_n,$
and $n\la d_{{n+1}}\ra \,\overset{_\approx}{_<}\, \la d_n\ra,$ {{$n\in\N.$}}
By \eqref{L1991-1}, for any $x\in A_+\setminus \{0\},$ there is $N$ such that $d_N\lesssim y_N\lesssim x.$

\end{proof}

\begin{prop}\label{Ptrpi=pis}
Let ${\cal P}$ be the class of separable purely infinite simple \CA.
Suppose that $A$ is a unital simple \CA\, which is
asymptotically tracially in ${\cal P}.$
Then $A$ is {{a}} purely infinite simple \CA.

\end{prop}

\begin{proof}
We may assume that $A$ is not elementary.
Let $a\in A_+\setminus\{0\}.$ It suffices to show that $1_A\lesssim a$
(\cite{Cuntz1981}, see also \cite{LZ1991}).
We may assume that $\|a\|=1.$
By applying {{Lemma}} \ref{L1991} to $\Her(a),$  we obtain  two nonzero mutually orthogonal  elements $a_0$ {\blue{and}}  $a_1$ with $\|a_0\|=1$ and $\|a_1\|=1$ such
that $a_0+a_1\lesssim a.$
Let $b=f_{1/2}(a_1)$ and let
{\blue{$\ep:=1/2^{10}.$}}
Since $A$ is asymptotically  tracially in ${\cal P},$ by
{{Proposition}}
{\blue{\ref{equivalent-definition-of-tracial-approximation}}}, there exists a unital
$C^*$-subalgebra
$B\subset  l^\infty(A)/c_0(A)$ which is strictly embedded
such that $B$ in {\blue{$\mathcal{P}$, and}}

(1) $1_B\iota_A(b)\approx_{\epsilon}\iota_A(b)1_B,$

(2) $1_B\iota_A(b)1_B\in_{\epsilon} B$, $\|1_B\iota_A(b)1_B\|\geq \|b\|-\epsilon$, {\blue{and}}

(3) $\iota_A(1_A)-1_B\lesssim_{
l^\infty(A)/c_0(A)}\iota_A(f_{1/2}(a_0))$.

\noindent
{{By}} (2), there exists an element $b_1\in B_+$ such that
\beq
\label{f4-4prop-1}
\|1_B{\blue{\iota_A(b)}}1_B-b_1\|<\ep=1/2^{10}.
\eneq
Since $B$ is purely infinite, {\blue{by \cite[Proposition 4.1.1]{Rordam 2002},}} there is $x\in B$ such that $x^*f_{1/2}(b_1)x=1_B.$
{There exists a   sequence of
projections $p_n \in A$} such that
$\pi_{\infty}(\{p_n\})=1_B,$ where $\pi_{\infty}: l^\infty(A)\to l^\infty(A)/c_0(A)$ is the quotient map.
Then we obtain $\{x_n\}, \{b_{1,n}\}\in l^\infty(A)$  (with $\pi_{\infty}(\{x_n\})=x$ and $b_1=\pi_{\infty}(\{b_{1,n}\})$) such that
\beq
\label{f4-4prop-2}
{\blue{\lim_{n\to\infty}\|x_n^*f_{1/2}(b_{1,n})x_n-p_n\|=0 \andeqn
\limsup\|p_nbp_n-b_{1,n}\|\le \ep.}}
\eneq
Then \eqref{f4-4prop-2}
{\blue{(see  \cite[Proposition 2.2]{Rordam-1992-UHF2} again)}}
implies {\blue{that,}}
for large $n$,
{{\beq
\label{f4-4prop-4}
p_n \lesssim f_{1/2}(b_{1,n})
\mbox{ and }
f_{1/2}(b_{1,n})\lesssim p_nbp_n.
\eneq}}
{{On the other hand,}}  by (3) and
Lemma \ref{projection-Cuntz-subequivalent-lem},
$1-p_n\lesssim f_{1/2}(a_0)$ for all large $n.$
It follows that, for %
{\blue{all sufficiently
large}} $n,$
\beq
1_A=(1-p_n)+p_n\lesssim f_{1/2}(a_0)+b=f_{1/2}(a_0)+f_{1/2}(a_1)\lesssim a.
\eneq
\end{proof}

\begin{rem}\label{Rm45}
Let $A$ be a unital separable simple \CA\, and let ${\cal P}$ be the class of unital purely infinite simple \CA s.
Suppose that $A$ satisfies the conclusion of the first part of  {{Proposition}}
\ref{equivalent-definition-of-tracial-approximation} with ${\cal P}$ above. Then the proof
of  {{Proposition}} \ref{Ptrpi=pis}  shows that $A$ is purely infinite.

\end{rem}

\begin{thm}
\label{quasitrace-on-tracially-exact-CA-are-trace}
Let ${\cal T}$ be  the
class of unital \CA s $B$ such that every $2$-quasitrace of $B$ is a trace.
Suppose that $A$ is a unital separable \CA\, satisfying the following conditions:
For any $\ep>0,$ any {\blue{$\eta>0,$}} and any finite subset
{{${\cal F}\subset A,$}}
there exist a unital \CA\ $B$ {{in}} ${\cal T}$,
and c.p.c maps
$\alpha: A\rightarrow B,$
$\beta_n: B\rightarrow A$, {\green{and}}
$\gamma_n: A\rightarrow A$
($n\in\mathbb{N}$) {\blue{such}} that

(1) $c\approx_{\eta}\gamma_n(c)+\beta_n\circ\alpha(c)$
{{for}} all $c\in{\cal F}$ and $n\in\mathbb{N},$

(2) $\alpha$ is {\blue{an}}
$(\calF, {{\eta}})$-{{approximate embedding}},

(3) $\lim_{n\rightarrow\infty}\|\beta_n(b_1b_2)-\beta_n(b_1)\beta_n(b_2)\|=0$ {\blue{and}}
{$\lim_{n\rightarrow \infty}\|\beta_n(b_1)\|= \|b_1\|$}
for all  $b_1,b_2\in B,$ {{and}}

(4) $\sup\{\tau(\gamma_n(1_A)): \tau\in QT(A)\}<\ep$ for all $n\in\mathbb{N}.$

\noindent
{{Then}} $A\in{\cal T}.$

In particular,  if $A$ is a unital separable simple \CA\, which is
asymptotically tracially in ${\cal T},$
then $A\in {\cal T}.$
\end{thm}

\begin{proof}



Let $\tau\in QT(A).$
Fix $x, y\in {\blue{A_{sa}}}$
{\blue{and fix $1/2>\ep>0.$
Choose $0<\dt<\ep$ which satisfies
the condition in  \cite[II. 2.6]{BH-trace1982}.}}

Fix $0<\eta<\dt.$
 Choose ${\cal F}=\{1_A, x, y, x+y\}.$
Let $B,$ $\af,\, \bt_n$ and $\gamma_n$ be as above associated with
{\blue{$\ep,$ $\eta$}} and ${\cal F}.$
{\blue{By {{Lemma}} \ref{orthogonal_p.c._map},
we may also assume, \wilog,}}

{\blue{(5{{)}} $\|\gamma_n(a)\bt_n\circ \af(a)-\bt_n\circ \af(a)\gamma_n(a)\|<\dt$
for all ${{a}}\in {\cal F}$ {{and all $n\in\N.$}}}}

\noindent
{{Let}} $\omega$ be a free ultra filter on $\N$.
{{Let {\blue{$J:=\{\{a_n\}\in l^\infty(A):\lim_\omega\|a_n\|=0\}.$}}
Note that $J$ is an ideal of $l^\infty(A).$}}
Let $\pi_\omega: l^\infty(A)\to l^\infty(A)/J$ be the quotient map.
Let $\tau_\omega: l^\infty(A)/J\to \C$ be defined by
$\tau_\omega(\pi_\omega(\{a_n\})):=\lim_{n\to\omega}\tau(a_n)$
for all $\{a_n\}\in l^\infty(A).$
Note that $\tau_\omega\in QT(l^\infty(A)/J)$
(see the paragraph above \cite[Corollary II.2.6]{BH-trace1982}).

Define an injective *-homomorphism from $\bt: B \rightarrow  A_\omega$ by
$\bt(x)=
\pi_{\omega}(\{\beta_1(x),\beta_2(x),\cdots\})$ for all $x\in B.$
Then
$\tau_\omega\circ\beta$ is a 2-quasitrace on $B$ (with $\|\tau_\omega\circ \bt\|\le 1$).
Since $B$ is  in ${\cal T},$
\beq
\lim_{i\rightarrow \omega}
\tau\circ\beta_i(\alpha(x)+\alpha(y))
&=&
\tau_\omega\circ\beta(\alpha(x)+\alpha(y))
=
\tau_\omega\circ\beta(\alpha(x))
+\tau_\omega\circ\beta(\alpha(y))
\\
&=&
\lim_{i\rightarrow \omega}
\tau\circ\beta_i(\alpha(x))
+\lim_{i\rightarrow \omega}
\tau\circ\beta_i(\alpha(y))
\\
&=&
\lim_{i\rightarrow \omega}
(\tau\circ\beta_i(\alpha(x))
+\tau\circ\beta_i(\alpha(y))).
\eneq
Therefore there exists $m\in\N$,
such that
\beq
\label{f-QT-2}
\tau\circ\beta_m(\alpha(x)+\alpha(y))
\approx_{\delta}
\tau\circ\beta_m(\alpha(x))
+\tau\circ\beta_m{\blue{(\alpha(y)).}}
\eneq
Note that, for any $a\in A_{sa},$  $\|\tau(a)\|\le \|\tau\|\|a\|\le  \|a\|$ (see \cite[II.2.5, (iii)]{BH-trace1982}).
{{Then}}

\beq
{\blue{\tau(x+y)}}
&\approx_{{{\eta}}}&
\tau(\beta_m\circ\alpha(x+y)+\gamma_m(x+y))
\nonumber\\
\mbox{{\blue{(by (5) and \cite[II.2.6]{BH-trace1982})}}}
&\approx_{{\green{\|x+y\|\ep}}}&
\tau(\beta_m\circ\alpha(x+y))+\tau(\gamma_m(x+y))
\nonumber\\
\mbox{by (4)}
&\approx_{\|x+y\|\ep}&
\tau(\beta_m\circ\alpha(x+y))
\nonumber\\
\mbox{(by \eqref{f-QT-2})}
&\approx_{{\green{\eta}}}&
\tau\circ\beta_m(\alpha(x))
+\tau\circ\beta_m(\alpha(y))
\nonumber\\
&\approx_{(\|x\|+\|y\|){\green{\ep}}}&
\tau\circ\beta_m(\alpha(x))+\tau(\gamma_m(x))
+\tau\circ\beta_m(\alpha(y))+{\blue{\tau(\gamma_m(y))}}
\nonumber\\
\mbox{{{(by (5) and \cite[II.2.6]{BH-trace1982})}}}
&\approx_{{\green{(\|x\|+\|y\|)\ep}}}&
\tau(\beta_m(\alpha(x))+\gamma_m(x))
+\tau(\beta_m(\alpha(y))+{\blue{\gamma_m(y)}})
\nonumber\\
\mbox{(by (1))}&\approx_{\green{2\eta}}&
\tau(x)+\tau(y).
\eneq
{\blue{Let $\ep,\ {{\eta}}\to 0.$}}
{\blue{We}} have $\tau(x+y)=\tau(x)+\tau(y).$
{\blue{It follows that}} $\tau$ is linear.  In other words,
$\tau$ is a trace.

To see the last part, assume that $A$ is a unital separable simple \CA\, which is asymptotically tracially in
${\cal T}.$ We may assume that $A$ is infinite dimensional.
Then, for any $\ep>0,$
by Lemma \ref{L1991} (cf. \cite[3.5.7]{Lnbook}),
there is a nonzero positive element
$a\in A$ with $\|a\|=1$ such that
$\sup\{{{d_\tau(a)}}: \tau\in QT(A)\}<\ep.$  By
%
the Definition \ref{def-tracial_approximation}
and  applying what has been proved,
{{we conclude}} that
every 2-quasitrace of $A$ is a trace.

\end{proof}

{\blue{\begin{cor}\label{CQT=T}
If $A$ is asymptotically tracially in ${\cal E},$ in particularly, in  ${\cal N},$  then $QT(A)=T(A).$
\end{cor}}}

The proof of the following is taken from the proof of \cite[3.6.10]{Lnbook}
(see also \cite[Theorem 3.4]{Lin 2001-1},
{{\cite[3.3]{FF-05-tsr}, and \cite[4.3]{EN08}}}).
{{Recall that a \CA\ $A$ is called has
(SP) property, if every nonzero hereditary $C^*$-subalgebra of $A$
contains a nonzero projection.}}

\begin{thm}[cf. {\cite[Theorem 3.6.10]{Lnbook}}]
\label{Tstablerank1}
Let ${\cal S}$ be the class of unital \CA s with stable rank one.  Suppose that $A$ is a unital simple \CA\,
satisfying the following condition:
{{For}} any finite subset $\calF\subset A$,  any $\epsilon>0$, and any
$a \in A_+\backslash\{0\}$,
there exists a unital
$C^*$-subalgebra
$B\subset  l^\infty(A)/c_0(A)$ which is strictly embedded
such that $B$ in $\mathcal{S}$, and

(1) $1_B\iota_A(x)\approx_{\epsilon}\iota_A(x)1_B$\, {\blue{for all}} $ x\in\calF$,

(2) $1_B\iota_A(x)1_B\in_{\epsilon} B$ {\blue{and}}  $\|1_B\iota_A(x)1_B\|\geq \|x\|-\epsilon$\,
{\blue{for all}} $ x\in \calF$,  {\blue{and}}

(3) $\iota_A(1_A)-1_B\lesssim_{
l^\infty(A)/c_0(A)}\iota_A(a)$.

\noindent
{\blue{{{Then}} $A$ {{in}} ${\cal S}.$}}
{{Consequently,}} if $A$
is asymptotically tracially in ${\cal S},$
then $A$ {{in}} ${\cal S}.$
\end{thm}

\begin{proof}
Note that \CA s in ${\cal S}$ are stably finite
(see {\blue{\cite[Proposition 3.3.4]{Lnbook}}}).
One may assume that $A$ is infinite dimensional. Let $x\in A.$
{\blue{It will be shown}} that, for any
{{$\ep\in(0,1/2),$}}
there exists an invertible
element $y\in A$ such that $\|x-y\|<\ep.$ One may assume that $\|x\|\le 1$ and $x$ is not invertible.
{\blue{As $A$ is stably finite, {{one}} may assume that $x$ is not one-sided invertible.}}
To show that $x$ is a norm limit of invertible elements, it suffices to show that $ux$ is a norm limit of invertible elements
for some unitary $u\in A.$ Thus, by  \cite[Lemma 3.6.9]{Lnbook}
({\blue{also see \cite[Lemma 3.5]{Rordam-1991-UHF}}}),
{\blue{one}} may assume
that there exists a nonzero element $c_1\in A_+$ such that $c_1x=xc_1=0.$

{\blue{First}} consider the case  that $A$ has (SP) property.
Then, by \cite[Lemma 3.6.6]{Lnbook},
there are nonzero mutually orthogonal projections $p_1, p_2\in \Her(c_1).$
Consider $A_1=(1-p_1)A(1-p_1).$
Since $A$ is simple and has (SP) property,
there is a nonzero projection
{{$p_1'\in A_1$ such that}}
$p_1'\lesssim p_1$
{{(see, {\blue{for example,}} \cite[Lemma 3.5.6]{Lnbook}).}}
Note $x\in A_1.$ Since ${\cal S}$ has property (H)
{{(see \cite[Corollary 3.6]{BP95}),}}
by  {{Proposition}} \ref{310},
$A_1$   has the same property that $A$ has, namely,
there is a projection
{{$q\in 
l^\infty(A_1)/c_0(A_1)$}}  and
a \SCA\, $B$ of $l^\infty(A_1)/c_0(A_1)$ with $B\in {\cal S}$  and with $1_B=q$ such that

{{($1'$)}} $\|q\iota_{A_1}(x)-\iota_{A_1}(x)q\|<\ep/32{{,}}$

{{($2'$)}} ${{q\iota_{A_1}(x)q}}\in_{\ep/32} B{{,}}$ and

{{($3'$)}} $\iota_{A_1}(1_{A_1})-q\lesssim_{l^\infty(A_1)/c_0(A_1)} \iota_{A_1}(p_1')\lesssim_{l^\infty(A)/c_0(A)}  \iota_{A}(p_1).$

\noindent
{{Write}} $x_1=q\iota_{A_1}(x)q$ and $x_2=(\iota_{A_1}(1_{A_1})-q)\iota_{A_1}(x)(\iota_{A_1}(1_{A_1}-q).$
{\blue{Then,}} by ($1'$), {\blue{one has}
\beq\label{STr-0}
\|\iota_{A_1}(x)-(x_1+{{x_2}})\|<\ep/16.
\eneq
Since $B\in {\cal S},$ there is an invertible element $y_1\in B$ such that
\beq\label{STr-1}
\|x_1-y_1\|<\ep/16.
\eneq
By {{($3'$)}}, there is $v\in l^\infty(A)/c_0(A)$ such that $v^*v=\iota_{A_1}(1_{A_1})-q=\iota_A(1_A-p_1)-q$
and $vv^*\le \iota_A(p_1).$  Set $y_2:=x_2+(\ep/16)v+(\ep/16)v^*+(\ep/16)(\iota_A(p_1)-vv^*).$
Note that $y_3:=x_2+(\ep/16)v+(\ep/16)v^*$ has the form}
$$
\begin{pmatrix} x_2 & {{(\ep/16)v^*}}\\
                        {{(\ep/16)v}}   & 0\end{pmatrix}.
                        $$
One checks that $y_3$ is invertible in
{{$\Her_{l^\infty(A)/c_0(A)}((\iota_A(1_A-p_1)-q)+vv^*).$}}
Therefore $y_2$ is invertible in
${{\Her_{l^\infty(A)/c_0(A)}(\iota_A(1_A)-q).}}$
Hence $y_1+y_2$ is invertible in $l^\infty(A)/c_0(A).$
Moreover,
\beq\label{STr-2}
\|x_2-y_2\|<\ep/8.
\eneq
Finally, one has   (by {{\eqref{STr-0}, \eqref{STr-1}}} and \eqref{STr-2})
\beq
\label{f4-9-1}
\nonumber
\hspace{-0.3in}\|\iota_A(x)-(y_1+y_2)\|
&\le&
{{\|\iota_A(x)-(x_1+x_2)\|+\|x_1-y_1\|+\|x_2-y_2\|}}\\\label{STr-10}
\hspace{-0.2in}&<&
\ep/16+\ep/16+\ep/8=\ep/4.
\eneq
Let $z\in l^\infty(A)/c_0(A)$ {\blue{be}} such that $z(y_1+y_2)=(y_1+y_2)z=1_{l^\infty(A)/c_0(A)}.$
Let $\{z(n)\}, \{y(n)\}\in l^\infty(A)$ such that $\pi_\infty(\{z(n)\})=z$ and $\pi_\infty(\{y(n)\})=y_1+y_2.$
Then, for all large $n,$ $\|z(n)y(n)-1_A\|<1/2$ and $\|y(n)z(n)-1\|<1/2.$ It follows
that $y(n)$ is invertible for all sufficiently large $n.$  By \eqref{f4-9-1}, for all sufficiently large $n,$
$$
\|x-y(n)\|<\ep.
$$
\noindent
This proves the case that $A$ has (SP) property.

If $A$ does not have (SP) {{property}},
one does not choose $p_1$ and $p_2.$ However, there is $a\in A_+\setminus\{0\}$
such that $\Her(a)$ has no nonzero projection.  Replacing $p_1$ by $a$ above.
Since $\gamma_n(1_A)$ is a projection,
$\gamma_n(1_A)\lesssim a$ implies
{\blue{that there is  $s\in A$ such that
$s^*s=\gamma_n(1_A)$ and $ss^*\in
\Her(a)$ (see {{Lemma}} \ref{projection-Cuntz-subequivalent-lem}) which forces $\gamma_n(1_A)=0.$}}
Thus, in this case,
one may assume that $\gamma_n=0$.
Argument becomes simpler.  Indeed, choosing $A_1=A,$
then $x\approx_{\ep/16} x_1 \approx_{\ep/16} y_1.$

The last part of the statement follows the first part and {{Proposition}}
\ref{equivalent-definition-of-tracial-approximation}.
\end{proof}

\begin{lem}\label{L410}
\label{tracially-almost-unperforated-lemma}
Let ${\cal W}$ be the class of unital \CA s whose Cuntz semigroup is almost unperforated
{{(recall {{Definition}} \ref{Dcuntz}).}}
Let   $A$ be  a unital  simple \CA\, which  is
asymptotically tracially in ${\cal W}$ and
$a,b,c\in A_+\backslash\{0\}.$
Suppose  that there exists $n\in\N$ satisfying $(n+1)\la a\ra\leq n\la b\ra.$
Then, for any $\ep>0$,
there exist $a_1,a_2\in A_+$ and a projection $p\in A$
such that

(1) $a\approx_{\epsilon}a_1+a_2$,

(2) $a_1\lesssim _A b$, {{and}}

(3) $a_2\leq \|a\|p\lesssim_A c$.

\end{lem}


\begin{proof}

Without loss of generality, {{one}} may assume that
$a,b,c\in A_+^1\backslash\{0\}$ and $\epsilon<1/2$.
%
Let $\{e_{i,j}\}$ be a set of matrix units of $M_{n+1}.$
Then $a\otimes \sum_{i=1}^{n+1}e_{i,i}\lesssim_{A\otimes M_{n+1}}
b\otimes \sum_{i=1}^{n}e_{i,i}$.
Let $r=\sum_{i,j=1}^{n+1}r_{i,j}\otimes e_{i,j}\in A\otimes M_{n+1}$
such that $a\otimes \sum_{i=1}^{n+1}e_{i,i}
\approx_{\epsilon/64}r^*(b\otimes \sum_{i=1}^{n}e_{i,i})r$.
Set
$$\calF:=\{a,b\}\cup\{r_{i,j},r_{i,j}^*:i,j=1, {\blue{2},} {{\cdots,}} n+1\}.$$
Let $M:=1+\|r\|$ and choose
$\delta_1:=\frac{\epsilon}{64M^2(n+1)^4}.$
Note that
\beq\label{410-n1}
r^*(b\otimes \sum_{i=1}^{n+1}e_{i,i})r
\approx_{\epsilon/64}r^*((b-\delta_1)_+\otimes \sum_{i=1}^{n+1}e_{i,i})r.
\eneq

{\blue{\noindent
{{Note}} that ${\cal W}$ has property (H)
(see the line {{following {{Definition}}  \ref{Dcuntz}}}).}}
Since $A$ is asymptotically tracially in ${\cal W},$
by Proposition \ref{hereditary-subalgebra-preserves-tracially-approximation},
for
any $\dt>0,$
there exist a unital \CA\ $B$ with  almost unperforated
$W(B),$
{{c}}.p.c.~maps $\alpha:A\rightarrow B,$
$\beta_i:B\rightarrow A$,
{{and}} $\gamma_i:A\rightarrow A\cap \beta_i(B)^{\bot}$
{\blue{($i\in\N$)}} {{such}} that

($1'$) $\alpha$ is {\blue{a u.c.p.~map,}}
$\beta_i(1_B)$ {\blue{and}} $\gamma_i(1_A)$ are projections, {\blue{and}}
$1_A=\beta_i(1_B)+\gamma_i(1_A)$
for all $i\in \mathbb{N}$,

($2'$) $x\approx_{\delta}\gamma_i(x)+\beta_i\circ\alpha(x)$
{{for}} all $x\in\calF$ and all $i\in\mathbb{N}$,

({{$3'$}}) $\alpha$ is {\blue{an}} $(\calF, \delta)$-{{approximate embedding}},

({{$4'$}}) $\lim_{i\rightarrow\infty}\|\beta_i(xy)-\beta_i(x)\beta_i(y)\|=0$ and
$\lim_{n\rightarrow \infty}\|\beta_i(x)\|= \|x\|$
for all $x,y\in B$, and

({{$5'$}}) $\gamma_i(1_A)\lesssim_A c$ for all  $i\in\mathbb{N}$.


\noindent
{{By}} ({{$3'$}})  and  \eqref{410-n1},
for some  sufficiently small $\dt$\,($<(\frac{\epsilon}{128M(n+1)^2})^4$),
one has
\beq\nonumber
\alpha(a)\otimes \sum_{i=1}^{n+1}e_{i,i}\approx_{\ep/16}
(\sum_{i,j=1}^{n+1}\alpha(r_{i,j})\otimes e_{i,j})^*((\alpha(b)-\delta_1)_+\otimes \sum_{i=1}^{n}e_{i,i})
(\sum_{i,j=1}^{n+1}\alpha(r_{i,j})\otimes e_{i,j}).
\eneq

\noindent
{{By}} \cite[Proposition 2.2]{Rordam-1992-UHF2},  with $R:=(\sum_{i,j=1}^{n+1}\alpha(r_{i,j})\otimes e_{i,j}),$  in $B\otimes M_{n+1},$
\begin{eqnarray*}
&&\hspace{-0.6in}(\alpha(a)-\epsilon/8)_+\otimes \sum_{i=1}^{n+1}e_{i,i}
=
((\alpha(a)\otimes \sum_{i=1}^{n+1}e_{i,i})-\epsilon/8)_+\\
&\lesssim&
R^*((\alpha(b)-\delta_1)_+\otimes \sum_{i=1}^{n}e_{i,i})
R
\lesssim
(\alpha(b)-\delta_1)_+\otimes \sum_{i=1}^{n}e_{i,i}.
\end{eqnarray*}

\noindent
{{Since}} $W(B)$ is almost unperforated,
one obtains $(\alpha(a)-\epsilon/8)_+
\lesssim_{B} (\alpha(b)-\delta_1)_+.$
Hence there exists $s\in B$ such that
$$
(\alpha(a)-\epsilon/8)_+
\approx_{\epsilon/64}
s^*(\alpha(b)-\delta_1)_+s.
$$
Then, by ({{$4'$}}), there exists $N\in\mathbb{N}$ such that
\beq
(\beta_N(\alpha(a))-\epsilon/8)_+
&\approx_{\epsilon/32}&
\beta_N((\alpha(a)-\epsilon/8)_+)
\approx_{\epsilon/32}
\beta_N(s^*(\alpha(b)-\delta_1)_+s)\\
&\approx_{\epsilon/32}&
\beta_N(s^*)(\beta_N(\alpha(b))-\delta_1)_+\beta_N(s).
\eneq
%

\noindent
Applying \cite[Proposition 2.2]{Rordam-1992-UHF2}, one has
\beq
(\beta_N(\alpha(a))-\epsilon/4)_+
\lesssim_A
\beta_N(s^*)(\beta_N(\alpha(b))-\delta_1)_+\beta_N(s)
\lesssim_A
(\beta_N(\alpha(b))-\delta_1)_+.
\label{tracially-almost-unperforated-lemma-eq1}
\eneq
Since
$\beta_N(\alpha(b))+\gamma_N(b)\approx_{\delta}b$,
with $\dt<\dt_1,$ applying \cite[Proposition 2.2]{Rordam-1992-UHF2} again (noting $\gamma_N(b)\perp\bt_N(B)$),  one has
\beq
(\beta_N(\alpha(b))-\delta_1)_+
\leq
(\beta_N(\alpha(b))-
\dt)_+\le ((\bt_N(\af(b))-\dt)_++\gamma_N(b))-\dt)_+
\lesssim_A b.
\label{tracially-almost-unperforated-lemma-eq2}
\eneq

\noindent
{{Choose}} $a_1:=(\beta_N(\alpha(a))-\epsilon/4)_+$,
$a_2:=\gamma_N(a)$ and
$p:=\gamma_N(1_A).$
Then, by (\ref{tracially-almost-unperforated-lemma-eq1})
and (\ref{tracially-almost-unperforated-lemma-eq2}),
one has $a_1\lesssim_A b$.
{\blue{Note that}} ({$5'$}) shows $a_2\leq \|a\|p\lesssim_A c$.
Thus $a_1,\,a_2,\,p$ satisfy (2) {\blue{and}} (3) of the lemma.
By ($2'$),
$$a\approx_{\delta}\gamma_N(a)+\beta_N(\alpha(a))
\approx_{\epsilon/4}
\gamma_N(a)+(\beta_N(\alpha(a))-\epsilon/4)_+
=a_2+a_1.
$$
So (1) of the lemma is also satisfied and  the lemma follow{{s.}}


\end{proof}

\begin{thm}
\label{tracially-Z-stable-has-almost-unperforated-Cuntz-semigroup}
Let $A$ be a unital simple \CA\, which is
asymptotically tracially  in ${\cal W}$ (see Lemma \ref{L410}).
Then $A\in {\cal W}.$

\end{thm}

\begin{proof}
Let $a,b\in M_m(A)_+\setminus \{0\}$ with $\|a\|=1=\|b\|$ for some integer $m\ge 1.$
Let $n\in\mathbb{N}$ and assume
 $(n+1)\la a\ra\leq n\la b\ra.$ To prove the theorem, it suffices to {\blue{prove}} that $a\lesssim b.$

Note that,  if $B\in {\cal W},$ then, for each integer $m,$ $M_m(B)\in {\cal W}.$
It follows that
$M_m(A)$ is
asymptotically tracially in ${\cal W}.$ To simplify notation, \wilog, one may {\blue{assume}}
 $a, b\in A_+.$

{\blue{First}} {\blue{consider the case}} that $A$ has (SP) property.
By {{Lemma}} \ref{L1991}, 
{{$\Her(f_{1/4}(b))_+$}} contains $2n+1$ {{nonzero}} mutually orthogonal   elements
$b_0, {\blue{b_1,}}
{{\cdots}},b_{2n}$ such that $\la b_i\ra =\la b_0\ra,$
$i=1, {\blue{2,}}
{{\cdots}},2n.$
Since $A$ has (SP) {{property}}, choose a nonzero projection $e_0\in \Her(b_0).$
Replacing $b$ by $g(b)$ for some $g\in C_0((0, 1]),$ one may assume that $be_0=e_0b=e_0.$
Put $c=b-{{e_0}}.$
Keep in mind that $b=c + e_0,$ $c\perp e_0,$ and ${{2n}}\la e_0\ra \le c=b-e_0.$
{\blue{One has}}
\beq
(2n+2)\la a\ra \le 2n \la b\ra = 2n(\la b-e_0\ra +\la e_0\ra)\le 2n\la c\ra +\la c\ra=(2n+1)\la c \ra.
\eneq

\noindent
By Lemma \ref{tracially-almost-unperforated-lemma},
for any {{$\ep\in(0,1/2),$}} 
there exist $a_1,a_2\in A_+$ such that

(i) $a\approx_{\epsilon/2}a_1+a_2$,

(ii) $a_1\lesssim_A c,$ and

(iii) $a_2\le \|a\|p\lesssim_A e_0$.

\noindent
{{By}}  (i), (ii) and (iii), and applying  \cite[Proposition 2.2]{Rordam-1992-UHF2}, one obtains (recall $be_0=e_0b=e_0$)
\beq
(a-\ep)_+\lesssim a_1+a_2 \lesssim  c+e_0=b.
\eneq
Since this holds for every
{{$\ep\in(0,1/2),$}} 
one concludes that  $a\lesssim b.$

If $A$ does not have (SP) {{property}}, choose $b_0\in A_+\setminus \{0\}$ such that $\Her(b_0)$ has no nonzero projections.
From $(n+1)\la a\ra \le n\la b\ra,$ {\blue{Lemma}} \ref{L410} implies
that $a\approx_{\ep} a_1+a_2,$ $a_1\lesssim b$ and $a_2\le p\lesssim b_0.$  {{Projectionlessness}} of $\Her(b_0)$ forces
$p=0,$ whence $a_2=0.$ Thus one arrives
\beq
(a-\ep)_+\lesssim a_1\lesssim b.
\eneq
It follows  $a\lesssim b$ and the lemma follows.

\end{proof}

\section{Order zero maps and nuclear dimension}


\begin{df}[{\blue{\cite[Definition 2.3]{WZ2009}}}]
\index{Order zero map}
Recall that a c.p.~map
$\phi:A\rightarrow B$
{\blue{has order zero}}, if,
for any $a,b\in A_+$ with $a\cdot b=0$, {\blue{one has}} $\phi(a)\cdot\phi(b)=0$.

\end{df}
{\blue{We would like to recall the following theorem.}}

\begin{thm}[{\blue{\cite[Theorem 3.3]{WZ2009}}}]
\label{WZ2009-Theorem 3.3}
Let $A$ and $B$ be $C^*$-algebras, and {\blue{let}}
$\phi:A\rightarrow B$ be a completely positive order zero map.
Let $C:= C^*(\phi(A))\subset B$.
Then there exists a positive element $h\in \mathcal{M}(C)\cap C'$
with $\|h\|=\|\phi\|$
and 
{{a}}  *-homomorphism
$
\pi_{\phi}:
A\rightarrow
\mathcal{M}(C)\cap \{h\}'
$
{\blue{such}} that
$
\phi(a)=\pi_{\phi}(a)h
$
for all $a\in A.$
%
If, {\blue{in addition,}} $A$ is unital, then $h=\phi(1_A)\in C$.

\end{thm}

\begin{prop}
\label{simple_order_zero_maps}

Let $\phi: A\rightarrow B$ be a c.p.~order zero map.
Let $h$ {\blue{and}} $\pi_{\phi}$ be as in Theorem \ref{WZ2009-Theorem 3.3}.
If $A$ is simple,  {\blue{then the map $a\otimes x\mapsto \pi_{\phi}(a)\cdot x$}}
defines an isomorphism
$\gamma: A\otimes C^*(h)\cong C^*(\phi(A)).$
Moreover, for all $a\in A$,
$\|\phi(a)\|=\|\phi\|\cdot\|a\|$.

\end{prop}


{

\begin{proof}
If $\|\phi\|=0$, {\blue{then $h=0$ and}} there is nothing to prove.
Assume that $\|\phi\|\neq 0$.
Since $A$ is simple,
{\blue{$\pi_\phi$ is injective}} and
$\pi_\phi(A)$ is also simple.

By (the proof of) \cite[Corollary 4.1]{WZ2009},
$\gamma$ gives a *-homomorphism
from $A\otimes C^*(h)$ to $C^*(\phi(A))$.
Since $\phi(A)\subset \gamma(A\otimes C^*(h)),$
$\gamma$ is surjective.
%

{\blue{Let us}} show that $\gamma$ is injective.
Since $A$ is simple, $\ker \gamma=A\otimes I,$ where $I$ is an ideal of $C^*(h)$
{{(see \cite[Proposition 2.16.(2) and Proposition 2.17(2)]{BK2004}).}}
Let $f(h)\in I$ for some $f\in C_0({\rm sp}(h)\setminus \{0\}).$
Then $a\otimes {{f(h)}}\in A\otimes I=\ker \gamma$
for all $a\in A,$ which implies that
$\pi_\phi(a)f(h)=0$ for all $a\in A.$
It follows that $\phi(a)f(h)=\pi_\phi(a)hf(h)=f(h)\pi_\phi(a)h=f(h)\phi(a)$ and $\phi(a)f(h)=\pi_\phi(a)f(h)h=0.$
Thus $f(h)\perp C^*(\phi(A))=C.$ Since $f(h)\in \mathcal{M}(C),$ this implies $f(h)=0.$ Thus  $I=\{0\}.$
In other words, $\gamma$ is injective.

Moreover, {\blue{recall,}} from Theorem \ref{WZ2009-Theorem 3.3},
$\|\phi\|=\|h\|$. Then, for $a\in A$,
$
\|\phi(a)\|=\|h\cdot\pi_{\phi}(a)\|=\|\gamma(\pi_{\phi}(a)\otimes h)\|
=\|\pi_{\phi}(a)\otimes h\|
=\|\pi_{\phi}(a)\|\cdot\|h\|
=\|a\|\cdot\|\phi\|
$.

\end{proof}

}


\begin{rem}
(1) {\blue{For the}} case {\blue{that}} $A$ is a matrix  algebra, {\blue{the proposition  above}} was obtained
in the proof of \cite[Proposition 5.1]{KW-Covering-dimension}.

(2) Consider $\phi: \mathbb{C}\oplus\mathbb{C}\rightarrow \mathbb{C}\oplus\mathbb{C}$,
$(x,y)\mapsto(x,y/2).$ {\blue{Then}} $\phi$ is an injective norm one c.p.c.~order zero map,
but $\phi$ is not an isometry
since $\|\phi((1,2))\|=1<2=\|(1,2)\|$.
Thus the last statement of
Proposition \ref{simple_order_zero_maps} would fail without {\blue{the assumption that $A$ is simple.}}

\end{rem}


The following proposition
shows the existence of inverse *-homomorphism
for norm one c.p.~order zero map {\blue{from}} simple $C^*$-algebras.

\begin{prop}
\label{homomorphism-inverse-of-simple-order-zero-map}

Let $A$ be a simple \CA, $B$ be another \CA, and let $\phi: A\rightarrow B$ be a nonzero c.p.~order zero map.
Then there exists a *-homomorphism
$\psi: C^*(\phi(A))\rightarrow A$
{\blue{such}} that $\psi\circ\phi=\|\phi\|\cdot\mathrm{id}_A$
and $\phi\circ\psi|_{\phi(A)}=\|\phi\|\cdot\mathrm{id}_{\phi(A)}$.

\end{prop}

\begin{proof}

We will use the same notation as
in Proposition \ref{simple_order_zero_maps},
such as $h,\pi_{\phi}$,
and {\blue{the}}  isomorphism:
$\gamma: A\otimes C^*(h)\rightarrow C^*(\phi(A))$,
$a\otimes x\mapsto \pi_{\phi}(a)\cdot x$.

Note that
$C^*(h)\cong C_0(\mathrm{sp}(h)\backslash\{0\})$ and $\|\phi\|=\|h\|.$
Define a *-homomorphism
$\psi': A\otimes C^*(h)\to A$ by $\psi'(a\otimes {{f(h)}})=
f(\|h\|)a $
for all $a\in A$ and $f\in C_0(\mathrm{sp}(h)\backslash\{0\}),$ and define
$\psi=\psi'\circ \gamma^{-1}: C^*(\phi(A))\to A$.
Then,  with {{the}} identity function $\imath: {\rm sp}(h)\to {\rm sp}(h),$  for any $a\in A,$
$$
\psi\circ \phi(a)=\psi'\circ \gamma^{-1}(\pi_\phi(a)h)=\psi'(a\otimes \imath)=a\|h\|=\|\phi\|a.
$$
Therefore, for $a\in A,$
$\phi\circ \psi(\phi(a))=\phi\circ \psi'\circ \gamma^{-1}(\pi_\phi(a)h)=\phi\circ \psi'(a\otimes {{h}})=\phi(\|\phi\|a)=\|\phi\|\phi(a).$
{\blue{The}} proposition follows.

\end{proof}

\begin{prop}
\label{inverse-of-finite-dim-injective-order-zero-map}

Let $A$ be a \CA, $F$ be a (nonzero) finite dimensional \CA, and let $\alpha: F\rightarrow A$
be an injective c.p.~order zero map.
Then  there exists a c.p.~map $\beta: A\rightarrow F$ such that
$\beta\circ \alpha=\mathrm{id}_F$.

Moreover, if $\alpha$ is {{an}} isometry,  one may choose $\beta$
to be a c.p.c.~map.
\end{prop}

\begin{proof}

Write $F=M_{k_1}\oplus \cdots\oplus M_{k_n}$
{\blue{($n,
k_1,\cdots,k_n\in\mathbb{N}$)}} {\blue{and}}
$\alpha_i:=\alpha|_{M_{k_i}}: M_{k_i}\rightarrow A$
{{($i=1,{\green{2,}}\cdots, n$).}}
{\blue{Then,}} by Proposition \ref{homomorphism-inverse-of-simple-order-zero-map},
there exists  {\blue{a}} *-homomorphism
$\beta_i:C^*(\alpha_i(M_{k_i}))\rightarrow M_{k_i}$
such that $\frac{1}{\|\alpha_i\|}\beta_i\circ\alpha_i=\mathrm{id}_{M_{k_i}}$.
{\blue{Then the map}}
{\blue{$\bar{\beta}:C^*(\alpha_1(M_{k_1}))\oplus\cdots\oplus C^*(\alpha_n(M_{k_n}))\rightarrow F=M_{k_1}\oplus \cdots\oplus M_{k_n}$
defined by $\bar{\bt}((x_1,\cdots,x_n))=
(\frac{\beta_1(x_1)}{\|\alpha_1\|},\cdots,\frac{\beta_n(x_n)}{\|\alpha_n\|})$
{\blue{is a c.p.~map.}}}}
{{Since $\af$ is a c.p.~order zero map,}}
$C^*(\alpha_i(M_{k_i}))$ are mutually orthogonal
($i=1,{\green{2,}}\cdots,n$).
{\blue{Thus}}
$C^*(\alpha_1(M_{k_1}))\oplus\cdots\oplus C^*(\alpha_n(M_{k_n}))$
is a
{\blue{\SCA}}\, of $A$.
By Arveson's extension theorem,
$\bar{\beta}$ has a c.p.~extension $\beta: A\rightarrow F$ with
$\beta\circ\alpha=\mathrm{id}_F$.
Moreover, if $\alpha$ is {\blue{an}} isometry,
then $\bar{\beta}$ is {{a}} *-homomorphism. {{Hence}}
the extension $\beta$ can be chosen to be {{a}} c.p.c.~map.

\end{proof}

\begin{df}
Let $F=M_{k_1}\oplus\cdots \oplus M_{k_n}$ be a finite dimensional
\CA. Let $A$ be a \CA\, {{and}} $\phi: F\rightarrow A$ {{be}} a linear map.
{\blue{Define}}
$$ |||\phi|||:=\max\{\|\phi|_{M_{k_i}}\|:
i=1,{\green{2,}}\cdots,n\}.$$

\end{df}

\begin{df}
\label{defn-n-decomposable-c.p.-map}
Let $A$ be a \CA\ and $F$ be a finite dimensional \CA\ and let
$\phi: F\rightarrow A$ be a c.p.~map.
Fix  $n\in\mathbb{N}$. {{Recall that}} the map
$\phi$ is called $n$-decomposable
{{(see \cite[Definition 2.2]{KW-Covering-dimension}),}}
if
$F$ can be written as $F=F_0\oplus \cdots\oplus F_n$
(where  $F_i$ {\blue{is a finite dimensional \CA}})
such that $\phi|_{F_i}$ 
{\blue{is a
c.p.~order zero map}}
($i=0,{1,}\cdots,n$).
If, {{in addition,}}  each $\phi|_{F_i}$ is  {{assumed}} to be contractive,
then $\phi$ is called
{\blue{{\em piecewise contractive}}} $n$-decomposable map.
\end{df}



\begin{rem}
\label{ideal-order-zero-map}
{\blue{Note {\blue{that}} Theorem \ref{WZ2009-Theorem 3.3} implies
the kernel of a c.p.~order zero map is always an ideal
(also see \cite[Lemma 2.7]{LeeOsaka}).}}
{{Thus,}} for a c.p.~order zero map $\phi: F\rightarrow A,$ {{where}}
$F$ is finite dimensional, {\blue{one}} can write
$F={{\ker\phi}}\oplus F_1,$
where $F_1$ is 
{{an ideal}} of $F.$
Note that $\phi|_{F_1}$ is injective.

\end{rem}

\begin{prop}
\label{normalization-of-c.p.-dividable-map}

Let $A$ {\blue{and}}
$C$ be $C^*$-algebras, {\blue{and $B$ be  a finite dimensional \CA.}}
Suppose that
$\alpha: A\to B$ and $\bt: B\to C$
are c.p.~maps
{\blue{such that
$\bt$ is $n$-decomposable.}}
Then there  {\blue{exist}} $\bar{n}\leq n\in\mathbb{N}\cup\{0\}$,
a finite dimensional $C^*$-algebra
$\bar{B}=\bar{B}_0\oplus\cdots\oplus \bar{B}_{\bar{n}}$
{\blue{which is a summand
 of $B,$}}
a c.p.~map $\bar{\alpha}: A\rightarrow \bar{B},$
and
{\green{a}} c.p.~$\bar{n}$-decomposable map
$\bar{\beta}: \bar{B}\rightarrow C$ {{such}} that

(1) $\bar{\beta}\circ\bar{\alpha}=\beta\circ\alpha$,

(2) $\|\bar{\alpha}\|\leq \min\{\|\beta\circ\alpha\|,
\|\alpha\|\cdot |||\beta|||\}$, {\blue{and}}

(3) $\bar{\beta}|_{\bar{B}_i}$
{\blue{is a}} c.p.c.~order zero {\blue{isometry}}
($i=0, 1,\cdots,\bar{n}$).

\end{prop}

\begin{proof}

Let $\bar{n}$ be the minimal integer such that $\beta$ is $\bar n$-decomposable.
Then we can write $B=B_0\oplus\cdots\oplus B_{\bar{n}}$
(where {\blue{each}} $B_i$
{\blue{is a direct summand of
$B$)
such that $\beta|_{B_i}$
is a
nonzero c.p.~order zero map.}}
By Remark \ref{ideal-order-zero-map}, we can
write $B_i=\ker(\beta|_{B_i})\oplus \bar{B}_i,$
where $\bar{B}_i$ is
{\blue{direct summand}}
of $B_i.$
Then $\beta|_{\bar{B}_i}$
{\blue{is a}} 
nonzero injective
c.p.~order zero map
%
($i=0,{\green{1,}}\cdots,\bar{n}$).
Define $\bar{B}:=\bar{B}_0\oplus\cdots\oplus \bar{B}_{\bar{n}}$.
Note that $\bar{B}$ is
a 
{\blue{direct summand}}
of $B$.

Write $\bar{B}=M_{k_1}\oplus\cdots\oplus M_{k_m}$,
where
$m,k_1,\cdots,k_m\in\mathbb{N}$.
{{Let}} $P_j: {B}\to M_{k_j}$ be the projection map.
{\blue{Set}} $\af^{(j)}=P_j\circ \af$ and $\bt^{(j)}=\bt|_{M_{k_j}}$
($j=1,{\green{2,\cdots}},m$).
Note that each ${\bt}^{(j)}$ is a c.p.~order zero map.
{\blue{Define}}  $\bar{\alpha}^{(j)}:=\|\beta^{(j)}\|\alpha^{(j)}$ and
$\bar{\beta}^{(j)}:=\frac{1}{\|\beta^{(j)}\|}\beta^{(j)}$
{\blue{$j=1, {\blue{2,}} {\cdots},m.$}}
By Proposition \ref{simple_order_zero_maps}, each
$\bar{\beta}^{(j)}$ is a c.p.c.~order zero isometry.
Note that
\beq
\beta\circ\alpha(x)
=
\sum_{j=1}^m
\beta^{(j)}\circ\alpha^{(j)}(x)\
{{
=
\sum_{j=1}^m
\bar\beta^{(j)}\circ\bar\alpha^{(j)}(x)
}}
{\blue{\rforal x\in A.}}
\label{f4-11-1}
\eneq

\noindent
Define c.p.~maps
$
\bar{\alpha}:
A\rightarrow
\bar{B}=M_{k_1}\oplus\cdots\oplus M_{k_m}$ by
$x\mapsto
(\bar{\alpha}^{(1)}(x),\cdots,
\bar{\alpha}^{(m)}(x))$ and
%
$
\bar{\beta}:
\bar{B}=M_{k_1}\oplus\cdots\oplus M_{k_m}
\rightarrow A$ by
$(x_1,\cdots,x_m)
\mapsto
\sum_{j=1}^m
\bar{\beta}^{(j)}(x_j).
$

{\blue{Write, for each $i,$  $\bar{B}_i=\bigoplus_{S_i} M_{k_j},$ where
$S_i$ is a subset of $\{1, 2, {{\cdots}},m\}.$}}
{\blue{Again, since ${\bt}|_{\bar{B}_i}$ is a c.p.~order zero map,   ${\bt}(M_{k_j})\perp {\bt}(M_{k_{j'}}),$
if $j\not=j'$ and $j,j'\in S_i$ for each $i\in \{0,{\green{1,\cdots}},\bar{n}\}.$
In other words, $\bar{\beta}|_{\bar{B}_i}$
is a sum of mutually orthogonal c.p.c.~order zero isometries.}}
{\blue{Hence}}  $\bar{\beta}|_{\bar{B}_i}$ is still {\blue{a}}
c.p.c.~order zero isometry.  {\blue{Therefore}} (3) holds.

For any $x\in A$, by \eqref{f4-11-1}, we have
$
\bar{\beta}\circ\bar{\alpha}(x)
=
\sum_{j=1}^m
\bar{\beta}^{(j)}(\bar{\alpha}^{(j)}(x))
=\beta\circ\alpha(x).
$ Thus (1) holds.

Let $a\in A^1_+$. 
{\blue{Recall}} that $\bar{\beta}^{(j)}$
{\blue{is a c.p.c~order zero}}
isometry ($j=1,{\green{2,}}\cdots,m$).
We have
%
$$
\|\bar{\alpha}^{(j)}(a)\|
=
\|\bar{\beta}^{(j)}(\bar{\alpha}^{(j)}(a))\|
\leq
\|\sum_{j=1}^m
\bar{\beta}^{(j)}(\bar{\alpha}^{(j)}(a))
\|
\overset{\eqref{f4-11-1}}{=}
\|\beta\circ\alpha(a)\|
\leq
\|\beta\circ\alpha\|.
$$
Thus $\|\bar{\alpha}(a)\|
=\max\{\|\bar{\alpha}^{(j)}(a)\|:j=1,{\green{2,}}\cdots,m\}\leq \|\beta\circ\alpha\|$,
which implies $\|\bar{\alpha}\|\leq \|\beta\circ\alpha\|$.
Also note that
\beq
\|\bar{\alpha}\|
&=&
\max\{\|\bar{\alpha}^{(j)}\|:j=1,{\green{2,}}\cdots,m\}
=\max\{\|\alpha^{(j)}\|\cdot\|\beta^{(j)}\|:j=1,{\green{2,}}\cdots,m\}
\nonumber\\
&\leq& \max\{\|\alpha^{(j)}\|:j=1,{\green{2,}}\cdots,m\}
\cdot|||\beta|||
{{\ \leq\ }}\|\alpha\|\cdot|||\beta|||.
\eneq
So (2) holds.

\end{proof}

\begin{df}[{\blue{\cite[Definition 2.2]{TW2014}
and \cite[Definition 2.1]{WZ2010}}}]
\label{definition-of-nuclear-dimension}
\index{Nuclear dimension}
Let $A$ {\blue{and}} $B$
be $C^*$-algebras {\blue{and}} let $h: A\rightarrow B$ be a *-homomorphism.
{\blue{Recall that}} $h$ has {\blue{\it nuclear dimension at most $n$,}}
and denote by $\dimnuc h\leq n$, if the following  {\blue{conditions}} hold:

For any finite subset
$\calF\subset A$ and any
$\epsilon>0,$
there exist finite dimensional
$C^*$-algebras $F_0,\cdots,F_n$
{{and,}} c.p.~maps $\phi:A\rightarrow F_0\oplus\cdots \oplus F_n$
and $\psi:F_0\oplus\cdots \oplus F_n \rightarrow B$ {{such}} that

(1) $\psi\circ\phi(x)\approx_{\epsilon}h(x)$ {\blue{for all}} $x\in \calF$,

(2) $\|\phi\|\leq 1$, {{and}}

(3) $\psi|_{F_i}$  {\blue{is a
c.p.c.~order zero map, $i=0,1,\cdots, n.$}}

\noindent
{{We}} say $A$ has nuclear dimension at most $n$,
and denoted by $\dimnuc A\leq n$,
if $\dimnuc \mathrm{id}_A\leq n.$
\end{df}


{\blue{The following may be known to experts.}}

\begin{prop}
\label{equivalent-definition-on-nuclear-dimension}

Let $h:A\rightarrow B$ be a *-homomorphism of $C^*$-algebras {\blue{and}}
{\blue{$n\in\mathbb{N}\cup\{0\}.$}}
{\blue{Then}} $\dimnuc h\leq n$ if and only if the following  {\blue{holds}}:
For any finite subset $\calF\subset A$ and any $\epsilon>0$,
there {\blue{exist}} a finite dimensional $C^*$-algebra $F$,
c.p.~maps $\phi: A\rightarrow F$ {\green{and}}
$\psi:F\rightarrow B$ {{such}} that

(1) $\psi\circ\phi(x)\approx_{\epsilon} h(x)$ {\blue{for all}} $x\in \calF,$ {{and}}

(2) $\psi$ is $n$-decomposable {\blue{(see
{{Definition}} \ref{defn-n-decomposable-c.p.-map}).}}

\end{prop}

\begin{proof}

The ``only if'' part is trivial.
For the ``if'' part,
let $\calF\subset A$ be a finite subset {\blue{and let}}   $\epsilon>0$.
Set $\delta:=\frac{\epsilon}{3+\max\{\|x\|:x\in\calF\}}.$
Choose
 $e\in A^1_+$ such that
$exe\approx_{\delta}x$ for all $x\in \calF.$

By our assumption,
there {\blue{exist}} 
a finite dimensional $C^*$-algebra $F,$
{\blue{c.}}p.~maps $\phi: A\rightarrow F$ {\blue{and}}
$\psi:F\rightarrow B$ {{such}} that

($1'$) $\psi\circ\phi(x)
\approx_{\delta} {\blue{h(x)}}$ {\blue{for all}} $x\in{\blue{\{eye:y\in\calF\}\cup\{e^2\}}}$, {\blue{and}}

($2'$) $\psi$ is $n$-decomposable.

Define a c.p.~map $\tilde{\phi}: A\rightarrow F$
{{by}}
$x\mapsto \frac{1}{1+\delta}\phi(exe)$
{{for all $x\in A$.}}
Then, for any $a\in A_1^+$,
$$\|\psi\circ\tilde{\phi}(a)\|
=\frac{1}{1+\delta}\|\psi\circ\phi(eae)\|
\leq \frac{1}{1+\delta}\|\psi\circ\phi(e^2)\|
\overset{\mathrm{(by\ (1'))}}{\leq}
\frac{1}{1+\delta}(\|h(e^2)\|+\delta)
\leq 1.
$$
It follows
\beq
\|\psi\circ\tilde{\phi}\|\leq 1.
\eneq
{{By}} Proposition \ref{normalization-of-c.p.-dividable-map},
there exist a finite dimensional $C^*$-algebra $\bar{F}$,
and c.p.~maps $\bar{\phi}: A\rightarrow \bar{F}$
and $\bar{\psi}: \bar{F}\rightarrow B$,
such that

($1''$) $\bar{\psi}\circ\bar{\phi}=\psi\circ\tilde{\phi}$,

($2''$) $\|\bar{\phi}\|\leq \|\psi\circ\tilde{\phi}\|\leq 1$, {\blue{and}}

($3''$) $\bar{\psi}$ is a piecewise contractive $n$-decomposable c.p.~map.

\noindent
{\blue{Then,}} by $(2'')$ and $(3'')$,
$\bar{\psi},\,\bar{\phi}$ {\blue{and}} $\bar{F}$ {\blue{satisfy}} (2) {\blue{and}} (3) of Definition
\ref{definition-of-nuclear-dimension}.
For all $x\in \calF$,
$$
\bar{\psi}\circ\bar{\phi}(x)
\overset{\mathrm{(by\ (1''))}}{=}\psi\circ\tilde{\phi}(x)
=\frac{1}{1+\delta}\psi\circ\phi(exe)
\overset{\mathrm{(by\ (1'))}}{\approx_{\frac{\delta}{1+\delta}}}
\frac{1}{1+\delta}h(exe)
\approx_{\frac{\delta}{1+\delta}}
\frac{1}{1+\delta}h(x)\approx_{\frac{\dt}{1+\dt}\|x\|}h(x).
$$
By the choice of $\delta,$
we have $h(x)\approx_{\epsilon}\bar{\psi}\circ\bar{\phi}(x)$.
Then,
by Definition \ref{definition-of-nuclear-dimension},
we have $\dimnuc h\leq n$.

\end{proof}


\begin{cor}
\label{equivalent-definition-on-nuclear-dimension-1}

Let $A$ be a \CA\, and let
$n\in\mathbb{N}.$
Then $\dimnuc A\leq n$ if and only if the following {\blue{conditions}} hold:
For any finite subset $\calF\subset A$ and any $\epsilon>0$,
there exist a finite dimensional $C^*$-algebra $F$
and c.p.~maps $\phi: A\rightarrow F$ {\blue{and}} $\psi:F\rightarrow A$
{\blue{such}} that

(1) $\psi\circ\phi(x)\approx_{\epsilon} x$ {\blue{for all $x\in \calF$}} {\blue{and}}

(2) $\psi$ is $n$-decomposable.

\end{cor}

\begin{prop}

\label{c.p.c.-almost-multiplicative-3}

Let $A,B$, {\blue{and}} $C$ be $C^*$-algebras,
$\{e_\lambda\}_{\lambda\in\Lambda}$ be an approximate identity of $A,$
$\phi: A\rightarrow B$,
$\psi: B\rightarrow C$
be c.p.~maps with $\|\psi\circ\phi\|\leq 1$,  and let
$\epsilon>0$.
{\blue{Suppose that}} $a\in A_+$
such that $\|\psi\circ \phi(a^2)-\psi\circ \phi(a)^2\|\leq \epsilon$ {\blue{and}}
$b\in B$  {\blue{such that $b$}} commutes with
{{$\{\phi(e_\lambda)\}_{\lambda\in\Lambda}$.}}
{\blue{Then}}
\beq\nonumber
&&\limsup_{\lambda}\|\psi(\phi(a)b)-\psi(\phi(a))\psi(\phi(e_\lambda)b)\|
\leq \epsilon^{1/2}\|b\|\andeqn\\
&&\limsup_{\lambda}\|\psi(b\phi(a))-\psi(\phi(e_\lambda)b)\psi(\phi(a)) \|
\leq \epsilon^{1/2}\|b\|.
\eneq



\end{prop}

\begin{proof}

We will show {\blue{that}} the first inequality holds.
{\blue{The}} second  {\blue{one}}
holds
by taking conjugate of the first one.
%
{\blue{
{{Put}} $M:=\|a\|+\|a\|^2$.}}
Let $\theta>0$.
Choose $\delta>0$ be such that
\beq\label{f4-22-3}
\left(\delta(2\|\psi\|+\|\psi\|\|\phi\|+1)
+\sqrt{(2M+1)(\|\psi\|+1)\delta +\epsilon}\right) \|b\|
<\epsilon^{1/2}\|b\|+\theta.
\eneq
{{Let $\lambda_1\in\Lambda$}}
such that, for any $\lambda \geq \lambda_1$,
{{any}} $x\in\{a,a^2\}$,
\beq\label{49-n-1}
e_{\lambda}^{1/2}x e_{\lambda}^{1/2}\approx_{\delta}x\andeqn
\psi\circ \phi(e_{\lambda}^{1/2}x e_{\lambda}^{1/2})\approx_{\delta}\psi\circ \phi(x).
\eneq
Fix $\lambda\geq \lambda_1$.
Note, for any $x\in A_+$, we have $0\le \phi(e_\lambda^{1/2} xe^{1/2}_\lambda)\le \|x\|\phi(e_\lambda).$
{\blue{Thus}} $\phi(e_\lambda^{1/2} xe^{1/2}_\lambda)\in \Her_B(\phi(e_\lambda))$.
Note {\blue{that}} $\{{\blue{E_{n}}}:=(1/n+\phi(e_{\lambda}))^{-1}\phi(e_{\lambda})\}_{{n\in\N}}$
{\blue{forms}} an approximate identity for $\Her_B(\phi(e_\lambda)).$
It follows {{that}}
$
\lim_{n\to\infty}
\|\phi(e_{\lambda}^{1/2}x e_{\lambda}^{1/2})-E_{\blue{n}}^{1/2}\phi(e_{\lambda}^{1/2}x e_{\lambda}^{1/2})
E_{\blue{n}}^{1/2}\|=0.
$
{\blue{Therefore}} there exists $\eta>0$
such that, for $x\in\{a,a^2\}$, 
\beq\label{49-n0}
\phi(e_{\lambda}^{1/2}x e_{\lambda}^{1/2})
\approx_{\delta}
\phi(e_{\lambda})^{1/2}(\eta+\phi(e_{\lambda}))^{-1/2}
\cdot\phi(e_{\lambda}^{1/2}x e_{\lambda}^{1/2})
\cdot(\eta+\phi(e_{\lambda}))^{-1/2}\phi(e_{\lambda})^{1/2}.
\eneq

\noindent
{{Define}} the following c.p.~maps:
{\blue{\beq
&&\tilde{\phi}:A\rightarrow B,
\quad
x
\mapsto
(\eta+\phi(e_{\lambda}))^{-1/2}\cdot
\phi(e_{\lambda}^{1/2}xe_{\lambda}^{1/2})
\cdot(\eta+\phi(e_{\lambda}))^{-1/2}\andeqn\\
&&\tilde{\psi}:
B\rightarrow C,
\quad
x\mapsto
\psi(\phi(e_{\lambda})^{1/2}x\phi(e_{\lambda})^{1/2}).
\eneq}}
\noindent
{{We}} claim that $\|\tilde{\phi}\|\leq 1$.
Indeed, for any $x\in A_+^1$,
\beq\nonumber
&&\hspace{-0.2in}\|\tilde{\phi}(x)\|
=\|(\eta+\phi(e_{\lambda}))^{-1/2}\cdot
\phi(e_{\lambda}^{1/2}xe_{\lambda}^{1/2})
\cdot(\eta+\phi(e_{\lambda}))^{-1/2}\|\\
&&\leq
\|(\eta+\phi(e_{\lambda}))^{-1/2}\cdot
\phi(e_{\lambda})
\cdot(\eta+\phi(e_{\lambda}))^{-1/2}\|
=\|\phi(e_{\lambda})
\cdot(\eta+\phi(e_{\lambda}))^{-1}\|\leq 1.
\eneq	
{\blue{We}} {\blue{also}} claim that $\|\tilde{\psi}\|\leq 1$.
Indeed, for any
{{$x\in B_+^1$,}}
$\|\tilde{\psi}(x)\|=\|\psi(\phi(e_{\lambda})^{1/2}x\phi(e_{\lambda})^{1/2})\|
\leq \|\psi(\phi(e_{\lambda}))\|\leq\|\psi\circ\phi\|\|e_\lambda\|\leq 1.$
{\blue{Thus}} $\|\tilde{\psi}\|\leq 1$.

Note that, by \eqref{49-n0} and \eqref{49-n-1},
for {\blue{$x\in \{a,a^2\},$}}
we have
\beq\label{49-n1}
\tilde{\psi}\circ \tilde{\phi}(x)
\approx_{\|\psi\|\delta}
\psi\circ \phi(e_{\lambda}^{1/2} x e_{\lambda}^{1/2})
\approx_{\delta}
\psi\circ \phi(x).
\eneq

\noindent
{{Then}} we have,  {\blue{applying \eqref{49-n1}}},
\beq
\tilde{\psi}\circ \tilde{\phi}(a)^2\,\,
&\approx_{M(\|\psi\|+1)\delta}&
\tilde{\psi}\circ \tilde{\phi}(a)\cdot\psi\circ\phi(a)
\nonumber\\
\mbox{(by \eqref{49-n1} and $\|\psi\circ\phi\|\leq 1$)}\hspace{0.5in}
&\approx_{M(\|\psi\|+1)\delta}&
\psi\circ\phi(a)^2
{\blue{\approx_{\ep}\psi\circ\phi(a^2)}}
\nonumber\\
\mbox{(by \eqref{49-n1})}\hspace{0.5in}
&\approx_{(\|\psi\|+1)\delta}&
\tilde{\psi}\circ \tilde{\phi}(a^2).
\label{f-4-21-1}
\eneq

\noindent
{{Then}}
\beq
\psi(\phi(a))\cdot\psi(\phi(e_\lambda)b)
&=&
\psi\circ \phi(a)\psi(\phi(e_\lambda)^{1/2}b\phi(e_\lambda)^{1/2})
\nonumber\\
&=&
\psi\circ \phi(a)\tilde{\psi}(b)
\nonumber\\
\mbox{(by \eqref{49-n1})}&\approx_{(\|\psi\|+1)\delta\|b\|}&
\tilde{\psi}\circ \tilde{\phi}(a)\tilde{\psi}(b)
\nonumber\\
\mbox{(by \eqref{f-4-21-1} and Lemma \ref{c.p.c.-almost-multiplicative})}
&\approx_{\|b\|\sqrt{(2M+1)(\|\psi\|+1)\delta +\epsilon}}&
\tilde{\psi}(\tilde{\phi}(a)b)
\nonumber\\
&=&
\psi(\phi(e_{\lambda})^{1/2}\tilde{\phi}(a)b\phi(e_{\lambda})^{1/2})
\nonumber\\
&=&
\psi(\phi(e_{\lambda})^{1/2}\tilde{\phi}(a)\phi(e_{\lambda})^{1/2}b)
\nonumber\\
\mbox{(by \eqref{49-n0})}
&\approx_{\|\psi\|\|b\|\delta}&
\psi\left(\phi(e_{\lambda}^{1/2}a e_{\lambda}^{1/2})b\right)
\nonumber\\
\mbox{(by \eqref{49-n-1})}
&\approx_{\|\psi\|\|\phi\|\|b\|\delta}&
\psi(\phi(a)b).
\eneq

\noindent
{{By}} \eqref{f4-22-3},
we have
$\psi(\phi(a))\cdot\psi(\phi(e_\lambda)b)
\approx_{\epsilon^{1/2}+\theta}\psi(\phi(a)b)$.
Thus
$$
\limsup_{\lambda}\|\psi(\phi(a)b)-\psi(\phi(a))\psi(\phi(e_\lambda)b)\|
\leq (\epsilon^{1/2}+\theta)\|b\|.
$$

\noindent
Let $\theta\to 0.$ {\blue{The proposition then}} follows.

\end{proof}



\begin{thm}
\label{equivalent-definition-on-nuclear-dimension-second}

Let $A$ and $B$ be \CAs, $h: A\rightarrow B$ be  a *-homomorphism, and
let $n\in\mathbb{N}\cup \{0\}$.
Then $\dimnuc h\leq n$ if and only if the following {\blue{condition holds}}:
For any finite subset $\mathcal{G}\subset A_+$ {\blue{and}} any $\epsilon>0$,
there {\blue{exist}} a \CA\ $C$, a finite subset $\tilde{\mathcal{G}}\subset C_+$,
a finite dimensional $C^*$-algebra $F$
{{and,}} c.p.~maps $\phi: C\rightarrow F$
{{and}} $\psi:F\rightarrow B$
such that

(1) $h(\mathcal{G})\subset_\epsilon \psi\circ\phi(\tilde{\mathcal{G}})$,

(2) $\psi$ is $n$-decomposable, {\blue{and}}

(3) $\psi\circ\phi(xy)\approx_\epsilon \psi\circ\phi(x)\cdot\psi\circ\phi(y)$ {\blue{for all}}
$x,y\in \tilde{\mathcal{G}}\cup(\tilde{\mathcal{G}}\cdot\tilde{\mathcal{G}})
$.

\end{thm}

\begin{proof}

For the ``only if'' part,
let $C=A$,
let $\tilde{\mathcal{G}}=\mathcal{G}$ and let
$M:=\{\|z\|: z\in
\tilde{\mathcal{G}}\cup(\tilde{\mathcal{G}}\cdot\tilde{\mathcal{G}})\}.$
Put $\theta:= \min\{1, \frac{\epsilon}{2(M+1)} \}$.
Since $\dimnuc h\leq n$,
by Definition \ref{definition-of-nuclear-dimension},
we can choose
{{a finite dimensional \CA\ $F$ and,
c.p.c.~maps $\phi: A\to F$ and $\psi: F\to B$ such that}}

{{$(1')$}} $h(x)\approx_{\theta}\psi\circ\phi(x)$ {\blue{for all}}
$x\in
\mathcal{G}\cup(\mathcal{G}\cdot \mathcal{G})$ {{and}}

{{$(2')$}} $\psi$ is $n$-decomposable.

\noindent
{{Then,}} by {{$(1')$}}, we have

{{$(3')$}} $\psi\circ\phi(xy)\approx_{\theta} h(xy)
=h(x)h(y)
\approx_{(2M+\theta)\theta}\psi\circ\phi(x)\psi\circ\phi(y)$
{\blue{for all}} $y\in
\mathcal{G}\cup(\mathcal{G}\cdot\mathcal{G})$.

\noindent
{{Note}} that,
by the choice of $\theta$, we have $(2M+1+\theta)\theta\leq \epsilon.$
Thus the ``only if'' part holds.


For the ``if'' part, let
{{$\mathcal{G}\subset A_{+}^1$ be a finite subset}}
and let $\epsilon>0.$
There exists $\delta_1>0$ such that, for all
$x\in A_+^1$ and for all $y\in A$ with $\|y\|\leq 2,$
if $yx\approx_{\delta_1}xy$,
then $x^{1/2}yx^{1/2}\approx_{\epsilon/4(n+1)}yx$.
Choose $\delta:=\min\{\frac{1}{100},
(\frac{\epsilon}{32(n+3)})^2,(\frac{\delta_1}{12})^2\}$.
Let $e\in A^1_+$ {\blue{be}} such that
\beq\label{f4-22thm-3}
exe\approx_{\delta}x\rforal
 x\in \mathcal{G}\cup(\mathcal{G}\cdot\mathcal{G}).
\eneq

\noindent
{{By}} our assumption,
there {\blue{exist}} a \CA\ $C$, a finite subset $\tilde{\mathcal{G}}\subset C_+$,
and a finite dimensional $C^*$-algebra $\tilde{F}$
{{and,}} c.p.~maps $\tilde{\phi}: C\rightarrow \tilde{F}$
{{and}}
$\tilde{\psi}: \tilde{F}\rightarrow B$,
such that

{{$(1'')$}} $h(\mathcal{G}\cup\{e\})
\subset_\dt \tilde{\psi}\circ\tilde{\phi}(\tilde{\mathcal{G}})$,

{{$(2'')$}} $\tilde{\psi}$ is $n$-decomposable, {\blue{and}}

{{$(3'')$}} $\tilde{\psi}\circ\tilde{\phi}(xy)\approx_\delta \tilde{\psi}\circ\tilde{\phi}(x)\cdot\tilde{\psi}\circ\tilde{\phi}(y)$
{\blue{for all}} $x,y\in \tilde{\mathcal{G}}\cup(\tilde{\mathcal{G}}\cdot\tilde{\mathcal{G}})
$.

\noindent
{{By}} Proposition \ref{normalization-of-c.p.-dividable-map},
there {\blue{exist}} $\bar{n}\leq n$, 
{{a}} finite dimensional \CA\ $F=F_0\oplus\cdots\oplus F_{\bar{n}}$,
and c.p.~maps $\phi: C\rightarrow F$ and
$\psi: F\rightarrow B$, such that

{{$(1''')$}} $\psi\circ\phi=\tilde{\psi}\circ\tilde{\phi}$ {\blue{and}}

{{$(2''')$}} $\psi|_{F_i}$ {\blue{is a}}  c.p.c.~order zero {\blue{isometry,}}
$i=0,{\green{1,}}\cdots,\bar{n}$.

\noindent
{{By}} {{$(1'')$}},
for each $x\in h(\mathcal{G}\cup\{e\}),$
there exists $\af(x)\in  \tilde{\mathcal{G}}$ such that
$x\approx_{\delta}\tilde{\psi}\circ\tilde{\phi}(\alpha(x)).$
%
Then, by {{$(1''')$}}, we have
\beq\label{f4-22thm-2}
\psi\circ\phi(\alpha(x))\approx_{\delta}x
{\blue{\rforal}} x\in h(\mathcal{G}\cup\{e\}).
\eneq
%
{\blue{Note that ${\cal G}\subset A_+^1.$ Then, by   \eqref{f4-22thm-2},
\beq\label{516-1229-1}
\|\psi\circ \phi(\af(y))\|\le 1+\dt\rforal y\in h({\cal G}\cup \{e\}).
\eneq}}
Combining  {{$(3'')$}},   {{$(1''')$}}, \eqref{f4-22thm-2} {\blue{and \eqref{516-1229-1}}},  for any $ x,y\in h(\mathcal{G}),$ we have
\beq
\label{f4-22-thm-11}
\psi\circ\phi(\alpha(x)\alpha(y))\approx_{\delta}
\psi\circ\phi(\alpha(x))\psi\circ\phi(\alpha(y))
\approx_{(1+\delta)\delta}{\blue{x \cdot}} \psi\circ\phi(\alpha(y))
\approx_\dt xy.
\eneq
In particular,
\beq
\label{f4-22thm-12}
\|\psi\circ\phi(\alpha(x)\alpha(y))\|\leq 1+(3+\delta)\delta
{\blue{\rforal}} x,y\in h(\mathcal{G}).
\eneq
Define a c.p.~map
$\bar{\phi}: C\rightarrow F$ {\blue{by}}
$x\mapsto \frac{1}{1+(3+\delta)\delta}\phi(\alpha(h(e))x\alpha(h(e))).$
Then, for any $x\in C^1_+$,
by \eqref{f4-22thm-12}, we have
$$
\|\psi\circ\bar{\phi}(x)\|
=\frac{\|\psi(\phi(\alpha(h(e))x\alpha(h(e))))\|}{1+(3+\delta)\delta}
\leq
\frac{\|\psi(\phi(\alpha(h(e))^2))\|}{1+(3+\delta)\delta}
\leq 1.
$$
Thus
\beq
\label{f4-22thm-8}
\|\psi\circ\bar{\phi}\|\leq 1.
\eneq
Let $x\in h(\mathcal{G}).$  Then
\beq
\psi\circ\bar{\phi}(\alpha(x))
&=&
\frac{1}{1+(3+\delta)\delta}\psi\circ\phi(\alpha(h(e))\alpha(x)\alpha(h(e)))
\nonumber\\
\mbox{(by {{$(3'')$}} and   {{$(1''')$}})}&\approx_{\delta}&
\frac{1}{1+(3+\delta)\delta}\psi\circ\phi(\alpha(h(e)))\psi\circ\phi(\alpha(x)\alpha(h(e)))
\nonumber\\
\mbox{(by  {{$(3'')$}}, {{$(1''')$}} and  {\blue{\eqref{516-1229-1}}})}
&\approx_{\delta}&
\frac{1}{1+(3+\delta)\delta}\psi\circ\phi(\alpha(h(e)))\psi\circ\phi(\alpha(x))
\psi\circ\phi(\alpha(h(e)))
\nonumber\\
\mbox{(by \eqref{f4-22thm-2})}
&\approx_{\frac{(1+\dt)^2\delta+(1+\dt)\delta+\dt}{1+(3+\dt)\dt}}&
\frac{h(e)xh(e)}{1+(3+\delta)\delta}
\nonumber\\
\mbox{(by \eqref{f4-22thm-3})}
&\approx_{\delta}&
\frac{x}{1+(3+\delta)\delta}
\approx_{4\delta}
x.
\label{f4-22thm-5}
\eneq

\noindent
{{Also,}} for $x\in h(\mathcal{G})$, we have
\beq
\psi\circ\bar{\phi}(\alpha(x)^2)
&=&
\frac{1}{1+(3+\delta)\delta}\psi\circ\phi(\alpha(h(e))\alpha(x)^2\alpha(h(e)))
\nonumber\\
\mbox{(by {{$(3'')$}} and {{$(1''')$}})}&\approx_{\delta}&
\frac{1}{1+(3+\delta)\delta}
\psi\circ\phi(\alpha(h(e))\alpha(x))\cdot\psi\circ\phi(\alpha(x)\alpha(h(e)))
\nonumber\\
\mbox{(by {{$(3'')$}}, {{$(1''')$}}, \eqref{f4-22thm-12}, {\blue{\eqref{516-1229-1}}})}&\approx_{2\delta}
&
\frac{1}{1+(3+\delta)\delta}
\psi\circ\phi(\alpha(h(e)))
\cdot
\psi\circ\phi(\alpha(x))^2
\cdot
\psi\circ\phi(\alpha(h(e)))
\nonumber\\
\mbox{(by \eqref{f4-22thm-2})}
&\approx_{4\delta(1+\delta)}&
\frac{h(e)x^2h(e)}{1+(3+\delta)\delta}
\nonumber\\
\mbox{(by \eqref{f4-22thm-3})}
&\approx_{\delta}&
\frac{x^2}{1+(3+\delta)\delta}
\approx_{4\delta}
x^2.
\label{f4-22thm-4}
\eneq

\noindent
{{By}} \eqref{f4-22thm-4} and \eqref{f4-22thm-5},
we have
\beq
\label{f4-22-thm-6}
\psi\circ\bar{\phi}(\alpha(x)^2)
\approx_{12\delta(1+\delta)}
x^2
\approx_{10\delta(2+\delta)}
\psi\circ\bar{\phi}(\alpha(x))^2
{\blue{\rforal}} x\in h(\mathcal{G}).
\eneq

\noindent
{{Let}} $p_i$ be the unit of $F_i,$ {\blue{$i=0,1,\cdots, {\bar n}.$}}
{{T}}hen  {\blue{each}} $p_i$
is  {\blue{a central projection}} of $F.$
%
We now apply Proposition \ref{c.p.c.-almost-multiplicative-3}.
Recall $\bar{\phi}: C\to F$ and $\psi: F\to B$ are c.p.~maps such that
$\|\psi\circ {\bar \phi}\|\le 1$ (see \eqref{f4-22thm-8}).
Thus, by \eqref{f4-22-thm-6} and
Proposition \ref{c.p.c.-almost-multiplicative-3},
there exists a positive element $c\in C_+^1$
such that, for $x\in h(\mathcal{G})$,
the following hold (note, $(\delta(32+{{22}}\delta))^{1/2}<6\delta^{1/2}$):
\beq
\label{f4-22thm-7}
\psi(\bar{\phi}(\alpha(x)))\cdot\psi(\bar{\phi}(c)p_i)
&\approx_{6\delta^{1/2}}&
\psi(\bar{\phi}(\alpha(x))p_i)
\nonumber\\
&=&
\psi(p_i\bar{\phi}(\alpha(x)))
\approx_{6\delta^{1/2}}
\psi(\bar{\phi}(c)p_i) \cdot \psi(\bar{\phi}(\alpha(x))).
\eneq
Note that $\psi(\bar{\phi}(c)p_i)=
\psi(\bar{\phi}(c)^{1/2}p_i\bar{\phi}(c)^{1/2})$
is a positive element,
and, by \eqref{f4-22thm-8},
\beq
\label{f4-22thm-13}
\|\psi(\bar{\phi}(c)p_i)\|=
\|\psi(\bar{\phi}(c)^{1/2}p_i\bar{\phi}(c)^{1/2})\|
\leq \|\psi(\bar{\phi}(c))\|\leq \|c\|\leq 1.
\eneq
Also note that
$\|\psi(\bar{\phi}(\alpha(x)))\|\leq 1+\delta\leq 2$
{\blue{for all}} $x\in h(\mathcal{G})$.
By \eqref{f4-22thm-7},
$\psi(\bar{\phi}(c)p_i)$ approximately commutes with
${\blue{\{\psi(\bar{\phi}(\alpha(x))): x\in h({\cal G})\}}}$
{\blue{within $12\dt^{1/2},$}}
{\blue{and,}} by the choice of $\delta$ and $\delta_1$,
we have
\beq
\label{f4-22thm-10}
\hspace{-0.4in}\psi(\bar{\phi}(c)p_i)^{1/2}\cdot
\psi(\bar{\phi}(\alpha(x)))
\cdot\psi(\bar{\phi}(c)p_i)^{1/2}
&\approx_{\frac{\epsilon}{4(n+1)}}&
\psi(\bar{\phi}(\alpha(x)))\cdot\psi(\bar{\phi}(c)p_i)
\nonumber\\
&\approx_{6\delta^{1/2}}&
\psi(\bar{\phi}(\alpha(x))p_i)
{\blue{\rforal  x\in h({\cal G}).}}
\eneq

\noindent
{{By}} {{$(2''')$}}
and by Proposition \ref{inverse-of-finite-dim-injective-order-zero-map},
there exists c.p.c.~maps
$\beta_i: B\rightarrow  F_i$
such that
\beq
\label{f4-22thm-9}
\beta_i\circ\psi|_{F_i}=\mathrm{id}_{F_i}, \,\,\,
{\blue{i=0, {\blue{1,}} {{\cdots}},{\bar n}.}}
\eneq

\noindent
{{Define}} c.p.~maps $(i=0, 1, {\blue{\cdots, \bar{n}}})$
$\gamma_i: A \rightarrow F_i$  {\blue{by}}
$x\mapsto
\beta_i\left(\psi(\bar{\phi}(c)p_i)^{1/2}\cdot
h(x)
\cdot\psi(\bar{\phi}(c)p_i)^{1/2}\right)$
and define c.p.~map
%
$\gamma: A
\rightarrow
F=F_0\oplus\cdots\oplus F_{\bar{n}}$ {\blue{by}}
$x\mapsto
(\gamma_0(x),\cdots,\gamma_{\bar{n}}(x)).$
For $x\in\mathcal{G}$,
\beq
\psi\circ\gamma(x)
=
\sum_{i=0}^{\bar{n}}
\psi \circ\gamma_i(x)
&=&
\sum_{i=0}^{\bar{n}}
\psi \circ\beta_i\left(\psi(\bar{\phi}(c)p_i)^{1/2}\cdot
h(x)
\cdot\psi(\bar{\phi}(c)p_i)^{1/2}\right)
\nonumber\\
\mbox{(By \eqref{f4-22thm-5}, \eqref{f4-22thm-13})}
&\approx_{10(n+1)\delta}&
\sum_{i=0}^{\bar{n}}
\psi \circ\beta_i\left(\psi(\bar{\phi}(c)p_i)^{1/2}\cdot
\psi(\bar{\phi}(\alpha(h(x))))
\cdot\psi(\bar{\phi}(c)p_i)^{1/2}\right)
\nonumber\\
\mbox{(By \eqref{f4-22thm-10})}
&\approx_{\frac{\epsilon}{4}+6(n+1)\delta^{\frac{1}{2}}}&
\sum_{i=0}^{\bar{n}}
\psi \circ\beta_i\left(
\psi(\bar{\phi}(\alpha(h(x)))p_i)
\right)
\nonumber\\
\mbox{(\eqref{f4-22thm-9}, $\bar{\phi}(\alpha(h(x)))p_i\in F_i$)}
&=&
\sum_{i=0}^{\bar{n}}
\psi (\bar{\phi}(\alpha(h(x)))p_i)
{\blue{=\psi(\bar{\phi}(\alpha(h(x))))}}
\nonumber\\
\mbox{(By \eqref{f4-22thm-5})}
&\approx_{10\delta}&
h(x).
\eneq

\noindent
{{Note,}} by the choice of $\delta\ (\leq (\frac{\epsilon}{32(n+3)})^2)$,
we have $10(n+1)\delta+\frac{\epsilon}{4}+6(n+1)\delta^{\frac{1}{2}}
+10\delta\leq \epsilon$.
Thus there {\blue{exist a}} c.p.~map $\gamma: A\rightarrow F$ {\blue{and a}}
c.p.~$n$-decomposable map $\psi: F\rightarrow A$
such that $h(x)\approx_{\epsilon}\psi\circ\gamma(x)$ {\blue{for all}}
$x\in\mathcal{G}$.
{\blue{Finally,}} by Proposition \ref{equivalent-definition-on-nuclear-dimension},
$\dimnuc h\leq n$.



\end{proof}



\begin{prop}
\label{homomorphism-dimnuc-equal-embedding-dimnuc}
Let $A$ {\blue{and}} $B$ be \CAs, $h: A\rightarrow B$ be a *-homomorphism {\blue{and}}
let $\iota: h(A)\hookrightarrow B$ be the embedding.
Then $\dimnuc h=\dimnuc \iota$.

\end{prop}

\begin{proof}

First, we {\blue{note}} $\dimnuc h=\dimnuc (\iota\circ h) \leq \dimnuc \iota$.


Next, if $\dimnuc h=\infty$, then we are  done.
Hence we may assume that $\dimnuc h= n$ for some $n\in{{\N\cup\{0\}}}$.
Let $\mathcal{G}\subset h(A)_+$ be a finite subset  and let $\epsilon>0$.
Then there exists a finite subset
$
\tilde{\mathcal{G}}\subset A_+
$
such that
\beq
\label{f4-24thm-1}
\iota(\mathcal{G})=\mathcal{G}=h(\tilde{\mathcal{G}}).
\eneq
{\blue{Choose}} $M=\max\{\|x\|+1: x\in\tilde{\mathcal{G}}\}$ and
 $\delta:=\frac{\min\{\epsilon,1\}}{2(M+1)^2}$.
Since $\dimnuc h\leq n$,
there {\blue{exist}} a finite dimensional \CA\ {\blue{$F,$}}
and c.p.~maps $\phi: A\rightarrow F$ {\blue{and}} $\psi: F\rightarrow B$
{\blue{such}} that

(1) $\psi\circ\phi(x)\approx_{\delta}h(x)=\iota(h(x))$ {\blue{for all}}
$x\in\tilde{\mathcal{G}}\cup(\tilde{\mathcal{G}}\cdot \tilde{\mathcal{G}})
\cup(\tilde{\mathcal{G}}\cdot \tilde{\mathcal{G}}\cdot \tilde{\mathcal{G}})
\cup(\tilde{\mathcal{G}}\cdot \tilde{\mathcal{G}}\cdot \tilde{\mathcal{G}}\cdot \tilde{\mathcal{G}})$, {\blue{and}}

(2) $\psi$ is $n$-decomposable.

\noindent
{{Then}} \eqref{f4-24thm-1} and (1) show
\beq
\label{f4-24thm-3}
\iota(\mathcal{G})\subset_{\epsilon}
\psi\circ\phi(\tilde{\mathcal{G}}).
\eneq
{{By (1),
for all $x\in \tilde{\mathcal{G}}\cup(\tilde{\mathcal{G}}\cdot \tilde{\mathcal{G}})$,
we have}}
\beq\label{418-n+1}
\|\psi\circ \phi(x)\|\le \dt+\|h({{x}})\|\le \dt+M^2.
\eneq

\noindent
{{Therefore,}}
using (1) and \eqref{418-n+1}, we have
$$
\psi\circ\phi(xy)
\approx_{\delta}
h(xy)=h(x)h(y)
\approx_{(M^2+\dt)\delta}
h(x)\psi\circ\phi(y)
\approx_{\delta(M^2+\delta)}
\psi\circ\phi(x)\psi\circ\phi(y).
$$
Then, by the choice of $\delta,$ we have
\beq
\psi\circ\phi(xy)\approx_{\epsilon}\psi\circ\phi(x)\psi\circ\phi(y)
\label{f4-24thm-2}
{\blue{\rforal}}
x,y\in \tilde{\mathcal{G}}\cup(\tilde{\mathcal{G}}\cdot \tilde{\mathcal{G}}).
\eneq

\noindent
{{Then}} \eqref{f4-24thm-3}, (2), together with \eqref{f4-24thm-2},
show that (with {{$A$}} in place of $C$),
the conditions of
Theorem \ref{equivalent-definition-on-nuclear-dimension-second}
are satisfied.
{\blue{Therefore}} we have $\dimnuc \iota\leq n=\dimnuc h.$
%

\end{proof}

The following corollary shows that
the image of {{a}} *-homomorphism of finite nuclear dimension
must be exact.

\begin{cor}
\label{finite-dimnuc-image-exact}
Let $A$ and $B$ be \CAs.  {\blue{If}}
$h:A\rightarrow B$ is a *-homomorphism with
$\dimnuc h<\infty$, then $h(A)$ is exact.
\end{cor}

\begin{proof}
By Proposition \ref{homomorphism-dimnuc-equal-embedding-dimnuc},
the embedding $\iota: h(A)\hookrightarrow B$
satisfies $\dimnuc \iota=\dimnuc h<\infty.$
Thus $\iota$ is a nuclear map.
It follows that  $h(A)$ is exact
{\blue{(see \cite[6.1.11]{Rordam 2002}).}}


\end{proof}

By \cite[Theorem 2.8]{KP2000},
every separable exact \CA\ admits an embedding into the Cuntz algebra $\mathcal{O}_2$.
By \cite[Theorem 7.4]{WZ2010},
{\blue{one has}} $\dimnuc \mathcal{O}_2=1.$
{\blue{Thus}} every embedding of separable exact \CA\ into $\mathcal{O}_2$
has nuclear dimension at most $1$.
Therefore,
it seems  {\blue{to be interesting to observe the following statement.}}

\begin{prop}
\label{finite-nuclear-dim-of-map-implies-finite-dimnuc-of-alg}
Let $h: A\rightarrow B$ be a *-homomorphism  such that
$h(A)$ is a hereditary $C^*$-subalgebra of $B$.
Then $\dimnuc h=\dimnuc h(A)$.
Moreover, if {{$B$ is separable and}}
$h(A)$ is a full hereditary $C^*$-subalgebra of $B$,
then $\dimnuc h=\dimnuc B$.
\end{prop}

\begin{proof}
First, let us assume that $h$ is surjective.
Then the embedding $\iota: h(A)\rightarrow B$
is the identity map $\mathrm{id}_B.$
{{By}} Proposition \ref{homomorphism-dimnuc-equal-embedding-dimnuc},
we have
$\dimnuc h=\dimnuc \iota=\dimnuc \mathrm{id}_B=\dimnuc B$.

Now we assume that $C:=h(A)$ is a hereditary $C^*$-subalgebra of $B$.
Then by \cite[Proposition 1.6]{BGSW2019}
(also see \cite[Proposition 2.4]{TW2014}),
$\dimnuc h=\dimnuc h^C,$ where $h^C: A\to C$ is the \hm\,
defined by
$h^C(a):=h(a)$
for all $a\in A$ (but $h: A\to B$).
Now since $h^C$ is surjective,
by what we have proved,
$\dimnuc h=\dimnuc h^C=\dimnuc h(A).$
Moreover, if {{$B$ is separable and}}
$h(A)$ is a full hereditary $C^*$-subalgebra of $B,$
by \cite[Corollary 2.8]{WZ2010}, {\blue{then
$\dimnuc B=\dimnuc h(A)=\dimnuc h.$}}

\end{proof}

\begin{cor}

Let $A$ be a \CA\, {\blue{and}}  $I\subset A$ {\blue{be}} a closed ideal.
If the quotient map $\pi: A\rightarrow A/I$ has finite nuclear dimension,
then $A/I$ also has finite nuclear dimension.

\end{cor}


\section{A criterion for
generalized inductive limits becoming finite nuclear dimension}

\begin{df}[\cite{BKGILFD}\,Generalized inductive system]
\label{Generalized inductive system}
Let $A_n$ be a sequence of $C^*$-algebras {\blue{and}}
$\phi_{m,n}: A_m\rightarrow A_{n}$ be a map ($m<n$).
We say $(A_n,\phi_{m,n})$ forms a
generalized inductive system if the following hold:
For any $k\in\mathbb{N}$, any $x,y \in A_k$,
any $\lambda\in\mathbb{C}$, {\blue{and}}
any $\epsilon>0$, there exists $M\in\N$
{\blue{such}} that, for any $n>m\geq M$,

(1) $\|\phi_{m,n}(\phi_{k,m}(x)+\phi_{k,m}(y))
-(\phi_{k,n}(x)+\phi_{k,n}(y))\|\leq \epsilon$,

(2) $\|\phi_{m,n}(\lambda\phi_{k,m}(x))-\lambda\phi_{k,m}(x)\|\leq\epsilon$,

(3) $\|\phi_{m,n}(\phi_{k,m}(x)^*)-\phi_{k,n}(x)^*\|\leq\epsilon$,

(4) $\|\phi_{m,n}(\phi_{k,m}(x)\phi_{k,m}(y))
-\phi_{k,n}(x)\phi_{k,n}(y)\|\leq\epsilon$, {\blue{and}}

(5) $\sup_r\|\phi_{k,r}(x)\|<\infty$.

%
\noindent
{{The}} system is called {\blue{p.c.}} (or c.p.c.), if
all $\phi_{m,n}$ are  {\blue{p.c.~maps}} {{(or c.p.c.~maps)}}.

If $(A_n,\phi_{m,n})$ forms a generalized inductive system,
then the following is a \CA\, {\blue{which}} we call it the
{\blue{generalized  inductive limit}} of $(A_n,\phi_{m,n})$:
$$
\lim_{n}(A_n,\phi_{m,n})
:=
\overline{\{{{\pi_\infty}}(\{
\phi_{n,1}(a),\phi_{n,2}(a),
{{\cdots}}\}): n\in\mathbb{N}, a\in A_n\}}^{\|\cdot\|}\subset \prod_{n=1}^\infty A_n/\bigoplus_{n=1}^{\infty}A_n,
$$
{{where}} $\phi_{m,n}:=0$
for {{$m>n$,}}
and $\phi_{n,n}:=\id_{A_n}$.
For $i\in\N$, define (see \cite[2.1.2, 2.1.3]{BKGILFD})
{\blue{$\phi_{i,\infty}:
A_i\rightarrow \lim_{n}(A_n,\phi_{m,n})$ by
$
x\mapsto{{\pi_\infty}}(\{\phi_{i,1}(x),\phi_{i,2}(x),\cdots\}).
$}}
\end{df}
\begin{nota}
\label{def-of-gen-ind-lim} 
{\blue{Given a sequence of $C^*$-algebras {\blue{$A_n$}}
and a sequence of maps $\phi_{n}: A_n\rightarrow A_{n+1}$,
for $m<n$, define
{{$\phi_{m,n}$ to be the composition of
$\phi_{m},\phi_{m+1},\cdots,\phi_{n-1}$:}}
$$
\phi_{m,n}:=\phi_{n-1}\circ\phi_{n-2}\circ
{\blue{\cdots}}
\circ\phi_m:
A_m\rightarrow A_n,
$$
and define
$\phi_{m,n}:=0$  for $m>n$, and define $\phi_{n,n}:=\id_{A_n}$.
We say $(A_n,\phi_n)$ forms a generalized inductive system,
if $(A_n,\phi_{m,n})$ forms a generalized inductive system.
Accordingly $\lim_n (A_n,\phi_{m,n})$ will be
denoted by $\lim_{n}(A_n, \phi_{n})$.
}}
\end{nota}


\begin{lem}
\label{positive-criterion-being-gen-ind-lim}

Let $A_n$ be $C^*$-algebras
and
 $\phi_{n}: A_n\rightarrow A_{n+1}$
{\blue{be}} {\blue{p.c.~maps}}  (or c.p.c.~maps), \,\,{\blue{$n=1,2,\cdots.$}}
{\blue{Let}}
$\phi_{m,n}$ be defined as in {{Notation}}
\ref{def-of-gen-ind-lim}.
If, for any $k\in\mathbb{N}$, any $\epsilon>0$, {\blue{and}}
any $x,y\in A_{k+}^1$,
there exists $m>k$ such {\blue{that,}} for all $n>m$,
$\phi_{k,n}(x)\phi_{k,n}(y)
\approx_{\epsilon}
\phi_{m,n}(\phi_{k,m}(x)\phi_{k,m}(y))$,
then $(A_n,\phi_n)$ forms  a 
{\blue{p.c.}} (or c.p.c.) generalized inductive system.

\end{lem}

\begin{proof}

Since $\phi_n$ are 
{\blue{p.c.}} (or c.p.c.)~maps,
(1), (2), (3) {\blue{and}} (5)
in Definition \ref{Generalized inductive system} are satisfied.
{\blue{It remains to}} show that
(4) in Definition \ref{Generalized inductive system} holds.

Let $k\in\mathbb{N}$, let $a,b\in A_{k+}^1$ and let $\epsilon>0$.
By {\blue{the}}  assumption,  there exists
$M\in\mathbb{N}$ such that, for any $i \ge M$,
$\phi_{k,i}(a)\cdot\phi_{k,i}(b)
\approx_{\epsilon/2}
\phi_{M,i}(\phi_{k,M}(a)\cdot\phi_{k,M}(b)).$
%
{\blue{Then,}}
for any $n>m\geq M$,
we have
\beq
\phi_{m,n}(\phi_{k,m}(a)\cdot\phi_{k,m}(b))
&\approx_{\epsilon/2}&
\phi_{m,n}(\phi_{M,m}(\phi_{k,M}(a)\cdot
\phi_{k,M}(b)))
\nonumber\\
&=&\phi_{M,n}(\phi_{k,M}(a)\cdot\phi_{k,M}(b))
\approx_{\epsilon/2}
\phi_{k,n}(a)\cdot\phi_{k,n}(b).
\nonumber
\eneq


\noindent
{{Thus}} (4) in Definition \ref{Generalized inductive system}
holds for any $a,b \in A_{k+}$.
Since $A_{k+}^1$ generates $A_k$ as linear {\blue{space,}}
then (4) in Definition \ref{Generalized inductive system}
holds for any $a,b \in A_{k}$.
{\blue{Lemma follows.}}

\end{proof}



\begin{lem}
\label{multiplicative-for-contractive-linear-system-corollary}
Let $(A_i,\phi_{j,i})$ be a
{\blue{p.c.}}~generalized inductive system
of $C^*$-algebras.
{\blue{Then,}} for any $n,k\in\N$,
any finite subset $\calF\subset A_k$, {\blue{and}}
any $\epsilon>0$,
there exists $M>k (\in\N)$
such {\blue{that,}} for any
$j>i\geq M$,
any $m_1,m_2\leq n\in\N$,
and any $x_1,x_2, {\blue{\cdots, x_{m_1}}},y_1,y_2, \cdots,y_{m_2}\in \calF$,
$$\phi_{i,j}\left(\prod_{r=1}^{m_1} \phi_{k,i}(x_r)
\cdot \prod_{r=1}^{m_2} \phi_{k,i}(y_r)\right)
\approx_{\epsilon}
\phi_{i,j}\left(\prod_{r=1}^{m_1} \phi_{k,i}(x_r)\right)
\cdot
\phi_{i,j}\left(\prod_{r=1}^{m_2} \phi_{k,i}(y_r)\right).
$$

\end{lem}

\begin{proof}
It suffices to show that, for any $k, n\in \N,$  any {\blue{$\ep>0,$}} and any finite subset ${\cal F}\in A_k^1,$
there exists $M>0$ such that ($1\le l \le n$), for $j>i>M,$
\beq
\phi_{i,j}(\prod_{r=1}^l \phi_{k,i}(x_r))\approx_{\ep} \prod_{r=1}^l\phi_{k,j}(x_r)\rforal x_1,
x_2,
{{\cdots}},x_{\green{l}}\in {\cal F}.
\eneq
This follows from Definition \ref{Generalized inductive system} and {\blue{the}} induction on $n$ immediately.
The case $n=2$ follows from (4) in
{{Definition}}
\ref{Generalized inductive system}.  Assume the above holds for $2,{\blue{3,}}
{{\cdots}},n-1.$
Then, for $\dt=\ep/3,$ there exists $M_0>0$ such that, for any $j>i>M_0,$
\beq\label{62+n}
\phi_{i,j}(\prod_{r=1}^{l'}\phi_{k,i}(x_r))\approx_{\dt} \prod_{r=1}^{l'} \phi_{k,j}(x_r)\rforal x_r\in {\cal F}\, (1\le r\le l'\le n-1).
\eneq
For all $x_r\in {\cal F},$   with  $y=\prod_{r=1}^{l'}\phi_{k,i}(x_r)$ and $z=\phi_{k,i}(x_{l'+1})$
(${{1\le l'}}<l'+1\le n$),
there is $M_1>0$ such that, for $K>j\ge M_1,$
{{$\phi_{j,K}(\phi_{i,j}(y)\phi_{i,j}(z))
\approx_{\dt}\phi_{i,K}(y)\phi_{i,K}(z).$ Then}}
\beq\nonumber
\phi_{j, K}(\prod_{r=1}^{l'+1}\phi_{k,{{j}}}(x_r)){\overset{(\ref{62+n})}\approx_{\dt}}\, \phi_{j,K}(\phi_{i,j}(y)\phi_{i,j}(z))
\approx_{\dt}\phi_{i,K}(y)\phi_{i,K}(z)
\overset{(\ref{62+n}\,)}\approx_{\dt}\, \prod_{r=1}^{l'+1}\phi_{k,K}(x_r).
\eneq

\end{proof}

\noindent
{{We}} end this section with a sufficient and necessary condition
for
{\blue{a c.p.c.~generalized inductive limit
having}} finite nuclear dimension.

\begin{thm}
\label{gene-lim-fin-nucdim}
Let {$n\in\mathbb{N}\cup\{0\}$}.
Let $(A_i,\phi_{i,j})$ be a 
{{c.p.c.~}}generalized inductive system
of $C^*$-algebras. 
Let $A=\lim_{i}(A_i,\phi_{i,j})$.
Then $\dimnuc A\leq n$
if and only if the following hold:

For any $i\in\mathbb{N}$, any finite subset $\mathcal{G}\subset A_i$,
{\blue{and}} any $\epsilon>0$,
there {{exist}} a finite dimensional $C^*$-algebra $F$,
{a} c.p.~map $\alpha: A_i\rightarrow F$,
and {\blue{an}} $n$-decomposable c.p.~map $\beta: F\rightarrow A$
{\blue{such}} {\blue{that}}
%
$$
\phi_{i,\infty}(x)\approx_{\epsilon}
\beta\circ\alpha(x) {\blue{\rforal x\in  {\cal G}.}}
$$

\end{thm}

\begin{proof}

For the ``only if'' part, {\blue{let us}} assume that $\dimnuc A\leq n$.
Let $i\in\mathbb{N},$
let $\mathcal{G}\subset A_i$ be a finite subset, and
let $\epsilon>0.$
There exist {{a}} finite dimensional $C^*$-algebra $F$,
{\blue{a}} c.p.c.~map $\alpha': A\rightarrow F$,
and {{an}}
$n$-decomposable  c.p.~map $\beta: F\rightarrow A$
{\blue{such}} that, for all
$x\in \mathcal{G}$,
$\phi_{i,\infty}(x)\approx_{\epsilon}
\beta\circ\alpha'(\phi_{i,\infty}(x))$.
Define a c.p.c.~map $\alpha:=\alpha'\circ\phi_{i,\infty}.$
Then, for all $x\in\mathcal{G}$,
$\phi_{i,\infty}(x)\approx_{\epsilon}
\beta\circ\alpha'(\phi_{i,\infty}(x))=
\beta\circ\alpha(x)$.

For the ``if'' part, we will
apply
Theorem \ref{equivalent-definition-on-nuclear-dimension-second}
to show that $\dimnuc \mathrm{id}_A\leq n.$

Let $\bar{\mathcal{G}}\subset A_+$ be a finite subset and
let $\epsilon>0.$
{\blue{Choose}}  $N:=1+\max\{\|x\|:x\in\bar{\mathcal{G}}\}$ and choose
$\delta:=\min\{1,\frac{\epsilon}{4(N+1)}\}$.
There {{exist}} $k\in\mathbb{N}$ and 
{{a}} finite subset $\hat{\mathcal{G}}\subset (A_k)_{sa}$
such that
\beq
\label{f5-6thm-1}
\{x^{1/2}:x\in \bar{\mathcal{G}}\}
\subset_{\delta} \phi_{k,\infty}({\hat{\mathcal{G}}}).
\eneq

\noindent
{{Since}} $A=\lim_{i}(A_i,\phi_{i,j})$ is a generalized inductive system,
there exists
$M_1>k\in\bbN$
{\blue{such that,}}
for any $j>i\geq M_1\in\N$ {\blue{and}}  any $x \in \hat{\mathcal{G}}
$,
\beq
\|\phi_{i,j}(\phi_{k,i}(x))-\phi_{k,j}(x)\|\leq\delta.
\eneq
{\blue{Hence,}} for any $i\geq M_1\in\N$, we have
\beq
\label{f5-6-thm-4}
\|\phi_{i,\infty}(\phi_{k,i}(x))-\phi_{k,\infty}(x)\|\leq\delta
{{\mbox{ for all }}} x \in \hat{\mathcal{G}}.
\eneq




\noindent
{{By}} {{Lemma}}
\ref{multiplicative-for-contractive-linear-system-corollary},
%
there exists $M>M_1\in\N$,
such {\blue{that,}}
for any $ j> M$,
{\blue{any}} $ 1\leq m_1,m_2\leq 4\in\N$,
and any $ x_1,{\green{x_2,}}\cdots, x_{m_1},y_1,{\green{y_2,}}\cdots,y_{m_2}\in
\hat{\mathcal{G}}$,
\beq
\label{f5-6-thm-1}
\phi_{M,j}\left(\prod_{r=1}^{m_1} \phi_{k,M}(x_r)
\cdot \prod_{r=1}^{m_2} \phi_{k,M}(y_r)\right)
\approx_{\delta}
\phi_{M,j}\left(\prod_{r=1}^{m_1} \phi_{k,M}(x_r)\right)
\cdot
\phi_{M,j}\left(\prod_{r=1}^{m_2} \phi_{k,M}(y_r)\right).
\eneq
{\blue{Let}} $\mathcal{G}:=\{\phi_{k,M}(x)^2:x\in\hat{\mathcal{G}}\}
\subset (A_{M})_+$. {\blue{Then,}} by \eqref{f5-6-thm-1},
we have
\beq
\label{f5-6-thm-2}
\phi_{M,j}(xy)\approx_{\delta}
\phi_{M,j}(x)\phi_{M,j}(y)
{\mbox{ for all }} x,y\in \mathcal{G}
\cup(\mathcal{G}\cdot \mathcal{G})
{\mbox{ and for all }} j>M.
\eneq
{\blue{Consequently,}} we have
\beq
\label{f5-6-thm-3}
\phi_{M,\infty}(xy)\approx_{\delta}
\phi_{M,\infty}(x)\phi_{M,\infty}(y)
{\mbox{ for all }} x,y\in \mathcal{G}
\cup(\mathcal{G}\cdot \mathcal{G}).
\eneq

\noindent
{{Let}} $N_1:=1+\max\{\sup_{j>M}\{\|\phi_{M,j}(x)\|:x\in
\mathcal{G}\cup(\mathcal{G}\cdot\mathcal{G})
\}\}$.
By the assumption of the theorem,
there exists a finite dimensional \CA\ $F$,
{\blue{a}}  c.p.~map $\alpha: A_M\rightarrow F$,
and {\blue{an}} $n$-decomposable c.p.~map $\beta: F\rightarrow A$
{\blue{such}} that
\beq
\label{f5-4thm-1}
\phi_{M,\infty}(x)\approx_{\frac{\delta}{N_1}}
\beta\circ\alpha(x)
{\mbox{ for all }} x
\in\mathcal{G}\cup(\mathcal{G}\cdot\mathcal{G})
\cup(\mathcal{G}\cdot\mathcal{G}\cdot\mathcal{G})
\cup(\mathcal{G}\cdot\mathcal{G}\cdot\mathcal{G}\cdot\mathcal{G}).
\eneq
For any $a\in\bar{\mathcal{G}}$,
by \eqref{f5-6thm-1}, there exists $x_a\in {\hat{\mathcal{G}}}$
{\blue{such}} that $a^{1/2}\approx_{\delta}\phi_{k,\infty}(x_a).$
Then
\beq
\nonumber
\hspace{-0.5in}a={{(a^{1/2})^2}} 
&\approx_{(2N+\delta)\delta}&
{{\phi_{k,\infty}{(x_a)^2}}}
\qquad\qquad\qquad\qquad\qquad\qquad\qquad~~
\text{(by \eqref{f5-6thm-1})}
\\\nonumber
&\approx_{(2N+\delta)\delta}
&{{\phi_{M,\infty}(\phi_{k,M}(x_a))^2}}
\qquad\qquad\qquad\qquad\qquad\quad~
\mbox{(by \eqref{f5-6-thm-4})}
\\\nonumber
&\approx_{\delta}&
\phi_{M,\infty}({{\phi_{k,M}(x_a)^2}})
\qquad\qquad
\hspace{1.1in}\mbox{(by \eqref{f5-6-thm-3})}
\\\nonumber
&\approx_{\delta}&
\beta\circ\alpha({{\phi_{k,M}(x_a)^2}})
\in\beta\circ\alpha(\mathcal{G}).
\qquad\qquad
\hspace{0.4in}\mbox{(by \eqref{f5-4thm-1})}
\eneq

\noindent
{{Thus}} $\bar{\mathcal{G}}\subset_{\epsilon} \beta\circ\alpha(\mathcal{G})$.
%
For 
$x,y\in \mathcal{G}\cup(\mathcal{G}\cdot \mathcal{G})$, {\blue{by \eqref{f5-4thm-1}, 
{{\eqref{f5-6-thm-3}}}
and \eqref{f5-4thm-1} again,}}
\beq
\beta\circ\alpha(xy)
\approx_{\delta}
\phi_{M,\infty}(xy)
\approx_{\delta}
\phi_{M,\infty}(x)\phi_{M,\infty}(y)
\approx_{2\delta}
\beta\circ\alpha(x)\beta\circ\alpha(y).
\eneq

\noindent
{\blue{Therefore}} $\beta\circ\alpha(xy)\approx_{\epsilon}
\beta\circ\alpha(x)\beta\circ\alpha(y)$ {{for all}} $x,y\in
\mathcal{G}\cup(\mathcal{G}\cdot \mathcal{G})$.
Then, by Theorem \ref{equivalent-definition-on-nuclear-dimension-second},
we have $\dimnuc \mathrm{id}_A\leq n.$ Consequently,  $\dimnuc A\leq n$.


\end{proof}



\section{Simple $C^*$-algebra of finite tracial nuclear dimension}

\begin{df}
\label{D71}
Let $A$ and $B$ be \CA s and
let $\phi: A\rightarrow B$ be a map.
Let {\blue{$\epsilon \geq 0.$}}
{\blue{If,}} for any $a_1,a_2\in A_+^1$ with $a_1a_2=0$,
we have $\|\phi(a_1)\phi(a_2)\|\leq \epsilon$,
then we say $\phi$ is {\blue{an}} $\epsilon$-almost order zero map.

\end{df}

\begin{df}
\label{defn-n-almost-dividable-c.p.-map}
{\blue{Let $A$ be a \CA\, and let $F$ be a finite dimensional \CA.}}
Let
$\phi: F\rightarrow A$ be a
c.p.~map {\blue{and let}}
$n\in{{\mathbb{N}\cup\{0\}}}$ {\blue{be an integer.}} The map
$\phi$ is called $(n,\epsilon)$-dividable
if $F$ can be written as $F=F_0\oplus \cdots\oplus F_n$
(where $F_i$ are {{ideals of $F$}})
{\blue{such that}} $\phi|_{F_i}$ {{is}}
{{a}} c.p.c.~$\epsilon$-almost order zero map
{{for}} $i=0,{\green{1,}}\cdots,n$.

\end{df}

{\blue{The next two propositions follow from the projectivity of
the cone of finite dimensional $C^*$-algebras.}}

\begin{prop}
\label{app-order-zero-exactly}

For any finite dimensional $C^*$-algebra $F$
and any $\epsilon>0$,
there exists
$\delta>0$
{\blue{such}} that, for any $C^*$-algebra $A$ and any c.p.c.~map
$\phi: F\rightarrow A$ {\blue{which is}}
$\delta$-almost order zero,
{\blue{there}} exists a c.p.c.~order zero map $\psi: {\blue{F}}\rightarrow A$
satisfying $\|\phi-\psi\| \leq\epsilon$.

\end{prop}

\begin{proof}
{\blue{Let $F$
be fixed.}}
{\blue{If such $\delta$ described in the proposition  does not exists,}}
then,  %
{\blue{there exist $\ep_0>0,$}}
a sequence of $C^*$-algebra $A_n$,
and c.p.c.~maps $\phi_n: F\rightarrow A_n$
{\blue{such}} that $\phi_n$ is $1/n$-almost order zero,
and,
for any $n\in\mathbb{N}$ {\blue{and}}
any c.p.c.~order zero map $\psi: F\rightarrow A_n$,
we have $\|\phi_n-\psi\|>  {\blue{\epsilon_0}}$.

{{Define a c.p.c.~order zero map}}
{\blue{$\Phi: F\rightarrow \prod_{n=1}^\infty A_n/\bigoplus_{n=1}^\infty A_n$}}
{{by}} $x\mapsto \pi_\infty(\{\phi_n(x)\})$.
{\blue{Then, by \cite[Proposition 1.2.4]{W-2009CovII},}}
$\Phi$ has a c.p.c.~order zero lift
$\Psi: F\rightarrow \prod_{n=1}^\infty A_n.$ 
Let $\psi_n$ be the components of $\Psi$ corresponding to $A_n.$
Since
$\lim_{n\rightarrow \infty}\|\phi_n(x)-\psi_n(x)\|{\blue{=0}}$  for all  $x\in F,$
and the unit ball of $F$ is compact,
there exists $n_0$ such that $\|\phi_{n_0}-\psi_{n_0}\|<{\blue{\epsilon_0 /2}}.$
This leads to
a contradiction. Thus $\delta$ {\blue{does exist}} {\blue{and the proposition follows.}}

\end{proof}

\begin{prop}
\label{n-almost-dividable-exact}
Let ${\blue{F=M_{k_0}\oplus M_{k_1}\oplus \cdots \oplus {\blue{M_{k_r}}}}}$ be a finite dimensional $C^*$-algebra
{\blue{with a {{standard}}
generating set ${\cal G}^F$ {\blue{(see {{Notation}} \ref{NoGF})}} in the unit ball of $F.$}}

{\blue{(1)}} For any $\epsilon>0$, there exists $\delta_1(\ep)>0$
such that, for any $n\in{{\mathbb{N}\cup\{0\}}}$($n\le r$), any $C^*$-algebra $A$ and any
{\blue{$(n,\delta_1(\ep))$-dividable}} c.p.~map
$\phi: F\rightarrow A$,
there exists a piecewise contractive $n$-decomposable  c.p.~map $\psi: F\rightarrow A$
satisfying $\|\phi-\psi\| \leq\epsilon$.

{\blue{(2) For any $\ep>0,$ there exists $\dt_2(\ep)>0$ such that,
for any $n\in {{\mathbb N\cup\{0\}}}$
($n\le r$), {\blue{any $\sigma>0,$  and any
{{$(n,\sigma)$-}}{\blue{dividable}} c.p.~map $\af: F\to A$
(for any \CA\, $A$),}}
and any c.p.c.~map $\bt: B:=C^*(\af(F))\to C$ (for any \CA\, $C$) which
is {{$(\alpha({\cal G}^F), \dt_2(\ep))$}}-multiplicative,  $\bt\circ \af$ is
{{an}} $(n, {{\sigma+}}\ep)$-dividable c.p.~map.}}

\end{prop}

\begin{proof}

{{For (1),}} by Proposition \ref{app-order-zero-exactly},
there exists $\delta_1(\ep)>0$ such {\blue{that,}}
for any $C^*$-algebra $A$ and any
{\blue{$\delta_1(\ep)$-almost order zero}} c.p.c.~map
$\chi: F\rightarrow A$,
there exists a c.p.c.~order zero map $\psi: F\rightarrow A$
satisfying $\|\chi-\psi\| \leq\frac{\epsilon}{{\blue{r+1}}}$.

Now let $\phi: F\rightarrow A$ be {\blue{an}}
$(n,{\green{\delta_1(\ep)}})$-dividable c.p.~map,
i.e. $F$ can be written as $F=F_0\oplus \cdots\oplus F_n$
{{such that}} each $\phi|_{F_{{j}}}$ {\blue{is a}}
c.p.c.~{\green{$\delta_1(\ep)$}}-almost order zero {{map.}} 
Then $n\leq {\blue{r}}$.

Let $\pi_j: F\rightarrow F_j$ be the {\blue{quotient map.}}
Note  that $\mathrm{id}_F=\sum_{j=0}^n\pi_j.$
Then $\phi\circ \pi_j$ {\blue{is a}}
c.p.c.~$\delta_1(\ep)$-almost order zero {\blue{map}} on $F,$
$j=0,{\green{1,}}\cdots,n.$
By  the choice of $\delta_1(\ep),$ there are
c.p.c.~order zero maps
$\psi_0,{\green{\psi_1,\cdots}},
\psi_n: F\rightarrow A$ such that
$\|\psi_j- \phi\circ \pi_j\|\leq \frac{\epsilon}{{\blue{r+1}}}.$
Therefore
$\psi:={\green{\sum_{j=0}^n\psi_j\circ\pi_j}}
:F\rightarrow A$ is piecewise contractive $n$-decomposable  c.p.~map
{\blue{and}} $\|\psi-\phi\|=\|(\psi-\phi)\circ (\sum_{j=0}^n\pi_j)\|
=\|\sum_{j=0}^n(\psi- \phi\circ \pi_j)\circ \pi_j\|\leq\epsilon$.

{\blue{For (2), write $F=F_{{0}}\oplus
F_{{1}}\oplus
{{\cdots}} \oplus F_n$
such that $\af|_{F_i}$ is {\blue{a}} $\sigma$-almost order zero map
(${{0}}\le i\le n$).
One observes that if $\bt$ is $(\af({\cal G}^F),\dt)$-multiplicative,
{{then}}
\beq
\|\bt(\af(a)\af(b))-\bt\circ \af(a)\bt\circ\af(b)\|<\dt\rforal a, b\in {\cal G}^F.
\eneq
Since
${\cal G}^F$ is a standard generating set (see {{Notation}} \ref{NoGF}) and the unit ball of $F$ is compact,
for any $\ep>0,$
one can find a universal constant $\dt_2(\ep)>0$
independent of $\af$  (but dependent of $F$)
such that
\beq
\sup_{\|a\|, \|b\|\le 1}\|\bt(\af(a)\af(b))-\bt\circ \af(a)\bt\circ\af(b)\|<\ep,
\eneq
if $\bt$ is $(\af({\cal G}^F),\dt_2(\ep))$-multiplicative.  Thus
$\bt\circ \af|_{F_i}$ is {{a $ (\sigma+\ep)$}}-almost order zero map
(${{0}}\le i\le n$).}}

\end{proof}


\begin{df}
\label{almost-n-dividable-error}

Let $F$ be a finite dimensional $C^*$-algebra and let $\epsilon>0.$
Define
$$
\Delta(F,\epsilon):=\min\{\dt_1(\ep), {{\dt_2(\ep)}},
\dt_2(\dt_1(\ep)), 1/2\}{{>0}},
$$
where $\dt_1(\ep),$ {\green{$\dt_2(\ep)$}} and $\dt_2(\dt_1(\ep))$ are as given in
Proposition \ref{n-almost-dividable-exact}.

\end{df}

\begin{df}\label{Didnuc}
Let $A$ be a unital simple $C^*$-algebra and let $n\in{\green{\mathbb{N}\cup\{0\}}}.$
We say {\blue{that}} $\mathrm{id}_A$ has tracial nuclear dimension no more than $n$,
{\blue{if,}} for any finite subset $\calF\subset A$,
any $\epsilon>0$, {\blue{and}}
any $a\in A_+\backslash\{0\}$,
there {\blue{exist}} a finite dimensional $C^*$-algebra $F$,
{\blue{a}}  c.p.c.~map $\alpha: A\rightarrow F$,  {\blue{a}}
nonzero piecewise contractive $n$-decomposable  c.p.~map $\beta: F\rightarrow A,$ {\blue{and a}}
c.p.c.~map $\gamma: A\rightarrow A\cap \beta(F)^{\bot},$
such that

(1) $x\approx_{\epsilon} \gamma(x)+\beta\circ\alpha(x)$ for all $x\in\calF,$ {{and}}

(2) $\gamma(1_A)\lesssim_A a$.


\noindent
{{If}} $\id_A$ has tracial nuclear dimension no more than $n,$ we write
$\trdimnuc \id_A\le n.$

\end{df}

{\blue{Note that,}}
for any simple unital \CA\ $A,$ {\blue{we have}}
$\trdimnuc \id_A\leq\dimnuc{\rm id}_A=\dimnuc A.$

\begin{df}
\label{Didwnuc}
Let $A$ be a unital simple $C^*$-algebra and let
$n\in{\green{\mathbb{N}\cup\{0\}}}.$
We say {\blue{that}} $\mathrm{id}_A$ has weakly
tracially nuclear dimension no more than $n$,
if, for any finite subset $\calF\subset A$,
any $\epsilon>0$, {\blue{and}}
any $a\in A_+\backslash\{0\}$,
there exist a finite dimensional $C^*$-algebra $F$,
and {\blue{a}} c.p.c.~map $\alpha: A\rightarrow F$,  {\blue{a}}
nonzero piecewise contractive $n$-decomposable  c.p.~map $\beta: F\rightarrow A,$ {\blue{and a}}
c.p.c.~map $\gamma: A\rightarrow A$,
such that

(1) $x\approx_{\epsilon} \gamma(x)+\beta\circ\alpha(x)$ for all $x\in\calF,$ {\green{and}}

(2) $\gamma(1_A)\lesssim_A a$.

\vspace{0.1in}

{\blue{Later, we will show that ${\rm Trdim_{nuc}}\id_A\le n$ is equivalent to the statement
that $A$ is
asymptotically tracially in ${\cal N}_n.$  Therefore, one may also say {\blue{that}} $A$ is
weakly
asymptotically tracially in ${\cal N}_n,$
if $\id_A$ has weakly tracial dimension no more than $n.$  In a subsequent paper, we will discuss that notion.
This notion can be also defined in the case that $A$ is not simple and not unital.}}


\end{df}

\begin{prop}\label{Pdinuc}
Let $A$ be a unital simple {\blue{separable}}
$C^*$-algebra and let $n\in{\green{\mathbb{N}\cup\{0\}}}.$
Assume that $\trdimnuc \mathrm{id}_A\leq n.$
{\blue{Then,}} for any finite subset ${\blue{\calF\subset A}}$,
any $\epsilon>0,$ {\blue{and}}
any $a\in A_+\backslash\{0\}$,
there exist a finite dimensional $C^*$-algebra $F$,
a c.p.c.~map $\alpha: A\rightarrow F$,
a  nonzero piecewise contractive $n$-decomposable  c.p.~map $\beta: F\rightarrow A$, and a
c.p.c.~map $\gamma: A\rightarrow A\cap \beta(F)^{\bot}$
{\blue{such}} that

(1) $x\approx_{\epsilon} \gamma(x)+\beta\circ\alpha(x)$ for all $x\in\calF$,

(2) $\gamma(1_A)\lesssim_A a,$ {\green{and}}

(3) $\|\beta\circ\alpha(x)\|\geq \|x\|-\epsilon$ for all $x\in\calF$.

\end{prop}

\begin{proof}

Let $\calF\subset A$ be a finite subset,
let $\epsilon>0$ and let $a\in A_+\backslash\{0\}$.
Let $\calF\subset X_1\subset X_2\subset \cdots\subset A$ be finite subsets
such that $\cup_{m\geq 1} X_m$ is norm dense in $A$.
Since $\trdimnuc \id_A\leq n$,
for each $m\in\N,$
there  exist  a finite dimensional $C^*$-algebra $F_m$,
a c.p.c.~map  $\alpha_m: A\rightarrow F_m$,
a nonzero piecewise contractive $n$-decomposable  c.p.~map $\beta_m: F_m\rightarrow A$,  and
a c.p.c.~map $\gamma_m: A\rightarrow A\cap \beta_m(F_m)^{\bot}$
such that, for all $m,$

(i) $x\approx_{\frac{\epsilon}{m}} \gamma_m(x)+\beta_m\circ\alpha_m(x)$
for all $x\in X_m$ {\blue{and}}

(ii) $\gamma_m(1_A)\lesssim_A a$.

\noindent
{{Define}} a c.p.c.~map
$\Gamma: A\rightarrow l^{\infty}(A)/c_0(A)$ {\blue{by}}
$x\mapsto \pi_{\blue{\infty}}(\{\gamma_1(x),\gamma_2(x),\cdots\})$
and define a c.p.~map
$\Phi: A\rightarrow  l^{\infty}(A)/c_0(A)$ {\blue{by}}
$x\mapsto \pi_{\blue{\infty}}
(\{\beta_1\circ\alpha_1(x),\beta_2\circ\alpha_2(x),\cdots\}).$
{\blue{Since}}
$\gamma_m(A)\bot  {\blue{(\beta_m\circ\alpha_m(A))}}$, we have
$\Gamma(A)\bot\Phi(A).$
Note that, by (i),
we have
$\iota_A=\Gamma+\Phi.$
{\blue{It follows that}}
$\Gamma$ and $\Phi$ are *-homomorphisms.

If $\Phi$ is a zero map, then $\iota_A=\Gamma$. Thus there exists
$m_0\in \N$ such that $\|1_A-\gamma_{m_0}(1_A)\|<1/2$.
Therefore $\gamma_{m_0}(1_A)$ is invertible in $A$.
Then $\gamma_{m_0}(1_A)\bot \beta_{m_0}(F_{m_0})$
implies $\beta_{m_0}(F_{m_0})=\{0\}$,
which is contradict to that $\beta_{m_0}$ is a nonzero map.
Hence $\Phi$ can not be a zero map.
{\blue{In other words,}} $\Phi(1_A)$ is a nonzero projection which has norm one.
Thus there exist natural numbers  $m_1<m_2<\cdots$ such that
\beq
\label{f6-9-prop-3}
\|\beta_{m_i}\circ\alpha_{m_i}(1_A)\|\geq 1-1/i, \,\,{\blue{i=1,2,{{\cdots.}}}}
\eneq
{\blue{Define a
*-\hm\,
$\Psi: A\rightarrow l^{\infty}(A)/c_0(A)$ by
$x\mapsto \pi_\infty(\{\beta_{m_1}\circ\alpha_{m_1}(x),
\beta_{m_2}\circ\alpha_{m_2}(x),\cdots\}).$}}
By \eqref{f6-9-prop-3} and Proposition \ref{strictly-embedding-simple-case},
$\Psi$ is a strict embedding.  {\blue{Therefore}} there exists $s\in\N$
such that
\beq
\|\beta_{m_s}\circ\alpha_{m_s}(x)\|\geq \|x\|-\epsilon
\rforal  x\in\calF.
\eneq
{\blue{Set}} $F:=F_{m_s}$, $\alpha:=\alpha_{m_s}$,
$\beta:=\beta_{m_s}$, {\blue{and}} $\gamma:=\gamma_{m_s}.$
{\blue{The}} proposition follows.

\end{proof}

\begin{rem}\label{69-n1126}
Note {\blue{that}}  condition (3) in {{Proposition}}
\ref{Pdinuc} implies that $\bt$ is nonzero. Therefore, in the light of {{Proposition}} \ref{Pdinuc}, in Definition \ref{Didnuc},
we may replace the condition that $\bt\not=0$ by  condition (3) in {{Proposition}} \ref{Pdinuc}.
\end{rem}


The following proposition is extracted from the proof of
\cite[Proposition 2.5]{WZ2010} {\blue{(see also \cite[Lemma 3.7, Proposition 3.8]{KW-Covering-dimension})}}.

\begin{prop}
\label{c.p.c._order_zero_map_into_hereditary_subalgebra}
Let $A$ be a $C^*$-algebra, let $n\in{{\mathbb{N}\cup\{0\}}}$, and let
$0<\epsilon<\frac{1}{2^{16}}.$
Let $a_0,a_1\in A_+$ be norm one positive elements.
%
Suppose that $F$  is a finite dimensional $C^*$-algebra,
$\alpha: A\rightarrow F$ is a c.p.c.~map, and
$\beta: F\rightarrow A$ is a
{{piecewise contractive $n$-decomposable}}
c.p.~map.
If $\beta\circ\alpha(a_1)a_0\approx_{\epsilon} \beta\circ\alpha(a_1)$,
then there exist
a $C^*$-subalgebra $\bar{F}\subset F$,
a c.p.c.~map $\bar{\alpha}: A\rightarrow \bar{F}$,
and {\blue{a}}
{{piecewise contractive $n$-decomposable}}
c.p.~map $\bar{\beta}: \bar{F}\rightarrow \Her_A(a_0)$
{\blue{such}} that, for any $x\in A_+$ with $x\leq a_1$,
$\|\beta\circ\alpha(x)-\bar{\beta}\circ\bar{\alpha}(x)\|
\leq 10(n+1)\epsilon^{1/8}.$

\end{prop}

\begin{proof}

Write $F=F_0\oplus\cdots\oplus F_n$ such that {\blue{each}}
$\beta|_{F_i}$ {\blue{is a}} c.p.c.~order zero {\blue{map}}
($i=0,{\green{1,}}\cdots,n$).
Let $\chi(x):[0,1]\rightarrow \{0,1\}$
be the characteristic function of the interval $[\epsilon^{1/2},1]$.
Since $F$ is a finite dimensional \CA,
$p:=\chi(\alpha(a_1))$ is a projection in $F.$
Note that $p\leq \frac{1}{\epsilon^{1/2}}\alpha(a_1)$.
Let $\bar{F}:=pFp.$ {\blue{Then}} $\beta|_{\bar{F}}$ is still a piecewise contractive $n$-decomposable  c.p.~map.
Moreover, for each $i,$ $pF_ip$ is a \SCA\, of $F_i$ with unit $p_i:=p1_{F_i}p.$
Thus
$\beta|_{pF_ip}$ is also a c.p.c.~order zero {\blue{map.}} {\blue{Moreover,}}
\begin{eqnarray*}
\|\beta|_{pF_ip}(p 1_{F_i})(1-a_0)\|^2
=
\|(1-a_0)\beta(p 1_{F_i})^2(1-a_0)\|
\leq
\|(1-a_0)\beta(p)(1-a_0)\|\\
\leq
\frac{1}{\epsilon^{1/2}}\|(1-a_0)\beta(\alpha(a_1))(1-a_0)\|
\leq
\epsilon^{1/2}\leq \frac{1}{2^8}.
\end{eqnarray*}
Then, by \cite[Lemma 3.6]{KW-Covering-dimension},
there exists a c.p.c.~order zero map
$\bar{\beta}_i: pF_ip\rightarrow \Her_A(a_0)$
satisfying
\beq\label{E1-1}
\|\beta|_{pF_ip}(x)-\bar{\beta}_i(x)\|\leq 8 \epsilon^{1/8}
{\blue{\rforal x\in (pF_ip)_+^1.}}
\eneq

Define $\bar{F}:=pFp=pF_0p\oplus\cdots \oplus pF_np$,
and define {\blue{a}} c.p.c.~map $\bar{\alpha}: A\rightarrow \bar{F}$,
$x\mapsto p\alpha(x)p$,
and define {\blue{a}} c.p.\hspace{-0.01in}\,map
$\bar{\beta}: \bar{F}\rightarrow \Her_A(a_0)$,
$x\mapsto \sum_{i=0}^n \bar{\beta}_i(p_ixp_i)$.
Note that $\bar{\beta}|_{pF_ip}=\bar{\beta}_i.$
Thus $\bar{\beta}$ is a piecewise contractive $n$-decomposable  c.p.~map.
{\blue{It follows,}} for $x\in A_+$ with $x\leq a_1{\blue{\le 1}}$,
\beq\label{E1}
\|(1-p)\alpha(x)\|
=
\|(1-p)\alpha(x)^2(1-p)\|^{1/2}
\leq
\|(1-p)\alpha(a_1)(1-p)\|^{1/2}
\leq \epsilon^{1/4}.
\eneq
\noindent
Then
\begin{eqnarray*}
\hspace{-1.2in}\beta\circ\alpha(x)-\bar{\beta}\circ\bar{\alpha}(x)
&=&
\sum_{i=0}^n
\beta(1_{F_i}\alpha(x)1_{F_i})-\bar{\beta}_i(p_i\bar{\alpha}(x)p_i)\\
&=&
\sum_{i=0}^n
\beta(1_{F_i}\alpha(x)1_{F_i})-\bar{\beta}_i(1_{F_i}p\alpha(x)p1_{F_i})\\
(\mbox{by \eqref{E1-1}})&\approx_{8(n+1)\epsilon^{1/8}}&
\sum_{i=0}^n
\beta(1_{F_i}\alpha(x)1_{F_i})-\beta(1_{F_i}p\alpha(x)p1_{F_i})\\
(\mbox{by \eqref{E1} and }\beta|_{F_i}\mbox{ are c.p.c.~maps})
&\approx_{2(n+1)\epsilon^{1/4}}&
\sum_{i=0}^n
\beta(1_{F_i}\alpha(x)1_{F_i})-\beta(1_{F_i}\alpha(x)1_{F_i})=0.
\end{eqnarray*}

\end{proof}



\begin{prop}
\label{refine-of-id-tracially-dimnuc-n-converge}

{\blue{Let $A$ be}} {a} unital simple $C^*$-algebra {\blue{with}} ${\rm{Trdim_{nuc}}} \id_A\le n$
for some integer $n\ge 0.$
Then, for any finite subset $\calF\subset A$,
any $\epsilon>0$ {\blue{and}}
any $a\in A_+\backslash\{0\}$,
there exist c.p.c.~maps
$\phi:A\rightarrow A$ and
$\gamma: A\rightarrow A\cap \phi(A)^{\bot}$,
a finite dimensional $C^*$-algebra $F$,
a c.p.c.~map $\alpha: A\rightarrow F$, and {\blue{a}}
piecewise contractive $n$-decomposable  c.p.~map $\beta: F\rightarrow \Her_A(\phi(1_A))$
such that

(1) $x\approx_{\epsilon} \gamma(x)+\phi(x)$ for all $x\in\calF$,

(2) $\phi(1_A)$ and $\gamma(1_A)$ are projections and $1_A=\gamma(1_A)+\phi(1_A)$,

(3) $\gamma(1_A)\lesssim_A a$,

(4) $\|\phi-\beta\circ\alpha\|\leq \epsilon$, and

(5) $\phi$ is {\blue{an}} $(\calF,\epsilon)$-{{approximate embedding}}.


\end{prop}

\begin{proof}

\Wlog, one may assume that ${\cal F}\subset A^1.$
Let $\delta:=\min\{\frac{1}{2^{10}}, {\blue{(\frac{\epsilon}{10})^{2}}}\}.$
%
{S}ince $\rm{Trdim_{nuc}}\mathrm{id}_A\le n,$
there {exist} a finite dimensional $C^*$-algebra $F_1$,
a c.p.c.~map $\af: A\rightarrow F_1$, {\blue{a}}
piecewise contractive $n$-decomposable  c.p.~map $\beta': F_1\rightarrow A$, and a
c.p.c.~map $\gamma': A\rightarrow A\cap \beta'(F_1)^{\bot}$
such that

{$(1')$} $x\approx_{\delta} \gamma'(x)+\beta'\circ\af (x)$
for all $x\in\calF\cup(\calF\cdot\calF)\cup\{1_A\}$,

{$(2')$} $\gamma'(1_A)\lesssim_A a$, {\blue{and}}

{$(3')$} $\|\beta'\circ\af (x)\|\geq \|x\|-\delta$ for all $x\in\calF$.

\noindent
{Since} {$(1')$} holds for $x=1_A$ and $\gamma'(1_A)\in \beta'(F_1)^{\bot},$
one has
\beq\label{610-n1}
\bt'\circ \af(1_A)\approx_\dt (\bt'\circ {\blue{\af(1_A))^2.}}
\eneq
It follows {from} {{Lemma}}
\ref{perturbation-c.p.c.-to-p.c.p} that there is a  c.p.~map $\phi: A\to A$ such that, $p:=\phi(1_A)$ is
a projection in {\blue{$C^*(\bt'\circ \af(1_A))$}} and
\beq\label{610-n2}
\|\phi(x)-\bt'\circ \af(x)\|\le 5\dt^{1/2}|x\|\rforal x\in A.
\eneq
{\blue{By}} {{$(1')$}}
again,
one has
$\gamma'(1_A)\approx_{\dt} \gamma'(1_A)^2.$
Applying {{Lemma}} \ref{perturbation-c.p.c.-to-p.c.p}
{{again}}, one also obtains a c.p.c.~map
$\gamma: A\to A$ {{such that}} $q:=\gamma(1_A)$ is a projection in $C^*(\gamma'(1_A))$
and
\beq\label{610-n4}
\|\gamma(x)-\gamma'(x)\|\le {\blue{{{5}}\dt^{1/2}}}\|x\|\rforal x\in A.
\eneq
Since $\gamma'(1_A)\bt'\circ \af(1_A)=0,$ it follows
that $qp=0.$  {\blue{By {{$(1')$}},  \eqref{610-n2}, \eqref{610-n4},
{\blue{and the choice of $\dt,$}}
$p+q=1_A.$}}
{\blue{It follows that $\bt'\circ \af(A)\subset pAp.$}}
{\blue{Let $F=\overline{\af(1_A)F_1\af(1_A)}^{{\|\cdot\|}}$ and
$\bt=\bt'|_{F_1}.$ Then $F$ is a finite dimensional
\CA\, and $\bt$ maps $F$ into ${\rm Her}(\phi(1_A))=pAp.$
Note that $\bt$ is also a piecewise contractive $n$-decomposable  c.p.~map.}}

By {\blue{{{$(1')$}},}}  \eqref{610-n2}, and \eqref{610-n4},
and by the choice of $\dt,$ one checks that  (1)  and (4) hold.
Since  $p+q=1_A,$
(2) also holds.
Since $\gamma(1_A)\in C^*({\blue{\gamma'(1_A)}}),$ by  {{$(2')$}}, {\blue{one concludes that}} (3) holds.

By {{$(1')$}}, since the image of $\gamma'$ is in $B\cap \phi(F_1)^\perp,$  one has
\beq\label{610-n10}
\gamma'(x)\gamma'(y)+\bt'\circ\af(x)\bt'\circ \af(y)
&=&(\gamma'(x)+\bt'\circ \af(x))(\gamma'(y)+\bt'\circ \af(y))
\\\nonumber
&\approx_{\dt(1+\dt)} &x(\gamma'(y)+\bt'\circ \af(y))\\\nonumber
&\approx_{\dt}& xy\approx_{\dt} \gamma'(xy)+\bt'\circ \af(xy)\,\,\,\rforal x, y\in {\cal F}.
\eneq
Using the fact that the image of $\gamma'$ is in $B\cap \phi(F_1)^\perp$ again,
one obtains
\beq
\bt'\circ\af(x)\bt'\circ \af(y)
&\approx_{\dt(2+\dt)}&\bt'\circ \af(xy)\rforal x, y\in {\cal F}.
\eneq
In other words,  $\bt'\circ \af$ is $({\cal F}, \dt(2+\dt))$-multiplicative.
By {{\eqref{610-n2} and}} the choice of $\dt,$ one checks that
$\phi$ is $(\calF,\epsilon)${{-multiplicative}}.
Finally, for any $x\in \calF$, by \eqref{610-n2} and {{$(3')$}},
\beq
\|\phi(x)\| \approx_{{{5}}\dt^{1/2}}
\|\beta'\circ\alpha(x)\|=
\|\beta\circ\alpha(x)\|\approx_{\delta}\|x\|.
\eneq
Hence (5) holds.


\end{proof}

\begin{prop}
\label{Tdimnuc-id-map-Tdimnuc}

Let $A$ be a simple unital $C^*$-algebra and let
$n\in {{\mathbb{N}\cup\{0\}}}.$
If $A$ is asymptotically tracially in ${\cal N}_n,$
then ${\rm Trdim_{nuc}}\mathrm{id}_A\le n.$

\end{prop}

\begin{proof}

Let $\calF\subset A^1$ be a finite subset,
let $\epsilon>0$ and let $a\in A_+\backslash\{0\}.$
We may assume that $1_A\in {\cal F}.$
Let $\delta:=\frac{\min\{1,\epsilon\}}{n+5}$.
Since $A$ is asymptotically tracially  in ${\cal N}_n,$
{\blue{by {{Proposition}} \ref{hereditary-subalgebra-preserves-tracially-approximation},}}
there exist a unital $C^*$-algebra $B$ {\blue{with}}  $\dimnuc B\leq n,$
and c.p.c.~maps
$\beta_i: B\rightarrow A$,
u.c.p.~maps
$\alpha': A\rightarrow B$, {\blue{and}}
$\gamma_i: A\rightarrow A\cap\beta_i(B)^{\bot}$
{\blue{($i\in\mathbb{N}$)}} {{such}} that

(1) $x\approx_{\delta}\gamma_i(x)+\beta_i\circ\alpha'(x)$
{{for all}} $x\in\calF$ {{and for all}} $i\in\mathbb{N}$,

(2) $\alpha'$ is {\blue{an}} $(\calF, \delta)$-{{approximate embedding}},

(3) $\lim_{i\rightarrow\infty}\|\beta_i(xy)-\beta_i(x)\beta_i(y)\|=0$ {\blue{and}}
$\lim_{i\rightarrow \infty}\|\beta_i(x)\|= \|x\|$
{{for all}} $x,y\in B$, {{and}}

(4) $\gamma_i(1_A)\lesssim_A a$ {{for all}} $i\in\mathbb{N}$.


\noindent
{{Since}} $\dimnuc B\leq n$,
there exist a finite dimensional $C^*$-algebra $F$ and
a c.p.c.~map $\phi: B\rightarrow F$,
and {\blue{a}} piecewise contractive $n$-decomposable  c.p.~map $\psi: F\rightarrow B$,
such that
\beq
\label{f6-13prop-1}
x\approx_{\delta}\psi\circ\phi(x) {\mbox{ for all }} x\in \alpha'(\calF).
\eneq
\noindent
{{By}} condition (3),
there exists $m\in\mathbb{N}$ {\blue{such that}}
$\|\beta_m\circ\alpha'(x)\|>\|\alpha'(x)\|-\delta$ for all $x\in\calF$
and
$\beta_m\circ\psi:F\rightarrow \Her_A(\beta_m(B))$
is an {\blue{$(n,\Delta(F,\delta))$-dividable}} c.p.~map,
where
{{$\Delta(-,-)$}}
is defined in
Definition \ref{almost-n-dividable-error}.
Then, by the definition of $\Delta(F,\delta)$
and Proposition \ref{n-almost-dividable-exact},
there exists {\blue{a}} piecewise contractive $n$-decomposable  {\blue{c.p.~map}} $\bt:F\rightarrow \Her_A(\beta_m(B))$
such that
\beq
\label{f6-13prop-2}
\|\bt-\beta_m\circ\psi\|\leq \delta.
\eneq
Set $\gamma=\gamma_m$ and $\af=\phi\circ \af'.$
Then, by (1), \eqref{f6-13prop-1}, and  \eqref{f6-13prop-2}, we have
$$
x\approx_{2\dt} \gamma(x)+\bt\circ \af(x)\rforal x\in {\cal F}.
$$
Moreover, $\gamma(A)\perp \bt(F)$ and (by (4)) $\gamma(1_A)\lesssim a.$

It remains to show that $\bt\not=0.$
By (2) and the choice of $m$, we have $\|\beta_m\circ\alpha'(1_A)\|\geq 1-2\dt$.
Then
%
$$
\|\beta\circ\phi\circ\alpha'(1_A)\|
\overset{\eqref{f6-13prop-2}}{\approx_{\delta}}
\|\beta_m\circ\psi\circ\phi\circ\alpha'(1_A)\|
\overset{\eqref{f6-13prop-1}}{\approx_{\delta}}\|\beta_m\circ\alpha'(1_A)\|
\ge 1-2\dt.
$$
Thus $\bt\not=0.$



\end{proof}

The proof of the following proposition is almost the same as
the proof for finite nuclear dimension case,
see \cite[Proposition 2.5]{WZ2010}.

\begin{prop}
\label{hereditary-subalgebra-preserves-map-tracially-approximation}
Let $A$ be a {{simple}} unital $C^*$-algebra
{\blue{with
$\mathrm{Trdim_{nuc}\id_A}\le n$ for some integer $n$}}  and let
$B\subset A$ be a unital hereditary $C^*$-subalgebra.
Then $\mathrm{Trdim_{nuc}}\id_B\le n.$

\end{prop}

\begin{proof}

Let $\calF\subset B_+^1$ be a finite subset with $1_B\in\calF$,
let $\epsilon>0$ and let  $b\in B_+\backslash\{0\}$.
Choose $\eta>0$ such that
\beq
((1+\eta)\eta)^{1/2}<1/2^{16}
\mbox{\quad and \quad}
10(n+1)((1+\eta)\eta)^{1/16}+2\eta^{1/2}<\ep.
\eneq

\noindent
{{Since}} ${\mathrm{Trdim_{nuc}}}\id_A\le n,$
there exist a finite dimensional $C^*$-algebra $F$,
a c.p.c.~map $\alpha: A\rightarrow F$,
{{a}}
piecewise contractive $n$-decomposable  c.p.~map $\beta: F\rightarrow A$, and a
c.p.c.~map $\gamma: A\rightarrow A\cap \beta(F)^{\bot}$
such that

(1) $x\approx_{\eta} \gamma(x)+\beta\circ\alpha(x)$ {{for all}} $x\in\calF
$,

(2) $\gamma(1_A)\lesssim_A b$, {{and}}

(3) $\|\beta\circ\alpha(x)\|\geq \|x\|-\eta$ {{for all}} $x\in\calF
$\,\,\,
(see Remark \ref{69-n1126}).

\noindent
{{Since}} $\gamma(A)\perp \bt(F),$
by (1),
{{$\|\bt\circ \af(1_{{B}})\|\le (1+\eta).$}}
It follows that
$\bt\circ \af(1_{{B}})^2\leq(1+\eta)\bt\circ \af(1_{{B}}).$
{{Therefore}}
\beq
\hspace{-0.3in}\|(1_A-1_B)\beta\circ\alpha(1_B)\|^2
&=&\|(1_A-1_B)\beta\circ\alpha(1_B)^2(1_A-1_B)\|\\
&&\leq (1+\eta)
\|(1_A-1_B)\beta\circ\alpha(1_B)(1_A-1_B)\|\\
&&\leq
(1+\eta)\|(1_A-1_B)(\beta\circ\alpha(1_B)+\gamma(1_B))(1_A-1_B)\|\\
&&\leq
(1+\eta)(\|(1_A-1_B)1_B(1_A-1_B)\|+\eta)
=
(1+\eta)\eta.
\eneq
Since $\gamma$ is a c.p.c.~map, a similar but simpler estimate shows that
\beq\label{614-n1126+1}
\|(1_A-1_B)\gamma(x)\|^2\le \eta\rforal x\in {\cal F}.
\eneq
By the choice of $\eta$ {\blue{and}}
by Proposition \ref{c.p.c._order_zero_map_into_hereditary_subalgebra}
(letting $a_0=a_1=1_B$),
there exists a $C^*$-subalgebra $\bar{F}\subset F$,
{\blue{a}} c.p.c.~map $\bar{\alpha}: A\rightarrow \bar{F}$,
and {\blue{a  piecewise contractive $n$-decomposable}}
c.p.~map $\bar{\beta}: \bar{F}\rightarrow \Her_A(1_B)=B$
{{such that,}} for any
$x\in {{B_+^1}}$,
\beq\label{f-7-13-1}
\|\beta\circ\alpha(x)-\bar{\beta}\bar{\alpha}(x)\|
\leq 10(n+1)((1+\eta)\eta)^{1/16}.
\eneq

By {\green{\eqref{f-7-13-1} and}} (3) {{and the choice of $\eta$,}}
for $x\in\calF,$ we have
\beq\nonumber
\|\bar{\beta}\bar{\alpha}(x)\|
{\green{\ \geq \|x\|-\epsilon.}}
\eneq

\noindent
{{Define}} a c.p.c.~map
$\bar{\gamma}: B\rightarrow B$, $x\mapsto 1_B\gamma(x)1_B$.
{{Then}} $\bar{\gamma}(1_B)\lesssim_A \gamma(1_B)\leq \gamma(1_A)\lesssim_A b.$
Since $B$ is hereditary $C^*$-subalgebra of $A$,
we have $\bar{\gamma}(1_B)\lesssim_B b$.

{{Finally, for}} $x\in \calF$,  by \eqref{614-n1126+1},
$\gamma(x)\approx_{2\eta^{1/2}}1_B\gamma(x)1_B=\bar{\gamma}(x)$
 for all $x\in\calF$.
Therefore
$$
x
\approx_{\eta}
\gamma(x)+\beta\circ\alpha(x)
\approx_{10(n+1)((1+\eta)\eta)^{1/16}+2\eta^{1/2}}
\bar{\gamma}(x)+\bar{\beta}\bar{\alpha}(x)\rforal x\in {\cal F}.
$$
Note that $10(n+1)((1+\eta)\eta)^{1/16}+2\eta^{1/2}<\ep.$ It follows that
${\rm Trdim_{nuc}}\id_B\le n.$

\end{proof}




\begin{prop}
[{cf. \cite[Proposition 3.4]{F-2019}}]
\label{fullness neighbourhood and simplicity}
Let $A$ be a unital $C^*$-algebra and let
$X\subset A_+$ {{be a finite subset.}}
Suppose that, {\blue{for}} each $x\in X$,  $f_{1/2}(x)$ is  full in $A.$
Then, there exist ${{\sigma({\cal G}_X)}}>0$ and a finite subset ${\cal G}_X\subset A$
such that, for any
{{unital \CA\ $B$ and any}}
u.c.p.~map $\psi:A\rightarrow B$ which is
$({\cal G}_X, {{\sigma({\cal G}_X)}})${{-multiplicative}}, 
$f_{1/2}(\psi(x))$ is a full element of $B$ for each $x\in X$.

\end{prop}


The following lemma is a construction of simple generalized
inductive {\green{limit}} of $C^*$-algebras.

\begin{lem}
\label{construction-simple-gen-lim}

Let $\{A_i\}$ be {\blue{a sequence of}} unital separable $C^*$-algebras
and let $\phi_i:A_i\rightarrow A_{i+1}$ be  $\mbox{u.c.p.~maps}$
{{$(i\in\mathbb{N})$.}}
Let $X_i=\{x_{i,1},{\blue{x_{i,2},}}\cdots\}\subset A_{i+}^1$
be a countable dense subset of $A_{i+}^1$,
$X_{i,k}:=\{x_{i,1}\cdots,x_{i,k}\}$
($i,k\in\mathbb{N}$),
and $Y_k:=\cup_{1\leq i\leq k}\phi_{i,k}(X_{i,k})$.
Then $(A_i,\phi_i)$ forms a generalized inductive
{{system}} 
and
{{$\lim_{i}(A_i,\phi_i)$}}
is simple,
if the following hold for any $k\in \N$:

(1)
$f_{1/2}(a)$ is full in $A_k$ for all
$a\in  {\cal F}_k:=Y_k\cap B_{{\blue{\frac{3}{4}}},1}(A_k)$
{\blue{(recall {\green{Notation}} \ref{DB1}),}}
{\green{and}}

(2) 
$\phi_{k}$ is $\epsilon_{k}${{-multiplicative}} on
$
Y_{k}
\cup \left(\cup_{1\leq j\leq k}\phi_{j,k}({\cal G}_{{\cal F}_k})\right),
$
where
$$
\epsilon_{k}: =
\frac{1}{4^k}
\min_{1\le j\le k}
\{{{1, \sigma({\cal G}_{{\cal F}_j})}}\}
$$
(see {\green{Proposition}}
\ref{fullness neighbourhood and simplicity}
for ${\cal G}_{{\cal F}_j}$ and $\sigma({\cal G}_{{\cal F}_j})$,
{{see {\green{Notation}} \ref{def-of-gen-ind-lim} for $\phi_{j,k}$}}).

\end{lem}

\begin{proof}

First we show that $(A_i,\phi_i)$ forms a generalized inductive limit.
Let $k\in\mathbb{N}\cup\{0\}$, $y_1,y_2\in A_{k+}^1\backslash\{0\}$ {\blue{and}}  $\epsilon>0$.
%
%
Then there {\blue{exist}}
$t_1,t_2\in\mathbb{N}$ such that
$y_1\approx_{\epsilon/4} x_{k, t_1}$ {\blue{and}}
$y_2\approx_{\epsilon/4} x_{k, {{t_2}}}$.
Note that $\sum_{i=1}^{\infty}\epsilon_i<\infty.$ Thus
there is $m>\max\{k, t_1,t_2\}$ such that
$\sum_{i=m}^{\infty}\epsilon_i<\epsilon/4$.
Then, for all
$j> m$,
by the choice of $Y_j$, we have
$\phi_{k,j}(x_{k,t_1})$, $\phi_{k,j}(x_{k,{{t_2}}})\in Y_j$.
By (2),  for all $i\geq m$, $\phi_i$ is $\epsilon_i${{-multiplicative}} on
$\{\phi_{k,i}(x_{k,t_1}),\phi_{k,i}(x_{k, {{t_2}}})\}.$
Hence $\phi_{m,j}$ is
$\sum_{i=m}^{j-1}\epsilon_i${{-multiplicative}} on
$\{\phi_{k,m}(x_{k,t_1}),\phi_{k,m}(x_{k, {{t_2}}})\}$.
Then, for all $j\geq m$,
\begin{eqnarray*}
\phi_{k,j}(y_1)\cdot\phi_{k,j}(y_2)
&=&
\phi_{m,j}(\phi_{k,m}(y_1))\cdot\phi_{m,j}(\phi_{k,m}(y_2))\\
&\approx_{\epsilon/4}&
\phi_{m,j}(\phi_{k,m}(x_{k,t_1}))\cdot\phi_{m,j}(\phi_{k,m}(x_{k,t_2}))\\
&\approx_{\sum_{i=m}^{j-1}\epsilon_i}&
\phi_{m,j}(\phi_{k,m}(x_{k,t_1})\cdot\phi_{k,m}(x_{k,t_2}))\\
&\approx_{\epsilon/4}&
\phi_{m,j}(\phi_{k,m}(y_1)\cdot\phi_{k,m}(y_2)).
\end{eqnarray*}

\noindent
{{By}} the choice of $m$, we have
$\phi_{k,j}(y_1)\cdot\phi_{k,j}(y_2)
\approx_{\epsilon}
\phi_{m,j}(\phi_{k,m}(y_1)\cdot\phi_{k,m}(y_2))$
for all $j\geq m$.
{\blue{By}} Lemma \ref{positive-criterion-being-gen-ind-lim},
$(A_i,\phi_i)$ forms a generalized inductive
{\blue{system.}}



Now we show that $A:=\lim_{i}(A_i,\phi_i)$ is simple.
It {\blue{suffices}}
to show {\blue{that}} every norm one positive element of $A$ is full.
Let $a\in A_+$ with $\|a\|=1$.
Then there exist $k,s\in \mathbb{N}$   such that
$\|a-\phi_{k,\infty}(x_{k,s})\|<1/4$.
Let $r>\max\{k,s\}$ {\blue{be}} such that $\|\phi_{k,r}(x_{k,s})\|\geq 3/4.$ {\blue{Then}} we have
$\phi_{k,r}(x_{k,s})\in \calF_r:=Y_r\cap B_{\frac{3}{4},1}(A_r)$.
Condition (2) shows that, for all $ j> r$,
$\phi_{r,j}$ is
$\sum_{i=r}^{j-1}\epsilon_i${{-multiplicative}}
on ${\cal G}_{{\cal F}_r}.$
By the choice of $\epsilon_i$ ($i\in\mathbb{N}$),  {\blue{the map}}
$\phi_{r,\infty}$ is
$\sigma({\cal G}_{{\cal F}_r})${{-multiplicative}}
on ${\cal G}_{{\cal F}_r}.$
%
%
Then, by Proposition \ref{fullness neighbourhood and simplicity},
$f_{1/2}(\phi_{k,\infty}(x_{k,s}))
=f_{1/2}(\phi_{r,\infty}(\phi_{k,r}(x_{k,s})))$
is a full element of $A$.
Since $\|a-\phi_{k,\infty}(x_{k,s})\|<1/4$,
by \cite[Proposition 2.2]{Rordam-1992-UHF2},
$f_{1/2}(\phi_{k,\infty}(x_{k,s}))=c^*ac$ for some $c\in A.$
Thus $a$ is also a full element of $A$.
Since $a$ is arbitrary, {\blue{so $A$ is simple.}}

\end{proof}

The following is a construction of
simple separable unital finite nuclear dimension
$C^*$-algebras using generalized inductive limits.

\begin{lem}
\label{construction-simple-finite-dimnuc}

Let $n\in{{\mathbb{N}\cup\{0\}}}$.
Let {{$\{A_i\}$}} be {{a sequence of}}
unital separable $C^*$-algebras
and $\phi_i:A_i\rightarrow A_{i+1}$ {\blue{be}} u.c.p.~maps
{{($i\in\N$)}}.
Let $X_i=\{{x_{i,1},} x_{i,2},
\cdots\}\subset A_{i+}^1$
be a countable dense subset of $A_{i+}^1$,
let $X_{i,k}:=\{x_{i,1},  {\green{x_{i,2},}}
\cdots,x_{i,k}\}$,
and {\blue{let}}
$Y_k:=\cup_{1\leq j\leq k}\phi_{j,k}(X_{j,k})$ ($i,k\in\mathbb{N}$).
{{Let $F_0=\C$ {\blue{and}} let $\beta_0: F_0\rightarrow A_1$ be the zero map.}}
Then $(A_i,\phi_i)$ forms a generalized inductive limit
and $A:=\lim_{i}(A_i,\phi_i)$ is simple {\blue{with}} $\dimnuc A\leq n$,
if the following hold for {{all}} $k\in\mathbb{N}$:

(1) For all
$a\in\calF_k:=Y_k\cap B_{{\blue{\frac{3}{4}}},1}(A_k)$,
$f_{1/2}(a)$ is full in $A_k$,

(2) there exist a finite dimensional $C^*$-algebra $F_k$,
{{a}} c.p.c.~map $\alpha_k: A_k\rightarrow F_k$,
and {{a}} piecewise contractive $n$-decomposable  c.p.~map
$\beta_k: F_k\rightarrow A_{k+1}$
{\blue{such}} that
$\phi_k(x)\approx_{\frac{1}{k}}
\beta_k\circ\alpha_k(x)$ for all $x\in Y_k$,
{\green{and}}

(3) 
$\phi_{k}$ is $\epsilon_{k}${{-multiplicative}} on
$$
Y_{k}
\cup \left(\cup_{1\leq j\leq k}\phi_{j,k}({\cal G}_{{\cal F}_j})\right)
\cup \left(\cup_{1\leq j\leq k}\phi_{j,k}(\beta_{j-1}({\cal G}^{F_{j-1}}))\right),
$$
where
$$
\epsilon_{k}:=\frac{1}{4^{k}}
\min_{1\leq j\leq k}
\{{{1, \sigma({\cal G}_{{\cal F}_j})}},
\Delta(F_{j-1},\frac{1}{j})\}\andeqn
$$
${\cal G}^{{{F_{j-1}}}}$
is {{the standard generating set}} of  $F_{j-1}$ in $F_{j-1}^1$
(see {{Proposition}}
\ref{fullness neighbourhood and simplicity}
for ${\cal G}_{{\cal F}_j}$ and $\sigma({\cal G}_{{\cal F}_j})$,
{\blue{see}} {{Definition}}
\ref{almost-n-dividable-error} for
$\Delta(-,-)$, {\blue{and}}
{{see {{Notation}} \ref{def-of-gen-ind-lim} for $\phi_{j,k}$}}{{).}}

\end{lem}

\begin{proof}

By Lemma \ref{construction-simple-gen-lim},
$(A_i,\phi_i)$ forms a generalized inductive {\blue{system}}
and $A:=\lim_{i\rightarrow \infty}(A_i,\phi_i)$ is {\blue{a simple \CA.}}

{\blue{To show  $\dimnuc A\leq n$,}}
{\blue{let}} $i\in\mathbb{N}$, $\epsilon>0,$ {\blue{and}}
$\calF \subset A^1_{i+}$ {{be a finite subset}}.
%
By the {\blue{definition}} of $Y_j$,
there exists $m\geq i+1+\frac{4}{\epsilon}$ such that
$\phi_{i,m}(\calF)\subset_{\epsilon/4}Y_m$.

{\blue


By (3), $\phi_{m+1,\infty}$ is
$\sum_{j=m+1}^{\infty}\epsilon_j${{-multiplicative}} on $\beta_m({\cal G}^{F_m^1})$.
By the choice of $\epsilon_j,$ {\blue{one has}}
$\sum_{j=m+1}^{\infty}\epsilon_j\leq \Delta(F_m,\frac{1}{m})$.
{\blue{Then
$\phi_{m+1,\infty}\circ\beta_m$ is {\blue{an}}
$(n,{{\delta_1(\frac{1}{m})}})$-dividable map}}
(see {{Definition \ref{almost-n-dividable-error} and}}
part (2) of {{Proposition}}
\ref{n-almost-dividable-exact}).
By 
Proposition \ref{n-almost-dividable-exact},
there exists a piecewise contractive $n$-decomposable  c.p.~map
$\beta:
F_{m}\rightarrow A$ such that
\beq
\label{f6-18lem-1}
\|\beta-\phi_{m+1,\infty}\circ\beta_{m}\| \leq 1/m.
\eneq
For any $x\in \calF$,
there exists $y\in Y_m$ {\blue{such}} that
$\phi_{i,m}(x)\approx_{\epsilon/4}y.$
{\blue{Then}}
\beq
\phi_{i,\infty}(x)
&=&
\phi_{m,\infty}(\phi_{i,m}(x))
\approx_{\frac{\epsilon}{4}}
\phi_{m,\infty}(y)
\overset{(2)}{\approx_{\frac{\epsilon}{4}}}
\phi_{m+1,\infty}\circ\beta_m\circ\alpha_m(y)
\overset{\eqref{f6-18lem-1}}{\approx_{\frac{\epsilon}{4}}}
\beta\circ\alpha_m(y)
\nonumber\\
&\approx_{\frac{\epsilon}{4}}&
\beta\circ\alpha_m\circ\phi_{i,m}(x).
\eneq

\noindent
{\blue{Then,}} by Theorem \ref{gene-lim-fin-nucdim}
{{(with $\alpha_m\circ\phi_{i,m}$ in place of $\alpha$)}}, $\dimnuc A\leq n$.

}
\end{proof}

\begin{thm}
\label{simple-T-finite-dimnuc-implies-T-simple-finite-dimnuc}
Let $n\in{{\mathbb{N}\cup\{0\}}}$.
Let $A$ be a simple separable unital infinite dimensional $C^*$-algebra
and ${\rm Trdim_{nuc}}\id_A\le n.$
Then $A$ is asymptotically tracially in
${\cal N}_{n,s,s}$ (recall {{Definition}} \ref{DNn} for the class ${\cal N}_{n,s,s}$).

\end{thm}

\begin{proof}

Let $\calF\subset B_{\frac{3}{4},1}(A_+)$ be a finite subset
{{with $1_A\in\calF,$}}
{\blue{let}} $\epsilon\in(0,1),$ and let
$a\in A_+\backslash\{0\}$ with $\|a\|=1$.
Since $A$ is {{simple, unital}} {\blue{and}} infinite dimensional,
{\blue{$A$ is}}  non-elementary.
{\blue{Thus}} there {\blue{exist}}
{{a sequence of}} norm one positive elements
$a_0,a_1,{\blue{\cdots, a_n,\cdots}}$ {\blue{in}} $\Her_A(f_{1/2}(a))_+\backslash\{0\}$ such that
$a_i\bot a_j$, $i\neq j$ {\blue{(see Lemma \ref{L1991}).}}

Let $A_0:=A.$ {{Let $\calF_0:=\calF$}}  {\blue{and let}}
$\epsilon_0:=\epsilon/8$.
Since ${\blue{\rm{Trdim}_{nuc}}\mathrm{id}_{A_0}\le n,}$
by Proposition \ref{refine-of-id-tracially-dimnuc-n-converge},
there {\blue{exist}} {{two}} c.p.c.~maps
$\phi_0:A_0\rightarrow {{A_0}}$,
$\gamma_0: A_0\rightarrow A_0\cap \phi_0(A_0)^{\bot}$,
and a finite dimensional $C^*$-algebra $F_0$,
and  {\blue{a}} c.p.c.~map $\alpha_0: A_0\rightarrow F_0,$ {\blue{and a}}
piecewise contractive $n$-decomposable  c.p.~map $\beta_0: F_0\rightarrow \Her_{A_0}(\phi_0(1_{A_0}))$
{\blue{such}} that

(0,1) $x\approx_{\epsilon_0} \gamma_0(x)+\phi_0(x)$ {{for all}}
$x\in{{\calF_0}}$,

(0,2) $\phi_0(1_{A_0})$ {\blue{and}} $\gamma_0(1_{A_0})$ are projections, and $1_{A_0}=\gamma_0(1_{A_0})+\phi_0(1_{A_0})$,

(0,3) $\gamma_0(1_{A_0})\lesssim_{A_0} a_0$.

(0,4) $\|\phi_0-\beta_0\circ\alpha_0\|\leq \epsilon_0$, {\blue{and}}

(0,5) $\phi_0$ is {\blue{an}} $(\calF_0,\epsilon_0)$-{{approximate embedding}}.


\noindent
{{Define}} $A_1:=\Her_{A_0}(\phi_{0}(1_{A_0}))$.
Note that $A_1$ is {\blue{a}} simple separable unital non-elementary \CA,
and there exists $\bar{a}_1\in A_{1+}\backslash\{0\}$
such {\blue{that}} $\bar{a}_1\lesssim_A a_1$.
There exists a norm one c.p.c.~order zero map
$\chi_1: M_1=\C\rightarrow A_1$.
Let $Z_1\subset \chi_1(M_1^1)$ be {{a finite subset {\blue{which}} is}}
a {{$\frac{1}{4}\Delta(M_1,1)$}}-net of $\chi_1(M_1^1)$.

Let $X_1=\{x_{1,1},{\blue{x_{1,2},}}\cdots\}
\subset A_{1+}^1$
be a countable dense subset of $A_{1+}^1$ {\blue{and}}
let $X_{1,k}:=\{x_{1,j}:1\leq j\leq k\}$, $k\in\mathbb{N}.$
{\blue{Set}} $Y_1:=\cup_{1\leq i\leq 1}\phi_{i,1}(X_{i,1}){\blue{=X_{1,1}}}$  (with
$\phi_{1,1}=\id_{A_1}$),
$\bar{Z}_1:=Z_1,$ {\blue{and}} $Y_1'=Y_1\cap B_{\frac{3}{4},1}(A_1).$
{\blue{Note $f_{1/2}(b)\not=0$ and (since $A$ is simple) therefore is full in $A_1$ for all $b\in Y_1'.$
Let ${\cal G}_1:=\mathcal{G}_{Y_1'}$  and ${{\sigma({\cal G}_{Y_1'})}}$ be  as in {{Proposition}}
\ref{fullness neighbourhood and simplicity} associated
with the set $Y_1'$ (in place of $X$).}}
Define
$$
\calF_1:=\phi_0(\calF_0)
\cup Y_1
\cup \mathcal{G}_1
\cup \bt_0({\cal G}^{F_0})
\cup \bar{Z}_1\,\,\,{{\rm and}}
$$
%
$$
{\blue{\epsilon_1:=\frac{1}{4}\min\{{{\sigma({\cal G}_{Y_1'})}}, \Delta({\blue{F_0}},1), \Delta(M_1,1),\epsilon/4, \}.}}
$$

\noindent
{{By}} Proposition
\ref{hereditary-subalgebra-preserves-map-tracially-approximation},
${\blue{{\rm{Trdim_{nuc}}}\mathrm{id}_{A_1}\le n.}}$
{\blue{By}} Proposition \ref{refine-of-id-tracially-dimnuc-n-converge},
there exist {{two}} c.p.c.~maps
$\phi_1:A_1\rightarrow A_1$,
$\gamma_1: A_1\rightarrow A_1\cap \phi_1(A_1)^{\bot}$,
%
{{a}} finite dimensional $C^*$-algebra $F_1$,
{\blue{a}} c.p.c.~map $\alpha_1: A_1\rightarrow F_1$,
{{and a}}
piecewise contractive $n$-decomposable  c.p.~map $\beta_1: F_1\rightarrow \Her_{A_1}(\phi_1(1_{A_1}))$
{\blue{such}} that

(1,1) $x\approx_{\epsilon_1} \gamma_1(x)+\phi_1(x)$
{{for all}} $x\in\calF_1$,

(1,2) $\phi_1(1_{A_1})$ {\blue{and}} $\gamma_1(1_{A_1})$ are projections and $1_{A_1}=\gamma_1(1_{A_1})+\phi_1(1_{A_1})$,

(1,3) $\gamma_1(1_{A_1})\lesssim_{A_1} \bar{a}_1$,

(1,4) $\|\phi_1-\beta_1\circ\alpha_1\|\leq \epsilon_1$, {\blue{and}}

(1,5) $\phi_1$ is {\blue{an}} $(\calF_1,\epsilon_1)$-{{approximate embedding}}.


\noindent
{{Assume}} that, for $1\leq k \in \mathbb{N},$
we {\blue{have constructed, for each ${{1}}\le j\le k,$}}
{\blue{a}} hereditary $C^*$-subalgebra $A_j:=\Her_{A_{j-1}}(\phi_{j-1}(1_{A_{j-1}}))\subset A$,
$\bar{a}_j\in A_{j+}\backslash\{0\}$ with $\bar{a}_j\lesssim_A a_j$,
and $X_j=\{x_{j,1},x_{j,2},\cdots\}\subset A_{j+}^1,$
{\blue{$Y_j:=\cup_{1\leq i\leq j}\phi_{i,j}(X_{i,j})$}}
{\blue{
({{see {{Notation}}
\ref{def-of-gen-ind-lim} for $\phi_{j,k}$}}),
$Y_j':=Y_j\cap B_{{\frac{3}{4}}, 1}(A_j),$
{\blue{${{\sigma({\cal G}_{Y_j'})}}>0$ and ${\cal G}_j:={\cal G}_{Y_j'}$
as in  {{Proposition}}
\ref{fullness neighbourhood and simplicity} associated with $Y_j'$
(in place of $X$), }}
{{a finite subset $Z_j\subset \chi_{k+1}(M_{k+1}^1)$ which is
a}} $\frac{1}{4}\Delta(M_j, 1/j)$-net of $\chi_j(M_j^1),$
${\bar{Z_j}}:=\cup_{1\le i\le j}\phi_{i,j}(Z_i),$
and a norm one c.p.c.~order zero map
$\chi_j: M_j\rightarrow A_j$,
a  finite subset}}
{\blue{
\beq\label{719-n14}
\calF_j:=\phi_{0,j}(\calF_0)\cup Y_{j}
\cup (\cup_{1\le i\le j}\phi_{i,j}(\mathcal{G}_{i}))\cup (\cup_{1\le i\le j}\phi_{i,j}(\bt_{i-1}({\cal G}^{F_{i-1}})))
\cup\bar{Z}_{j}\subset A_{j},
\eneq}}
and
$$
{\blue{\epsilon_j=\frac{1}{4^{j}}\min_{1\leq i\leq j}
\{{{\sigma({\cal G}_{Y_i'})}},
\Delta(F_{i-1},\frac{1}{i}),
\Delta(M_{i},\frac{1}{i}),\epsilon/4\}>0, \,\,1\le j\le k\,\, (\hspace{-0.06in}\andeqn \ep_0=\ep/8),}}
$$
and there {\blue{exist}} {{two}} c.p.c.~maps
$\phi_j:A_j\rightarrow A_j$,
$\gamma_j: A_j\rightarrow A_j\cap \phi_j(A_j)^{\bot}$,
a finite dimensional $C^*$-algebra $F_j$,
{\blue{a}} c.p.c.~map $\alpha_j: A_j\rightarrow F_j$, {\blue{and}} a
piecewise contractive $n$-decomposable  c.p.~map $\beta_j: F_j\rightarrow \Her_{A_j}(\phi_j(1_{A_j}))$
{\blue{such}} that

($j,1$) $x\approx_{\epsilon_j} \gamma_j(x)+\phi_j(x)$
{{for all}} $x\in\calF_j$,

($j,2$) $\phi_j(1_{A_j})$ {\blue{and}}  $\gamma_j(1_{A_j})$ are projections and $1_{A_j}=\gamma_j(1_{A_j})+\phi_j(1_{A_j})$,

($j,3$) $\gamma_j(1_{A_j})\lesssim_{A_j} \bar{a}_j$,

($j,4$) $\|\phi_j-\beta_j\circ\alpha_j\|\leq \epsilon_j$, {\blue{and}}

($j,5$) $\phi_j$ is {\blue{an}} $(\calF_j,\epsilon_j)$-{{approximate embedding}}.



\noindent
{{Define}} $A_{k+1}:=\Her_{A_k}(\phi_k(1_{A_{k}}))$.
Note that there exists $\bar{a}_{k+1}\in (A_{k+1})_+\backslash\{0\}$
such that $\bar{a}_{k+1}\lesssim_A a_{k+1}$.
Also note that $A_{k+1}$ is simple, separable, unital  and non-elementary.
{\blue{Then,}} {{by \cite[Proposition 4.10]{KR2000}}},
there exists a norm one c.p.c.~order zero map
$\chi_{k+1}: M_{k+1}\rightarrow A_{k+1}$.
Let {{${Z}_{k+1}\subset \chi_{k+1}(M_{k+1}^1)$ be a finite subset which is a}}
{{$\frac{1}{4} \Delta(M_{k+1},\frac{1}{k+1})$}}-net
of $\chi_{k+1}(M_{k+1}^1)$.
Let $X_{k+1}=\{x_{k+1,1},{\blue{x_{k+1,2},}}\cdots\}
\subset (A_{k+1})_+^1$
be a countable dense subset of $(A_{k+1})_+^1$, {\blue{and}}
let $X_{k+1,i}:=\{x_{k+1,j}:1\leq j\leq i\}$, $i\in\mathbb{N}$.
Let $Y_{k+1}:=\cup_{1\leq j\leq k+1}\phi_{j,k+1}(X_{j,k+1})$ {\blue{and}}
$\bar{Z}_{k+1}:=\cup_{1\leq j\leq k+1}\phi_{j,k+1}(Z_j)$.
{\blue{Note that $f_{1/2}(b)$ is full in $A_k.$
Set $Y_{k+1}':=Y_{k+1}\cap B_{\frac{3}{4},1}(A_{k+1}),$
${\blue{\sigma({\cal G}_{Y_{k+1}'})>0,}}$
and finite subset ${\cal G}_{k+1}:={\cal G}_{Y_{k+1}'}$
be as in  {{Proposition}}
\ref{fullness neighbourhood and simplicity} associated with $Y_{k+1}'$ (in place of $X$).}}
{{Define}}
\beq\label{719-n1}
&&\hspace{-0.5in}\calF_{k+1}:=\phi_{0,k+1}(\calF_0)
\cup Y_{k+1}
\cup (\cup_{1\le i\le k+1}\phi_{i,{{k+1}}}(\mathcal{G}_i))
\cup (\cup_{1\le i\le k+1}\phi_{i,{{k+1}}}({{\beta_{i-1}}}({\cal G}^{F_{i-1}})))
\cup\bar{Z}_{k+1}
\nonumber
\\\label{719-n2}
&&\andeqn\epsilon_{k+1}:=
\frac{1}{4^{k+1}}\min_{1\leq j\leq k+1} \{{\blue{
{{\sigma({\cal G}_{Y_{k+1}'})}},}}
\Delta(F_{j-1},\frac{1}{j}),
\Delta(M_{j},\frac{1}{j}),\epsilon/4\}>0.
\eneq	
(Note ${\cal F}_{k+1}$ is a finite set.)

Note  {\blue{also}} ${\blue{{\rm{Trdim_{nuc}}}\mathrm{id}_{A_{k+1}}\le n}}$ (by
{{Proposition}}
\ref{hereditary-subalgebra-preserves-map-tracially-approximation}).
{\blue{Then,}} by Proposition \ref{refine-of-id-tracially-dimnuc-n-converge},
there {\blue{exist}} {{two}} {{c.p.c.~maps}}
$\phi_{k+1}:A_{k+1}\rightarrow A_{k+1}$,
$\gamma_{k+1}: A_{k+1}\rightarrow A_{k+1}\cap \phi_{k+1}(A)^{\bot}$,
{\blue{a}} finite dimensional $C^*$-algebra $F_{k+1}$,
and {\blue{a}} c.p.c.~map $\alpha_{k+1}: A_{k+1}\rightarrow F_{k+1}$, {\blue{a}}
piecewise contractive $n$-decomposable  c.p.~map $\beta_{k+1}: F_{k+1}\rightarrow \Her_{A_{k+1}}(\phi_{k+1}(1_{A_{k+1}}))$
{\blue{such}} that

($k+1,1$) $x\approx_{\epsilon_{k+1}} \gamma_{k+1}(x)+\phi_{k+1}(x)$
{{for all}} $x\in\calF_{k+1}$,

($k+1,2$) $\phi_{k+1}(1_{A_{k+1}})$ {\blue{and}}
$\gamma_{k+1}(1_{A_{k+1}})$ are {\blue{projections,}} and
$1_{A_{k+1}}=\gamma_{k+1}(1_{A_{k+1}})+\phi_{k+1}(1_{A_{k+1}})$,

($k+1,3$) $\gamma_{k+1}(1_{A_{k+1}})\lesssim_{A_{k+1}} \bar{a}_{k+1}$,

($k+1,4$) $\|\phi_{k+1}-\beta_{k+1}\circ\alpha_{k+1}\|\leq \epsilon_{k+1}$, {\blue{and}}

($k+1,5$) $\phi_{k+1}$ is {\blue{an}}
$(\calF_{k+1},\epsilon_{k+1})$-{{approximate embedding}}.

\noindent
{{Then,}} by induction, for each $k\in \mathbb{N}$,
we obtain {\blue{a}}
hereditary $C^*$-subalgebra $A_k\subset A$,
$\bar{a}_k\in A_{k+}\backslash\{0\}$ with $\bar{a}_k\lesssim_A a_k$,
{\blue{a}} norm one c.p.c.~order zero map
$\chi_k: M_k\rightarrow A_k$,
{\blue{a}} finite subset
$\calF_k\subset A_k$ {\blue{satisfying \eqref{719-n1},}}
and
$\epsilon_k>0$ {\blue{satisfying \eqref{719-n2},}}
{\blue{and,}} there exist {{two}} c.p.c.~maps
$\phi_k:A_k\rightarrow A_k$,
$\gamma_k: A_k\rightarrow A_k\cap \phi_k(A_k)^{\bot}$,
{\blue{a}} finite dimensional $C^*$-algebra $F_k$,
{\blue{a}} {\blue{c.p.c.~map}} $\alpha_k: A_k\rightarrow F_k$,  {\blue{and}}
{{a}} piecewise contractive $n$-decomposable {\blue{c.p.~map}} $\beta_k: F_k\rightarrow \Her_{A_k}(\phi_k(1_{A_k}))$
such {\blue{that conditions}} ($k,1$)  {\blue{to}}  ($k,5$) hold.


By Lemma \ref{construction-simple-finite-dimnuc}
(see  (k+1, 4) {{and (k+1, 5)}}),
$(A_k,\phi_k)$
{\blue{forms}} a generalized inductive {\blue{system}} and
${\bar A}:=\lim_k(A_k, \phi_k)$
 is  {\blue{a}} simple separable unital {\blue{\CA\, which}}
has nuclear dimension 
{{at most}} $n$.

{\blue{Let us now}} show  {\blue{that}} $\bar{A}$ is infinite dimensional.
For $4\leq k \in\mathbb{N}$ and {\blue{for all}} $m\geq k$,
by ($m,5$) and the choice of $\bar{Z}_m$ and $\epsilon_m$,  {\blue{the map}}
$\phi_{k,\infty}$ is
$\frac{1}{4}\Delta(M_k,\frac{1}{k})${{-multiplicative}} on
$Z_k$.
{\blue{Since $Z_k$  is $\frac{1}{4}\Delta(M_k, {\frac{1}{k}})$-net of $\chi_k(M_k^1),$}}
{\blue{the composition}}
$\phi_{k,\infty}\circ\chi_k:
M_k\rightarrow \bar{A}$
is $\Delta(M_k,\frac{1}{k})$-almost order zero.
Then,  by Proposition \ref{n-almost-dividable-exact},
and the definition of $\Delta(M_k,\frac{1}{k})$,
there exists  {\blue{a}} c.p.c.~order zero map
$\bar{\chi}_k:M_k\rightarrow \bar{A}$
such that $\|\bar{\chi}_k-\phi_{k,\infty}\circ\chi_k\|\leq \frac{1}{k}$.
By ($m,5$), for $m\geq k$, we have
$\|\phi_{k,\infty}\circ\chi_k(1_{M_k})\|
\geq 1-\frac{1}{k}-\sum_{i=k}^{\infty}\epsilon_i
\geq 1/2$,
whence
$\|\bar{\chi}_k\|\geq
\|\phi_{k,\infty}\circ\chi_k\|-\frac{2}{k}
\geq 1-\sum_{i=k}^{\infty}\epsilon_i-\frac{2}{k}>0.$
{\blue{Thus}} $\bar{\chi}_k$ is nonzero.
Since  $\bar{A}$ admits nonzero c.p.c.~order zero map
$\bar{\chi}_k:M_k\rightarrow \bar{A}$ ({\blue{for all}} $k\geq 4$),
$\bar{A}$ must be infinite dimensional.

Note that
$1_{\bar{A}}=
\pi_{\infty}(\{\phi_1(1_{A_1}),\phi_{2}(1_{A_2}),
{{\cdots}}\})$ {\blue{and}}
$$
\pi_{\infty}(1_{A})-1_{\bar{A}}
=\pi_{\infty}(\{\gamma_1(1_{A_1}),\sum_{i=1}^2\gamma_{i}(1_{A_i}),
{{\cdots}}\}).
$$
{\blue{Since, $\gamma_j(1_{A_j})\lesssim {\bar a}_j\lesssim a_j,$ and $a_i\perp a_j$($i\not=j$),}}
{\blue{for all}} $k\in\mathbb{N}$,
$\sum_{i=1}^k\gamma_{i}(1_{A_i})\lesssim
\sum_{i=1}^k a_k\lesssim f_{1/2}(a).$
{\blue{It follows}}
\beq\label{719-n12}
\pi_{\infty}(1_{A})-1_{\bar{A}}\lesssim_{l^{\infty}(A)/c_0(A)}a.
\eneq


\noindent
{{For}} $x\in \calF$ {\blue{and}} $k\in\mathbb{N}$,
$x1_{A_k}
\approx_{\epsilon_0}(\gamma_0(x)+\phi_0(x))1_{A_k}
=\phi_0(x)1_{A_k}
\approx_{\epsilon_1}(\gamma_1(\phi_0(x))+\phi_1(\phi_0(x)))1_{A_k}
=\phi_{0,2}(x)1_{A_k}
\approx_{\epsilon_2}
\cdots
\approx_{\epsilon_{k-1}}\phi_{0,k-1}(x)1_{A_k}=\phi_{0,k-1}(x)$.
Similarly,
{{we have}}
$1_{A_k}x\approx_{\sum_{i=0}^{k-1}\epsilon_i}\phi_{0,k-1}(x).$
{\blue{Thus}}
$1_{A_k}x\approx_{\sum_{i=0}^{k-1}\epsilon_i}x1_{A_k}.$
{\blue{Note}} that $2\sum_{i=0}^{\infty}\epsilon_i<\epsilon.$
{\blue{Hence}}
\beq\label{719-n10}
1_{\bar{A}}\iota_A(x)\approx_{\epsilon}\iota_A(x)1_{\bar{A}}.
\eneq
Moreover, $1_{A_k}x1_{A_k}\approx_{2\sum_{i=0}^{k-1}\epsilon_i}\phi_{0,k-1}(x)$
implies
\beq\label{719-n11}
1_{\bar{A}}\iota_A(x)1_{\bar{A}}\approx_{\epsilon}
\pi_{\infty}(\{\phi_{0,1}(x),\phi_{0,2}(x),
{\green{\cdots}} \})\in \bar{A}.
\eneq
By Proposition \ref{equivalent-definition-of-tracial-approximation}
(see \eqref{719-n10}, \eqref{719-n11} and \eqref{719-n12})
$A$ is asymptotically tracially {\blue{in ${\cal N}_{n,s,s}.$}}


\end{proof}

\begin{cor}
\label{id-finite-tdimnuc-equal-finite-dimnuc}

Let $A$ be a simple separable infinite dimensional unital
$C^*$-algebra,
then the following are equivalent:

(1) $A$ is asymptotically tracially  {\blue{in ${\cal N}_n$}}
for some $n\in{{\mathbb{N}\cup\{0\}}},$

(2)  {\blue{${\rm{Trdim_{nuc}}}{\id}_A\le n$}}
for some $n\in{{\mathbb{N}\cup\{0\}}},$  {\blue{and}}

{\blue{(3) $A$ is
asymptotically tracially in {\blue{${\cal N}_{n,s,s}$}}}}
{{for some $n\in\mathbb{N}\cup\{0\}.$}}
\end{cor}

\begin{proof}
{\blue{Note that (3) $\Rightarrow$ (1) is automatic.}}
(1) $\Rightarrow$ (2) follows from
Proposition \ref{Tdimnuc-id-map-Tdimnuc}, {\blue{and that}}
(2) $\Rightarrow$ {\blue{(3)}} follows from
Theorem \ref{simple-T-finite-dimnuc-implies-T-simple-finite-dimnuc}.

\end{proof}

\begin{thm}
\label{tracial-finite-dimnuc-implies-tracial-Z-stable}
Let $n\in{{\mathbb{N}\cup\{0\}}}$.
Let $A$ be a simple separable unital infinite dimensional $C^*$-algebra
and $A$ is asymptotically tracially  {\blue{in ${\cal N}_n.$}}
Then $A$ is asymptotically tracially {\blue{in ${\cal N}_{{\cal Z},s,s}.$}}

\end{thm}

\begin{proof}

This follows from
Theorem \ref{simple-T-finite-dimnuc-implies-T-simple-finite-dimnuc}
and
\cite[Theorem 7.1]{W-2012-pure algebras}.

\end{proof}

\section{$\mathcal{Z}$-stable generalized inductive limits}


{\blue{The following notation is taken from  \cite{W2010decompo}  with a modification.}}

\begin{nota}
(cf. \cite[Notation 2.2]{W2010decompo})
\label{Z-stable-relation}
Let $A$ be a unital \CA, $n\in\N$, $\epsilon \ge 0$,
and let $\calF\subset A$ be a finite subset.
If $\psi: M_n\rightarrow A$ is a {\blue{c.p.c.~map}}
and $v\in A^1$ such that

(i) $\|v^*v-(1_A-\psi(1_{M_n}))\|\leq \epsilon$,

(ii) $\|vv^*\psi({\blue{e_{1,1}^{(n)}}})-vv^*\|\leq \epsilon$,

(iii) $\|[\psi(y),x]\|\leq\epsilon$ for all $x\in\calF$
and for all $y\in M_n^1$,

(iv) $\|[v,x]\|\leq \epsilon$ for all $x\in\calF$, {\blue{and}}

{\blue
(v) $\psi$ is c.p.c.~$\epsilon$-almost order zero map} {\blue{(recall  Definition \ref{D71}),}}

\noindent
then we say {\blue{$\psi$ and $v$}} satisfy the relation $\check{\cal R}_A(n,\calF,\epsilon)$
{\blue{or the pair $(\phi, v)$ satisfies the relation $\check{\cal R}_A(n,\calF,\epsilon).$}}
\end{nota}


\begin{lem}
\label{approximately-multiplicative-relation-R}

Let $A$ be a unital \CA, $n\in\N$, $\epsilon>0$,
and let $\calF\subset A$ be a finite subset.
{\blue{Suppose}} that a c.p.c.~map $\psi: M_n\rightarrow A$
and $v\in A^1$
satisfy the relation $\check{\cal R}_A(n,\calF,\epsilon).$ %
{\blue{Suppose {\blue{also}} that}}
$B$ {\blue{is}}  a unital \CA, $\phi: A\rightarrow B$  is a u.c.p.~map {\blue{and }}
$0<\delta<\Delta(M_n,\epsilon)$  {\blue{is a positive number}}
(see {{Definition}} \ref{almost-n-dividable-error}
for the definition of $\Delta(-,-)$).
If $\phi$ is $\delta${{-multiplicative}}
on ${{\calF\cup\psi({\cal G}^{M_n})\cup\{v,v^*,vv^*\}}}$ {\blue{(recall
that ${\cal G}^{M_n}$ is the standard generating set of $M_n,$ see  {{Notation}} \ref{NoGF}),}}
{{then}}
$\phi\circ\psi$ {\blue{and}} $\phi(v)$
satisfy the relation $\check{\cal R}_B(n,\phi(\calF),{{2\epsilon+3\delta^{1/2}}}).$

\end{lem}

\begin{proof}
We verify this as follows.

{
\noindent
\hspace{0.13in}{\blue{(1)}} $\|\phi(v)^*\phi(v)-(1_B-\phi\circ\psi(1_{M_n}))\|
\approx_\dt
\|\phi(v^*v)-(\phi(1_A)-\phi\circ\psi(1_{M_n}))\|
\leq\ep$ ({\blue{see}} (i) of
{{Notation}}
\ref{Z-stable-relation}).
\beq\nonumber
(2)\, \|\phi(v)\phi(v)^*\phi\circ\psi(e_{1,1}^{(n)})-\phi(v)\phi(v)^*\|
&\approx_{2\dt}&
\|\phi(vv^*)\phi\circ\psi(e_{1,1}^{(n)})-\phi(vv^*)\|
\\\nonumber
&\overset{({\blue{\rm Lemma\,\ref{c.p.c.-almost-multiplicative-easy}}})}{\approx_{\dt^{1/2}}}&
\|\phi(vv^*\psi(e_{1,1}^{(n)}))-\phi(vv^*)\|
\\
&\overset{{({\rm(iii)~ of~{{Notation}}~} \ref{Z-stable-relation})}}{\leq}&\ep.
\eneq

\noindent
\hspace{0.13in}{\blue{(3)}} Let $x\in\calF$ {\blue{and}} $y\in M_n^1,$ {\blue{Then,}} by
{{Lemma}} \ref{c.p.c.-almost-multiplicative-easy},
$\phi\circ\psi(y)\phi(x)\approx_{\dt^{1/2}}\phi(\psi(y)x)$.
Similarly,\linebreak
 {\blue{$\phi(x)\phi\circ\psi(y)\approx_{\dt^{1/2}}\phi(x\psi(y))$.}}
Thus $\|[\phi\circ\psi(y),\phi(x)]\|\leq \ep+2\dt^{1/2}$
(using (iii) of {{Notation}}  \ref{Z-stable-relation}).

\noindent
\hspace{0.13in}{\blue{(4)}} Let $x\in\calF$, then $\|\phi(v)\phi(x)-\phi(x)\phi(v)\|
\approx_{2\dt}\|\phi(vx-xv)\|\leq\ep$ (using (iv) of
{{Notation}} \ref{Z-stable-relation}).

\noindent
\hspace{0.13in}{\blue{(5)}} By {{Definition}}
\ref{almost-n-dividable-error} and (v) of
{{Notation}} \ref{Z-stable-relation},
$\phi\circ\psi$  is $2\ep$-almost order zero map.

\noindent
{{Thus}} $\phi\circ\psi$, $\phi(v)$
satisfy the relation $\check{\cal R}_B(n,\phi(\calF),2\epsilon+3\delta^{1/2}).$

}
\end{proof}

{\blue

Also recall the following proposition (with a mild modification):

\begin{prop}
[{cf. \cite[Proposition 2.3]{W2010decompo}}]
\label{W2010Z-stable-prop}
Let $A$ be a separable unital \CA.
{{Then $A$ is $\mathcal{Z}$-stable if and only if the following {\blue{condition holds}}:}}
For any $n\in\N$, any finite subset $\calF\subset A_+$ {\blue{and}}
any $0<\epsilon<1$, there are {\blue{$m\in\N,$ a}}
c.p.c.~map $\psi: M_{mn}\rightarrow A$
and $v\in {\blue{A^1}}$ satisfying  {\blue{the}} relation {{$\check{\cal R}_A(mn, {\cal F},\ep).$}}

\end{prop}

\begin{proof}
Note that if $A$ is ${\cal Z}$-stable,
then ${\cal Z}$ (hence the dimension drop algebra
${\cal Z}_{n,n+1}$) is unitaly embedded
into $(l^\infty(A)/c_0(A))\cap A'$
(see \cite[Theorem 8.7]{JS1999}, see also \cite[Theorem 2.2]{TomsW07}).
{\blue{It follows from ``(iv) $\Rightarrow$ (iii)" of \cite[Proposition 5.1]{RW-JS-Revisited} that
there is an order zero map $\Psi: M_n\to (l^\infty(A)/c_0(A))\cap A'$ and $V\in (l^\infty(A)/c_0(A))^1$
satisfy  condition (i), (ii)
{{and (v)}}
with $\ep=0.$
There is a c.p.c.~map  $\Psi: M_n\to l^\infty(A)$
and  there is a $\{v_n\}\in (l^\infty(A))^1$
{{such that,}} $\pi_\infty\circ\Psi=\Phi$
{{(see \cite[Proposition 1.2.4]{W-2009CovII})}}
and
{{$\pi_\infty(\{v_n\})=V.$}}
Then the ``only if" part follows.}}

{\blue{For the ``if" part, let}}
$n\in \N$, let $\calF\subset A_+$ be a finite subset, {\blue{and}}
let $0<\epsilon<1$.
{\blue{Choose}} $N:=1+\max\{\|x\|:x\in\calF\}$ {\blue{and}}
$\delta:=\min\{\Delta(M_n,\epsilon/2N), (\ep/4)^2\}$
(see {{Definition}}
\ref{almost-n-dividable-error}
for the definition of $\Delta(-,-)$).
Then, by our assumption, there are
$m {{\ \in\ }} \N,$ {\blue{a}}
c.p.c.~$\delta$-almost order zero map $\psi: M_{mn}\rightarrow A$
and $v\in {\blue{A^1}}$ satisfying {\blue{the}} relation
$\check{\cal R}_A(mn, {\cal F},\dt).$

Let $h:M_n\hookrightarrow M_{mn}$ be a unital embedding
such that $e_{1,1}^{(mn)}\leq h(e_{1,1}^{(n)})$.
Then $\psi\circ h: M_n\rightarrow A$ is a c.p.c.~$\delta$-almost order zero map.
By the choice of $\delta$ and the definition of $\Delta(M_n,\epsilon/2N)$,
there exists a c.p.c.~order zero map $\phi: M_n \rightarrow A$
such that
$\|\psi\circ h-\phi\|\leq \epsilon/2N.$
Then {\blue{one has}}
\beq\nonumber
\|vv^*\phi(e_{1,1}^{(n)})-vv^*\|
&\approx_{\epsilon/2N}&
\|vv^*\psi\circ h(e_{1,1}^{(n)})-vv^*\|
=\|vv^*(1_A-\psi\circ h(e_{1,1}^{(n)}))^2vv^*\|^{1/2}\\\nonumber
&\le& \|vv^*(1_A-\psi\circ h(e_{1,1}^{(n)}))vv^*\|^{1/2}
\le \|vv^*(1_A-\psi(e_{1,1}^{(mn)}))vv^*\|^{1/2}\\\nonumber
&\leq&
\|vv^*(1_A-\psi(e_{1,1}^{(mn)}))\|^{1/2}
\le \delta^{1/2}.
\eneq
{{Thus}} $\phi,v$ satisfy (ii) in the relation
$\check{\cal R}_A(n, {\cal F},\ep).$
One easily {\blue{checks}} that  {\blue{$\phi$}} {\blue{and}} $v$ also satisfy the rest terms in the relation
$\check{\cal R}_A(n, {\cal F},\ep).$  {\blue{Since $\phi$ is an order zero c.p.c.~map,
\cite[Proposition 2.3]{W2010decompo} applies and
$A$ is ${\cal Z}$-stable.}}

\end{proof}

}

\begin{lem}
\label{construction-of-simple-Z-stable-gen-limit}

Let $A_i$ be {{a}} unital separable {\blue{\CA}}
and {\blue{let}} $\phi_i:A_i\rightarrow A_{i+1}$ {\blue{be}} u.c.p.~maps
{\blue{($i\in\mathbb{N}$).}}
Let $X_i=\{x_{i,1},{\blue{x_{i,2},}}\cdots\}\subset A_{i+}^1$
be a countable dense subset of $A_{i+}^1$,
let $X_{i,k}:=\{{\blue{x_{i,1},}}\,{\green{x_{i,2},}}\cdots,x_{i,k}\}$,
and {\blue{let}} $Y_k:=\cup_{1\leq i\leq k}\phi_{i,k}(X_{i,k})$
($i,k\in{{\mathbb{N}}}$).
{{{\blue{Set}} $A_0=A_1$, $Y_0=\{0\}\subset A_0$
and $\phi_0:={\rm id}_{A_0}: A_0\to A_1$.}}

Then {\blue{the system $(A_i,\phi_i)$ forms a generalized inductive  system}}
and
$A:=\lim_{i}(A_i,\phi_i)$ is {\blue{a}} {\blue{simple and}}
$\mathcal{Z}$-stable {\blue{\CA,}}
if the following {\blue{conditions}} hold for any $n\in\mathbb{N}$:


(1)
{\blue{$f_{1/2}(x)$ is full in $A_n$ for all
$x\in {\cal F}_n:=Y_n\cap B_{\frac{3}{4},1}(A_n),$}}

(2) there {\blue{exist}} a c.p.c.~map
$\psi_{n}: M_{n!}\rightarrow A_n$ {\blue{and}}
$v_n\in A_n^1$
such that
$\psi_n$ {\blue{and}} $v_n$ {{satisfy}} the relation
$\check{\cal R}_{A_n}(n!,\phi_{n-1}(Y_{n-1}),{{\frac{1}{n!}}}),$ {\blue{and}}

(3) 
$\phi_{n}$ is $\epsilon_{n}${{-multiplicative}} on
$$
Y_{n}
\cup \left(\cup_{1\leq j\leq n}\phi_{j,n}{\blue{({\cal G}_{{\cal F}_j}
)}}\right)
\cup
{{\left(
\cup_{1\leq j\leq n}\left(
\phi_{j,n}\circ\psi_{j}({\cal G}^{M_{j!}})\cup\{\phi_{j,n}(v_j),\phi_{j,n}(v_j)^*,
{{\phi_{j,n}(v_{j}v_{j}^*)}}\}\right)
\right),}}
$$
where
$$
\epsilon_{n}:={\blue{\frac{1}{4^{n}}
\min_{1\leq j\leq n} \{1,\sigma({\cal G}_{{\cal F}_j}),
\Delta(M_{j!},\frac{1}{j!})\}}}
{{\mbox{ and }}}
$$
{{${\cal G}^{M_{j!}}$
is {{the standard generating set}} of {\blue{$M_{j!}$}}}}
(see {{Proposition}}
\ref{fullness neighbourhood and simplicity}
for ${\cal G}_{{\cal F}_j}$ and {\blue{${{\sigma({\cal G}_{{\cal F}_j})}},$}}
{{{\blue{and}} see {{Notation}}
\ref{def-of-gen-ind-lim} for $\phi_{j,k}$}}).

\end{lem}

\begin{proof}
By Lemma \ref{construction-simple-gen-lim},
$(A_i,\phi_i)$ forms a generalized inductive {\blue{system}}
and $A:=\lim_{i}(A_i,\phi_i)$ is {\blue{a simple \CA.}}
%
{\blue{We}} will show that $A$ is $\calZ$-stable.

Let 
$\ep>0,$ $n\in \mathbb{N}$ and let $\calF\subset A_+^1$
be a finite subset.
{\blue{Then}} there exists $n_1>n\in\mathbb{N}$ 
{\blue{such that}} 
\beq
\label{f8-lem8-1}
\calF\subset_{\frac{\epsilon}{16}} \phi_{n_1,\infty}(Y_{n_1}).
\eneq
{\blue{Choose $n_2>n_1$ such that
${{\frac{2}{n_2!}+
3(\sum_{i=n_2}^\infty\epsilon_{i})^{1/2}}}<\ep/8.$}}
By our assumption,
there exist a c.p.c.~{{map}}
$\psi_{n_2}: M_{n_2!}\rightarrow A_{n_2}$
and $v_{n_2}\in A_{n_2}^1$
{\blue{such}} that

{{$(1')$}}  {\blue{the pair  ($\psi_{n_2},\,v_{n_2}$) satisfies}} the relation
$\check{\cal R}_{A_{n_2}}(n_2!,\phi_{n_2-1}(Y_{n_2-1}),
{{\frac{1}{{n_2!}}}}),$ {{and}}

{{$(2')$}} {\blue{for any}} $k\geq n_2$, {\blue{the map}}
$\phi_{k}$ {\blue{(from $A_k$ to $A_{k+1}$)}} is $\epsilon_{k}${{-multiplicative}} on
$$
{
Y_k\cup
\phi_{{n_2},k}\circ\psi_{{n_2}}({\cal G}^{M_{{n_2}!}})
\cup\{\phi_{{n_2},k}(v_{n_2}),\phi_{{n_2},k}(v_{n_2})^*,
\phi_{{n_2},k}(v_{{n_2}}v_{{n_2}}^*)\}.}
$$

\noindent
{{By}} {{$(2')$}}, for any  $k\geq n_2$,
$\phi_{n_2,k}$ is
(${
Y_{n_2}
\cup\psi_{n_2}({\cal G}^{M_{{n_2}!}})
\cup
\{v_{n_2},v_{n_2}^*,v_{n_2}v_{n_2}^*\}},
\sum_{i=n_2}^k\epsilon_{i}$){{-multiplicative}}.
Therefore $\phi_{n_2,\infty}$ is
(${
Y_{n_2}
\cup\psi_{n_2}({\cal G}^{M_{{n_2}!}})
\cup
\{v_{n_2},v_{n_2}^*,v_{n_2}v_{n_2}^*\}},
\sum_{i=n_2}^\infty\epsilon_{i}$){{-multiplicative}}.
{{Note that $\sum_{i=n_2}^\infty\epsilon_{i}
<\Delta(M_{n_2!}, \frac{1}{n_2!})$}}.
Then, by Lemma \ref{approximately-multiplicative-relation-R},
{\blue{the pair ($\phi_{n_2,\infty}\circ\psi_{n_2},$
$\phi_{n_2,\infty}(v_{n_2})$)
satisfies}} the relation
$$
\check{\cal R}_{A}(n_2!,\phi_{n_2,\infty}(\phi_{n_2-1}(Y_{n_2-1})),
\frac{2}{n_2!}+
3(\sum_{i=n_2}^\infty\epsilon_{i})^{1/2}).
$$
By \eqref{f8-lem8-1}, we have
$\calF\subset_{\frac{\epsilon}{16}}\phi_{n_2,\infty}(\phi_{n_2-1}(Y_{n_2-1}))$.
Also note
${{\frac{2}{n_2!}+
3(\sum_{i=n_2}^\infty\epsilon_{i})^{1/2}}}<\ep/8.$
Therefore {\blue{the pair ($\phi_{n_2, \infty},\,\phi_{n_2,\infty}(v_{n_2})$)
satisfies}} the relation
$
\check{\cal R}_{A}(n_2!,\calF,
\epsilon).
$
{\blue{Thus,}} by {{Proposition}}
\ref{W2010Z-stable-prop}, $A$ is ${\cal Z}$-stable.


\end{proof}

\begin{lem}
\label{regulation-of-tracial-nuclear-Z-stable-lem}

Let $A$ be a unital  {\blue{simple \CA\,
which is 
asymptotically tracially in ${\cal C}_{{\cal Z},s}$ {\blue{(see
{{Definition}} \ref{DNn}).}}}}
Then, for any finite subset  $\calF\subset A$,
{{any}} $\epsilon>0$,
{{any}} $n\in\mathbb{N},$ {\blue{and}}
{{any}} $a\in A_+\backslash\{0\}$,
%
the following {\blue{conditions}} hold.

There exist a
separable unital 
$C^*$-algebra $B$
{\blue{and}}  {\blue{a}} u.c.p.~map
$\alpha: A\rightarrow B$
such that

(1) $\alpha$ is {\blue{an}} $(\calF,\epsilon/2)$-{{approximate embedding}}, {\blue{and}}

\noindent
for any finite subset $\mathcal{G}\subset B$,
there {\blue{exist}} {\green{three}} c.p.c.~maps
$\beta: B\rightarrow A$,
$\gamma: A\rightarrow (\beta\circ\alpha(A))^{\bot}$,
$\psi: M_n\rightarrow \Her_A(\beta\circ\alpha(1_A))$,
and $v\in \Her_A(\beta\circ\alpha(1_A))^1$
{\blue{such}} that

(2) $\beta\circ\alpha(1_A)$, $\gamma(1_A)$ are projections
and $1_A=\beta\circ\alpha(1_A)+\gamma(1_A)$,

(3) $x\approx_{\epsilon}\beta\circ\alpha(x)+\gamma(x)$ {\blue{for all}} $x\in\calF$,

(4) $\beta$ is a $(\mathcal{G},\epsilon)$-{{approximate embedding}}.

(5) $\gamma(1_A)\lesssim_A a$, {\blue{and}}

(6) $\psi$ {\blue{and}} $v$ satisfy the relation
$\check{\cal R}_{\Her_A(\beta\circ\alpha(1_A))}
(n, \beta\circ\alpha({\cal F}),\ep).$

\noindent
{{If,}} in addition, $A$ is assumed to be
asymptotically tracially in ${\cal N}$,
then $B$
above
can be chosen to be nuclear.

\end{lem}

\begin{proof}

Let $\calF\subset A$ be a finite subset.
{\blue{Without loss of generality,}}
we may assume that $\|x\|\leq 1$ for all $x\in\calF$.
Let $\epsilon\in(0,1)$,
let $n\in\mathbb{N}$, {{and}}
let $a\in A_+\backslash\{0\}$.
Since $A$ is {{simple,}} unital and
asymptotically tracially in {{${\cal C}_{{\cal Z},s}$,}}
$A$ is non-elementary.
Then there exist $a_0,a_1\in \Her_A(a)_+\backslash\{0\}$ such that
$a_0a_1=0$.
{Let $\delta:=
\min\{{\blue{(\epsilon/8)^2, \Delta(M_n,{{\epsilon/4}}),1/2}}\}$.}

{{By \cite[Corollary 3.1]{TomsW07}, {{${\cal C}_{{\cal Z},s}$}} has property (H).}}
{\blue{Then,}} by
Proposition \ref{hereditary-subalgebra-preserves-tracially-approximation},
there exist a  unital 
separable $\mathcal{Z}$-stable $C^*$-algebra ${\bar B}$ and c.p.c.~maps
$\bar \alpha: A\rightarrow  \bar{B}$,
$\bar \beta_i:  \bar{B}\rightarrow A,$ {\blue{and}}
$\bar \gamma_i: A\rightarrow A\cap(\bar \beta_i\circ \bar{\alpha}(A))^{\bot}$
($i\in\mathbb{N}$) {{such}} that

{{$(1')$}} {{$\bar {\alpha}(1_A)=1_{\bar B},$}}
$\bar \beta_i(1_{\bar B})$ {\blue{and}}  $\bar \gamma_i(1_A)$ are projections,
$1_A=\bar \beta_i(1_{\bar B})+\bar \gamma_i(1_A)$
for all $i\in \mathbb{N}$,

{{$(2')$}} $x\approx_{\dt}\bar \gamma_i(x)+\bar \beta_i\circ\bar \alpha(x)$
for all $x\in\calF$ for all $i\in\mathbb{N}$,

{{$(3')$}} $\bar {\alpha}$ is {\blue{an}} $(\calF, \delta)$-{{approximate embedding}},

{{$(4')$}} $\lim_{i\rightarrow\infty}
\|\bar \beta_i(xy)-\bar \beta_i(x)\bar \beta_i(y)\|=0$
{{and}}
$\lim_{i\rightarrow \infty}\|\bar \beta_i(x)\|= \|x\|$
for all $x,y\in \bar{B}$, {{and}}

{{$(5')$}} $\bar\gamma_i(1_A)\lesssim_A a_0$ for all $i\in\mathbb{N}$.

\noindent
{{Since}} $\bar B$ is $\calZ$-stable,
{{by {{Proposition}} \ref{W2010Z-stable-prop},}}
there {\blue{is a}}  c.p.c.~order zero map $\bar \psi: M_n\rightarrow \bar B$
and  {\blue{there is}} $\bar v\in \bar B^1$ such that

{{$(6')$}} {\blue{the pair ($\bar\psi,$\, $\bar v$) satisfies}} the relation
$\check{\cal R}_{\bar B}
(n, \bar\alpha({\cal F}),{{\ep/8}}).$

\noindent
{\green{Set}} $B:=\bar B$ {\blue{and}}
$\alpha:=\bar\alpha$.
Then,   {\blue{by {{$(3')$}},
(1) holds.}}

Let $\bar{\cal G}\subset \bar B$ be a finite subset
{\blue{containing $\bar\psi({\cal G}^{M_n})\cup\{{\bar{v}}, {\bar{v}}^*, \bar{v}\bar{v}^*\}.$}}
By {{$(4')$}} and {{$(6')$}}, 
for a  sufficiently large $k\in\N$,
the following {{$(7')$}} and {{$(8')$}} hold:


{{$(7')$}} {\blue{The map}} $\bar\beta_k$ is  {{an}} $(\bar\alpha(\calF)\cup\bar{\cal G},
\delta)$-approximate embedding.

{{$(8')$}} {\blue{The pair ($\bar\beta_k\circ\bar\psi,$\,  $\bar\beta_k(\bar v)$) {\blue{satisfies}}}} the relation
$\check{\cal R}_{\Her_A(\bar\beta_k(1_{\bar B}))}
(n, \bar\beta_k(\bar\alpha({\cal F})),{{\ep/4}}).$

\noindent
{{Set}} ${\cal G}:=\bar{\cal G}$,
$\beta:=\bar\beta_k$,
$\gamma:=\gamma_k$, $\psi:={\bar \bt_k}\circ \bar\psi$, $v:={\bar \bt}_k(\bar v)$.
{\blue{Then, by {{$(1')$}}, (2) above holds,  by {{$(2')$}}, (3) holds, by {{$(7')$}}, (4) holds,  by {{$(5')$}}, (5) holds,  and, by {{$(8')$}}, (6) holds.
This proves the first part of the lemma.}}

If, in addition, $A$ is also assumed to be
asymptotically tracially in ${\cal N}$,
then,  by
Proposition \ref{hereditary-subalgebra-preserves-tracially-approximation},  $\Her_A(\bar\beta_k(1_{\bar B}))$ is
simple and asymptotically {\blue{tracially in ${\cal N}$}}.
There exists $a_2\in \Her_A(\bar\beta_k(1_{\bar B}))_+\backslash\{0\}$
such that $a_2\lesssim_A a_1$.
Since $\Her_A(\bar\beta_k(1_{\bar B}))$ is
asymptotically {\blue{tracially in ${\cal N}$}},
by
Proposition \ref{hereditary-subalgebra-preserves-tracially-approximation},
there exist a unital nuclear $C^*$-algebra $B$
and c.p.c.~maps
$\hat\alpha: \Her_A(\bar\beta_k(1_{\bar B}))\rightarrow B$,
$\hat\beta_i: B\rightarrow \Her_A(\bar\beta_k(1_{\bar B}))${\green{, and}}
$\hat\gamma_i: \Her_A(\bar\beta_k(1_{\bar B}))\rightarrow \Her_A(\bar\beta_k(1_{\bar B}))\cap\hat \beta_i(B)^{\bot}$
($i\in\mathbb{N}$) {{such}} that

{{$(1'')$}} $\hat \alpha$ is {{a}} u.c.p.~{{map}},
$\hat \beta_i(1_B)$ {\blue{and}}
$\hat \gamma_i(\bar\beta_k\circ\bar\alpha(1_A))$ are projections,
$\bar\beta_k(1_{\bar B})
=\hat \beta_i(1_B)+\hat \gamma_i(\bar\beta_k(1_{\bar B}))$
{{for all}} $ i\in \mathbb{N}$,

{{$(2'')$}} $x\approx_{\delta}\hat \gamma_i(x)+\hat \beta_i\circ\hat \alpha(x)$
{{for all}} $ x\in\bar\beta_k\circ\bar\alpha(\calF)$
{{and for all}} {\blue{$i\in\mathbb{N}$,}}

{\blue{$(3'')$}} $\hat \alpha$ is {\blue{a}}
$(\bar\beta_k\circ\bar\alpha(\calF)
\cup{{\bar\beta_k\circ\bar\psi({\cal G}^{M_n})}}
\cup\{\bar\beta_k(v),\bar\beta_k(v)^*,{{\bar\beta_k(vv)^*}}\}, \delta)$-{{approximate embedding}},

{{$(4'')$}} $\lim_{i\rightarrow\infty}\|\hat \beta_i(xy)-\hat \beta_i(x)\hat \beta_i(y)\|=0$ {\blue{and}}
$\lim_{i\rightarrow \infty}\|\hat \beta_i(x)\|= \|x\|$
{{for all}} $ x,y\in B$, {\blue{and}}

{{$(5'')$}} $\hat \gamma_i(\bar\beta_k\circ\bar\alpha(1_A))
\lesssim_{\Her_A(\bar\beta_k(1_{\bar B}))} a_2$ {{for all}} $i\in\mathbb{N}$.

\noindent
{{Let}}
$\alpha:=\hat\alpha\circ\bar\beta_k\circ\bar\alpha.$
{\blue{Then,  since $\dt<\ep/8,$  by {{$(3')$}} and {{$(3'')$}},
(1) of the lemma holds.}}
%
%
{{Let}} ${\cal G}\subset B$ be a finite subset.
{\blue{By  {{$(4'')$}},  there exits  a large $m\in\N$  such that}}

{\blue{{{$(6'')$}} $\hat\beta_m$ is a $({\cal G}
\cup\hat\alpha\circ\bar\beta_k\circ\bar\alpha(\calF)\cup \hat\alpha\circ\bar\beta_k(\bar\psi({\cal G}^{M_n}))
\cup\{\hat\alpha\circ\bar\beta_k(v),\hat\alpha\circ\bar\beta_k(v)^*, \hat\alpha\circ\bar\beta_k(vv)^*\},
\delta)$-approximate embedding.}}



\noindent
{{Then,}}  by  the choice of ${\bar{\cal G}},$
{{and by}} {{$(7')$}}, {{$(3'')$}}, and {{$(6'')$}}, {\blue{the map}}
$\hat\beta_m\circ\hat\alpha\circ\bar\beta_k$ is
$3\dt${{-multiplicative}} on
$\bar\af(\calF)\cup\bar\psi({\cal G}^{M_n})
\cup \{\bar{v}, \bar{v^*},\bar{v}\bar v^*\}.$
{{Moreover, by {{$(8')$}} and Lemma
\ref{approximately-multiplicative-relation-R}, we have}}

{{$(7'')$}} the pair {\blue{($\hat\beta_m\circ\hat\alpha\circ\bar\beta_k\circ\bar\psi,$\,
$\hat\beta_m\circ\hat\alpha\circ\bar\beta_k(\bar v)$) satisfies}} the relation
$\check{\cal R}_{\Her_A(\hat\beta_m\circ\hat\alpha\circ\bar\beta_k(1_{\bar B}))}
(n, \hat\beta_m\circ\hat\alpha\circ\bar\beta_k(\bar\alpha({\cal F})),5\dt).$

{\blue{\noindent
{{Define}}  $\beta:=\hat\beta_m$, $\gamma:=\bar\gamma_k+\hat\gamma_m\circ\bar\beta_k\circ\bar\alpha$,
$\psi:=\hat\beta_m\circ\hat\alpha\circ\bar\beta_k\circ\bar\psi$ and
$v:=\hat\beta_m\circ\hat\alpha\circ\bar\beta_k(\bar v).$
Since $\bar\gamma_k(A)\bot\hat\gamma_m\circ\bar\beta_k\circ\bar\alpha(A)$,
we have {\blue{that}} $\gamma:=\bar\gamma_k+\hat\gamma_m\circ\bar\beta_k\circ\bar\alpha$
is {\blue{also}} a c.p.c.~map.
Then,
by {{$(1')$}} and {{$(1'')$}}, (2) holds,
by {{$(2')$}} and {{$(2'')$}}, (3) holds,  by {{$(6'')$}}, (4) holds,  by {{$(5')$}}, {{$(5'')$}},   and by
the fact that $a_2\lesssim_A a_1$ and {\blue{$a_0\bot a_1,$}} and $a_0+a_1\lesssim_A a, $  (5) holds.
Finally, by {{$(7'')$}},  (6) holds.}}






\end{proof}

{{The following lemma is well known.}}
\begin{lem}
\label{f8-9lem}
Let $A$ {{be a \CA\, and $B$ be a
nuclear \CA.}}
If there exist {\green{c.p.c.}}~maps
$\af: A\to B$ and $\bt: B\to A$ such that $\bt\circ \af=\id_A,$
then $A$ is also nuclear.

\end{lem}

\begin{proof}
Let $\calF\subset A$ be a finite subset and let $\ep>0$.
Since $B$ is nuclear, there exist a finite dimensional \CA\ $F$
and {{two}} {\green{c.p.c.}}~maps {\blue{$\phi: B\to F,$}} and $\psi: F\to B$
{\blue{such that}} $\af(x)\approx_{\ep/(\|\bt\|+1)} \psi\circ\phi(\af(x))$ for all $x\in \calF$.
Note that $\phi\circ\af: A\to F$ and $\bt\circ \psi: F\to A$
are {\green{c.p.c.}}~maps.
For any $x\in\calF$,
$x=\beta(\af(x))\approx_\ep\beta(\psi\circ\phi(\af(x)))$.
Thus $A$ is nuclear.
\end{proof}



\begin{thm}
\label{Thm-tracial-simple-nuclear-Z-stable}

Let $A$ be a simple separable unital $C^*$-algebra.
Assume that $A$ is
asymptotically tracially in {{${\cal C}_{{\cal Z},s}.$}}
Then,
%
for any finite subset $\calF\subset A$,  any $\epsilon>0$, and any
$a \in A_+\backslash\{0\}$,
there exists a unital
$C^*$-subalgebra
$B\subset  l^\infty(A)/c_0(A)$ which is strictly embedded
such that $B$ in {{${\cal C}_{{\cal Z},s,s}$}}, {\blue{and}}

(1) $1_B\iota_A(x)\approx_{\epsilon}\iota_A(x)1_B$\, {\blue{for all}} $ x\in\calF$,

(2) $1_B\iota_A(x)1_B\in_{\epsilon} B$ {\blue{and}}  $\|1_B\iota_A(x)1_B\|\geq \|x\|-\epsilon$\,
{\blue{for all}} $ x\in \calF$,  and

(3) $\iota_A(1_A)-1_B\lesssim_{
l^\infty(A)/c_0(A)}\iota_A(a)$.

\noindent
{{Moreover,}} if, in addition,
$A$ is also
asymptotically tracially in ${\cal N},$ then $A$ is
asymptotically tracially in
{{${\cal N}_{{\cal Z}, s,s}$}} {{(see {{Definition}} \ref{DNn}).}}


\end{thm}


\begin{proof}

Let $\calF\subset B_{\frac{4}{5},1}(A_+)$ be a finite subset,
let $\epsilon\in(0,1)$ {\blue{and}}
let $a\in A_+\backslash\{0\}$.
{\blue{We may assume that $A$ is infinite dimensional.  Since $A$ is also simple and unital, we further assume that $A$ is non-elementary.}}
Then there exists  a sequence of  mutually orthogonal norm one positive elements
$\{a_n\}$ in
$\Her_A(a)_+\backslash\{0\}.$

{\blue{Choose   $A_0= A$ and $Y_0=\{0\}\subset A_0$.
Let $\phi_0={\rm id}_{A_0}: A_0\rightarrow A_0.$
{\blue{Set}} $\epsilon_0:=\epsilon/100,$
 $\psi_1: M_1(=\C) \rightarrow \Her_{A}(\phi_0(1_{A_0}))$ the unital *-homomorphism, {\blue{and set}}
 $v_1=1_A\in \Her_{A}(\phi_0(1_{A_0}))(=A).$
Let $B_0=\C,$
let $\beta_0: B_0\rightarrow A_0$ be the zero map, {\blue{and}}
let $W_{0,i}=\{0\}$ for all $i\in\N$.
}}

We claim that, for each $k\in \N$,
we can make the following choices:

(k - 1) A hereditary $C^*$-subalgebra
$A_k=\Her_A(\phi_{k-1}(1_{A_{k-1}}))\subset A,$

(k - 2) a positive element $\bar a_k\in (A_k)_+\backslash\{0\}$
such that $\bar{a}_{k}\lesssim_A a_{k},$


(k - 3) a countable dense subset
$X_k=\{x_{k,1}:=1_{A_k},{\blue{x_{k,2},}}\cdots\}\subset A_k^1$,
and finite subsets
$X_{k,i}=\{1_{A_{k}}\}\cup\{x_{k,j}:1\leq j\leq i\}\subset A_k^1$
{{($i\in\mathbb{N}$),}}

(k - 4) a finite subset
$Y_k=\cup_{\blue{1\leq j\le k}}\phi_{j,k}(X_{j,k})\cap B_{\frac{3}{4},1}(A_{k})\subset A_k$,

(k - 5) a finite subset $G_{Y_k}\subset A_k$,
and a positive number $\sigma(G_{Y_k})>0$
{\blue{as}} in
Proposition \ref{fullness neighbourhood and simplicity},

(k - 6) a finite subset $\calF_k\subset A_k$ {\blue{such}} that
(see {{Notation}}
\ref{def-of-gen-ind-lim}
for notation
$\phi_{j,k}$)
\beq
\calF_{k}&=&
\phi_{1,k}(\calF)
\cup
Y_{k}
\cup \left(\cup_{1\leq j\leq k}\phi_{j,k}({\cal G}_{Y_j})\right)
\cup(
\cup_{1\leq j\leq k}
\phi_{j,k}(\beta_{j-1}(W_{j-1,k})){{)}}
\nonumber\\&&
\cup
(\cup_{{1\leq j\leq k}}
\left(
\phi_{j,k}(\psi_{j}({\cal G}^{M_{j!}}))
\cup
\{\phi_{j,k}(v_j),\phi_{j,k}(v_j)^*,
\phi_{j,k}(v_{j}v_{j}^*)\}
\right),
\eneq

(k - 7) a positive number
\beq
\epsilon_{k}=
\frac{1}{4^{k}}\min_{1\leq j\leq k} \{
\sigma({\cal G}_{Y_{j}}),
\Delta(M_{j!},\frac{1}{j!}),
\frac{\ep}{100}\},
\eneq

(k - 8) a unital \CA\ $B_k$,
and a c.p.c.~map $\alpha_k: A_k\rightarrow B_k$,
such that
$B_{k}$ is a separable unital  (if, in addition, $A$ is
asymptotically tracially in ${\cal N},$ $B_k$ is also nuclear)
$C^*$-algebra,
and $\alpha_{k}$ is {{an}}
$(\calF_{k},\epsilon_{k}/2)$-{\blue{approximate}} embedding,

(k - 9) a countable dense subset $W_k=\{w_{k,1},{\blue{w_{k,2},}}\cdots\}\subset B_k^1$,
and
finite subsets $W_{k,i}=\{w_{k,1},{\green{w_{k,2},}}\cdots,w_{k,i}\}\subset B_k$ ($i\in\N$),

(k - 10) a finite subset ${\blue{ {\cal G}_k=\alpha_k(\calF_k)\subset B_k}}
$,

(k - 11) a c.p.c.~map $\beta_k: B_k\rightarrow A_k$,
and a c.p.c.~map $\gamma_k: A_k\rightarrow {\blue{(\beta_{k}\circ\alpha_{k}(A_{k}))^{\perp}}}$
{\blue{such that}} the following (k - 12) - (k - 15) hold:


(k - 12) $\beta_{k}\circ\alpha_{k}(1_A)$ {\blue{and}}  $\gamma_{k}(1_{A_{k}})$
are projections,
and $1_{A_{k}}=\beta_{k}\circ\alpha_{k}(1_{A_{k}})
+\gamma_{k}(1_{A_{k}}),$

(k - 13) $x\approx_{\epsilon_{k}}\beta_{k}\circ\alpha_{k}(x)
+\gamma_{k}(x)$ for all $x\in\calF_{k},$

(k - 14) $\beta_{k}$ is a $(\mathcal{G}_{k},\epsilon_{k})$-{\blue{approximate}} embedding,

(k - 15) $\gamma_{k}(1_A)\lesssim_A \bar{a}_{k},$

(k - 16) a {\blue{c.p.c.~map}}
$\psi_{{k+1}}: M_{{(k+1)!}}\rightarrow \Her_A(\beta_k\circ\alpha_k(1_{A_{k}}))$,
and an element $v_{{k+1}}\in
{{\Her_A(\beta_k\circ\alpha_k(1_{A_{k}}))}}$
{\blue{such}} that the pair  ($\psi_{{k+1}},$\, $v_{{k+1}}$) {\blue{satisfies}} the relation
$\check{\cal R}_{\Her_A(\beta_{k}\circ\alpha_{k}(1_A))}
(k!, \beta_{k}\circ\alpha_{k}({\cal F}_{k}),{{\frac{1}{(k+1)!}}}),$ {\blue{and}}

(k - 17) a c.p.c.~map $\phi_k:=\beta_k\circ\alpha_k: A_k\rightarrow A_k$.

\noindent
{{We}} make our choices recursively.
For the case $k=1:$

(1 - 1) Define $A_1:=\Her_{A}(\phi_0(1_{A_0}))=A.$

(1 - 2) Choose  $\bar{a}_1:=a_1$.

(1 - 3)
Choose  {\blue{a countable dense subset}} $X_1=\{x_{1,1},{\blue{x_{1,2},}}\cdots\}
\subset A_{1+}^1,$
{\green{and}}
let $X_{1,{\green{i}}}:=\{1_{A_1}\}\cup\{x_{1,j}:1\leq j\leq {\green{i}}\}
\subset A_1^1$
{\green{($i\in\mathbb{N}$).}}

(1 - 4) Set $Y_{1}:= X_{1,1}\cap B_{\frac{3}{4},1}(A_{1+})$.

(1 - 5) {\blue{Choose}} ${\cal G}_{Y_1}$ {\blue{and}}  $\sigma({\cal G}_{Y_1})$ {\blue{as}} in
Proposition \ref{fullness neighbourhood and simplicity}.

(1 - 6)
Choose
\beq
\calF_{1}&=&
\phi_{1,1}(\calF)
\cup
Y_{1}
\cup \left(\cup_{1\leq j\leq 1}\phi_{j,1}({\cal G}_{Y_j})\right)
\cup(
\cup_{1\leq j\leq 1}
\phi_{j,1}(\beta_{j-1}(W_{j-1,1})){{)}}
\nonumber\\&&
\cup
\left(\cup_{{1\leq j\leq 1}}
\left(\phi_{j,1}(\psi_{j}({\cal G}^{M_{j!}}))
\cup
\{\phi_{j,1}(v_j),\phi_{j,1}(v_j)^*,
\phi_{j,1}(v_{j}v_{j}^*)\}
\right)
\right){\blue{.}}
\eneq

(1 - 7) Choose
\beq
\epsilon_{1}:=
\frac{1}{4^{1}}\min_{{1\leq j\leq 1}}
\{\sigma({\cal G}_{Y_{j}}),
\Delta(M_{j!},\frac{1}{j!}), \frac{\ep}{100}\}.
\eneq

\noindent
{{Since}} $A_1$
is asymptotically tracially in ${\cal C}_{\cal Z}$
({\blue{and}} is  asymptotically tracially in ${\cal N}$),
by Lemma \ref{regulation-of-tracial-nuclear-Z-stable-lem},
we can {\blue{further}} make  the following choices:


(1 - 8) There {\blue{exist}} a
separable unital {\blue{(nuclear, in case that $A$ is
asymptotically tracially in ${\cal N}$)}}
$C^*$-algebra $B_1$
and a c.p.c.~map
$\alpha_1: A_1\rightarrow B_1$
such that
$\alpha_1$ is {{an}} $(\calF_1,\epsilon_1/2)$-{\blue{approximate}} embedding, and,

(1 - 9)  {\blue{a countable dense subset}} $W_1=\{w_{1,1},w_{1,2},\cdots\}\subset B_1^1$
and finite  {{subsets}}
$W_{1,i}:=\{w_{1,1},{\green{w_{1,2},}}\cdots,w_{1,i}\}$ {{($i\in \mathbb{N}$)}},

(1 - 10) and a finite subset ${{\mathcal{G}_1:=\alpha_1(\calF_1),}}
$
{\blue{and}}

(1 - 11) there   exist a c.p.c.~map
$\beta_1: B_1\rightarrow A_1$
and a c.p.c.~map
$\gamma_1: A_1\rightarrow (\beta_1\circ\alpha_1(A_1))^{\bot}$
{\blue{such}} that

(1 - 12) $\beta_{1}\circ\alpha_{1}(1_A)$ {\blue{and}}  $\gamma_{1}(1_{A_{1}})$
are projections,
{\blue{and}} $1_{A_{1}}=\beta_{1}\circ\alpha_{1}(1_{A_{1}})
+\gamma_{1}(1_{A_{1}}),$

(1 - 13) $x\approx_{\epsilon_{1}}\beta_{1}\circ\alpha_{1}(x)
+\gamma_{1}(x)$ for all $x\in\calF_{1},$

(1 - 14) $\beta_{1}$ is a $(\mathcal{G}_{1},\epsilon_{1})$-{\blue{approximate}} embedding,

(1 - 15) $\gamma_{1}(1_A)\lesssim_A \bar{a}_{1},$ {\blue{and}}


(1 - 16) there exist a c.p.c.~map
$\psi_{{2}}: M_{{2!}}\rightarrow \Her_A(\beta_1\circ\alpha_1(1_{A_1}))$,
and an element
$v_{{2}}\in \Her_A(\beta_1\circ\alpha_1(1_{A_1}))^1$
such that
the pair {\blue{($\psi_{{2}},$\, $v_{{2}}$) satisfies}} the relation
$\check{\cal R}_{\Her_A(\beta_1\circ\alpha_1(1_A))}
({{2!}}, \beta_1\circ\alpha_1({\cal F}_1),\frac{1}{{2!}}).$

(1 {{-}} 17) Define $\phi_1:=\beta_1\circ\alpha_1$
{\blue{which is a c.p.c.~map from $A_1$ to $A_1.$}}

\noindent
{{Assume,}} for $k\geq 1 \in \mathbb{N},$
we  {\blue{have
made the choices}} (j - 1)-(j - 17)
{{for all $1\le j\le k.$}}
Then, for $k+1$, we make the following choices:

(k+1 - 1) Define $A_{k+1}:=\Her_A(\phi_{k}(1_{A_{k}}))$.

(k+1 - 2) Choose $\bar{a}_{k+1}\in (A_{k+1})_+\backslash\{0\}$
such that $\bar{a}_{k+1}\lesssim_A a_{k+1}$.

(k+1 - 3) Choose {\blue{a countable subset}}
$X_{k+1}=\{x_{k+1,1}, {\blue{x_{k+1,2},}}\cdots\}\subset (A_{k+1})_+^1$
{\blue{which is dense in}} $(A_{k+1})_+^1$
{\blue{and choose}}
$X_{k+1,i}:=\{1_{A_{k+1}}\}\cup\{x_{k+1,j}:1\leq j\leq i\}$
{{($i\in\mathbb{N}$).}}

(k+1 - 4) Choose $Y_{k+1}:=\cup_{1\leq j< k+1}\phi_{j,k+1}(X_{j,k+1})\cap B_{\frac{3}{4},1}(A_{k+1})$.

(k+1 - 5) Let ${\cal G}_{Y_{k+1}}$ {\blue{and}}
$\sigma({\cal G}_{Y_{k+1}})$ be as in
Proposition \ref{fullness neighbourhood and simplicity}.

(k+1 - 6)
Let
\beq
\calF_{k+1}&=&
\phi_{1,k+1}(\calF)
\cup
Y_{k+1}
\cup \left(\cup_{1\leq j\leq k+1}\phi_{j,k+1}({\cal G}_{Y_j})\right)
\cup(
\cup_{1\leq j\leq k+1}
\phi_{j,k+1}(\beta_{j-1}(W_{j-1,k+1})){{)}}
\nonumber\\&&
\cup
\left(\cup_{{1\leq j\leq k+1}}
\phi_{j,k+1}(\psi_{j}({\cal G}^{M_{j!}}))
\cup
\{\phi_{j,k+1}(v_j),\phi_{j,k+1}(v_j)^*,
\phi_{j,k+1}(v_{j}v_{j}^*)\}
\right){\blue{.}}
\nonumber
\eneq

(k+1 - 7)
Let
\beq
\epsilon_{k+1}=
\frac{1}{4^{{k+1}}}\min_{{1\leq j\leq k+1}}
\{\sigma({\cal G}_{Y_{j}}),
\Delta(M_{j!},\frac{1}{j!}),
\frac{\ep}{100}\}.
\eneq

\noindent
{{Since}} $A_{k+1}$
is asymptotically tracially in ${\cal C}_{\cal Z}$
{\blue{(and is  asymptotically tracially in ${\cal N}$),}}
by Lemma \ref{regulation-of-tracial-nuclear-Z-stable-lem},
we can {\blue{further}} make the following  choices:


(k+1 - 8)
There exist a separable unital
{\blue{(nuclear, in the case that $A$ {\green{is}} asymptotically  tracially in ${\cal N}$)}} $C^*$-algebra $B_{k+1}$,
and a c.p.c.~map {\blue{$\alpha_{k+1}: A_{k+1}\to B_{k+1}$}}
such {\blue{that}} $\alpha_{k+1}$ is {\blue{an}}
$(\calF_{k+1},\epsilon_{k+1}/2)$-{{approximate embedding}}, and,

(k+1 - 9)
{\blue{a countable dense subset}} $W_{k+1}=\{w_{k+1,1},{\blue{w_{k+1,2},}}\cdots\}\subset B_{k+1}^1,$
{\blue{and}}
{{finite subsets}}
$W_{k+1,i}=\{w_{k+1,1},{\green{w_{k+1,2},}}\cdots,w_{k+1,i}\}$
{{($i\in \mathbb{N}$),}}

{\blue{(k+1 - 10)}}
and for ${{\mathcal{G}_{k+1}=
\alpha_{k+1}(\calF_{k+1})}} {\blue{\subset B_{k+1},}}
$

(k+1 - 11) {\blue{there exist a}}  c.p.c.~map
$\beta_{k+1}: B_{k+1}\rightarrow A_{k+1}$
{\blue{and}}  a c.p.c.~map $\gamma_{k+1}: A_{k+1}\rightarrow (\beta_{k+1}\circ\alpha_{k+1}(A_{k+1}))^{\bot}$
{\blue{such}} that

(k+1 - 12) $\beta_{k+1}\circ\alpha_{k+1}(1_A)$ {\blue{and}} $\gamma_{k+1}(1_{A_{k+1}})$
are projections,
and $1_{A_{k+1}}=\beta_{k+1}\circ\alpha_{k+1}(1_{A_{k+1}})
+\gamma_{k+1}(1_{A_{k+1}}),$

(k+1 - 13) $x\approx_{\epsilon_{k+1}}\beta_{k+1}\circ\alpha_{k+1}(x)
+\gamma_{k+1}(x)$ for all $x\in\calF_{k+1},$

(k+1 - 14) $\beta_{k+1}$ is a $(\mathcal{G}_{k+1},\epsilon_{k+1})$-{\blue{approximate}} embedding,

(k+1 - 15) $\gamma_{k+1}(1_A)\lesssim_A \bar{a}_{k+1},$ {\blue{and}}

(k+1 - 16)
there exist a c.p.c.~map
$\psi_{{k+2}}: M_{{(k+2)!}}\rightarrow
\Her_A(\beta_{k+1}\circ\alpha_{k+1}(1_{A_{k+1}}))$
and an element $v_{{k+2}}\in \Her_A(\beta_{k+1}\circ\alpha_{k+1}(1_{A_{k+1}}))^1$
such that
the pair {\blue{($\psi_{{k+2}},$ $v_{{{k+2}}}$)}} satisfies the relation
$\check{\cal R}_{\Her_A(\beta_{k+1}\circ\alpha_{k+1}(1_A))}
({{(k+2)!}}, \beta_{k+1}\circ\alpha_{k+1}({\cal F}_{k+1})
{{\frac{1}{(k+2)!}}}){\blue{.}}$

(k+1 - 17) Define $\phi_{k+1}:=\beta_{k+1}\circ\alpha_{k+1}.$

\noindent
{{Therefore,}}  by induction, for each $k\in \mathbb{N}$,
we {\blue{have made  choices (k - 1) - (k - 17).}}


For {\blue{each}} $k\in N$,
by {\blue{(k+1 - 1),}} we may
view $\phi_k$ {\blue{as}} a map from $A_k$ to $A_{k+1}.$

{\blue{Since $A_k$ is simple, $f_{1/2}(x)$ is full in $A_k$ for each
$x\in {\cal F}_k\cap B_{3/4,1}(A_k).$}}
{\blue{Then,}}  by (k - 4), (k - 8), (k - 14), (k - 6), (k - 7),
and by Lemma \ref{construction-of-simple-Z-stable-gen-limit},
we {\blue{conclude}} that
$(A_k,\phi_k)$ forms a generalized inductive limit which is
{\blue{simple, separable, unital,}} and
$\mathcal{Z}$-stable.
We denote this generalized inductive limit by $\bar{A}$.

If in addition $A$ is also
asymptotically tracially in ${\cal N}$,
then each $B_k$ are chosen to be nuclear as mentioned above.
{\blue{We claim}} that $\bar{A}$ is nuclear.

Denote the map $\alpha_{k+1}\circ\beta_k: B_k\to B_{k+1}$ by $\theta_k$
{{($k\in\N$).}}
Let $k\in\N$, let $z_1,z_2\in B_k^1$ and let $\dt>0$.
Then there {\blue{are}} $i_1,i_2\in\N$ {\blue{such that}} $z_1\approx_{\dt/8} w_{k,i_1}$
and $z_2\approx_{\dt/8} w_{k,i_2}$.
Let $K\in\N$ {\blue{such that}} $K>\max\{k,i_1,i_2,\frac{8}{\dt}\}$ and $\frac{1}{4^K}<\frac{\dt}{8}$.
Note that $\theta_{i,j}=\alpha_j\circ\phi_{j,i+1}\circ\beta_{i}$ for $j>i\in N$
(see {{Notation}}
\ref{def-of-gen-ind-lim} for the notation $\theta_{i,j}$),
then by (K - 10) and (K - 6),
$\theta_{k,K}(w_{k,i_1}), \theta_{k,K}(w_{k,i_2})\in {\cal G}_K$.
For any $j\geq K,$
keep using {{(i - 14)}} {\blue{and (i - 8)}} for $j\geq i\geq K$,
we have
$$
\theta_{K,j}(\theta_{k,K}(w_{k,i_1})\theta_{k,K}(w_{k,i_2}))
\approx_{2\sum_{i=K}^j \frac{1}{4^j}}
\theta_{K,j}(\theta_{k,K}(w_{k,i_1}))\theta_{K,j}(\theta_{k,K}(w_{k,i_2})).
$$
Note that $2\sum_{i=K}^\infty \frac{1}{4^j}<\dt/2$.
{\blue{Then,}} for any $j\geq K,$
\beq
\theta_{K,j}(\theta_{k,K}(z_1)\theta_{k,K}(z_2))
&\approx_{\dt/4}&
\theta_{K,j}(\theta_{k,K}(w_{k,i_1})\theta_{k,K}(w_{k,i_2}))
\nonumber\\
&\approx_{\dt/2}&
\theta_{K,j}(\theta_{k,K}(w_{k,i_1}))\theta_{K,j}(\theta_{k,K}(w_{k,i_2}))
\nonumber\\
&\approx_{\dt/4}&
\theta_{K,j}(\theta_{k,K}(z_1))\theta_{K,j}(\theta_{k,K}(z_2))
= \theta_{k,j}(z_1) \theta_{k,j}(z_2).
\nonumber
\eneq
Then, {\blue{by}} {{Lemma}}
\ref{positive-criterion-being-gen-ind-lim},
$(B_k,\theta_k)$ forms a generalized inductive limit.
Since $\theta_k$ 
{{is a}} c.p.c.~map for all $k\in\N$,
%
by \cite[Proposition 5.1.3]{BKGILFD},
$\lim_{k\rightarrow \infty}(B_k,\alpha_{k+1}\circ\beta_k)$ is
{{a nuclear \CA.}}

{\blue{Recall that $\bt_k: B_k\to A_k$ and $\af_k: A_k\to B_k$ are c.p.c.~maps, and
$\phi_k=\bt_k\circ \af_k$ {\blue{(see (k - 17)).}}  By the
commutative diagram}}
$$
\xymatrix{
A_1\ar[r]^{\phi_{1}}\ar[d]_{\af_1}  &    A_2\ar[r]^{\phi_{2}}\ar[d]_{\af_2}  &    A_3\ar[r]^{\phi_{3}}\ar[d]_{\af_3}    &\cdots \cdots ~~\bar{A}  \\
    B_1\ar[ru]^{{\bt_1}}\ar[r]^{{\theta_1}}&
    B_2\ar[ru]^{{\bt_2}}\ar[r]^{{\theta_2}}      &
    B_3\ar[r]^{{\theta_3}}   &\cdots\cdots ~~B   \\
 }
$$
{\blue{we obtain two c.p.c.~maps
$\af: {{\bar A}}
\to B$ and $\bt: B\to {{\bar A}}$ such that  $\bt\circ \af=\id_{\bar{A}}.$}}
{{By {{Lemma}} \ref{f8-9lem}, $\bar A$ is also nuclear.}}
{{This proves the claim.}}

{\blue{Now back to the general case.
We embed ${\bar A}$ into $l^\infty(A)/c_0(A)$ as follows.
Let $x\in A_k.$ Define
$$\iota(\phi_{k, \infty}(x))
=\pi_\infty(\{0,0,
{{\cdots}},0, \phi_{{{k,k}}}(x),\phi_{k, k+1}(x),
{{\cdots}}\}),$$
where $\pi_\infty: l^\infty(A)\to l^\infty(A)/c_0(A)$ is the quotient map.}}
{\blue{By (k - 17), (k - 8), and (k - 14),
\beq\label{89-strict}
\liminf_{n\to\infty}\|\phi_{k, k+n}(x)\|\ge (1-4\sum_{j=k}^\infty\ep_j)\|x\|\ge (1/2)\|x\|\rforal x\in {\cal F}_k.
\eneq
It follows (see 
{{Proposition}} \ref{strictly-embedding-simple-case})
that $\iota$ defines a strict embedding from $\bar{A}$ into
{\blue{$l^\infty(A)/c_0(A).$}}}}
{{Note}} that
$$
1_{\bar{A}}=
\pi_\infty(\{\phi_1(1_{A_1}),\phi_{2}(1_{A_2}),
{{\cdots}}\})
{{\andeqn}}
$$
$$
1_{A}-1_{\bar{A}}
=\pi_\infty(\{\gamma_1(1_{A_1}),\sum_{i=1}^2\gamma_{i}(1_{A_i}),
{{\cdots}}\}).
$$
{\blue{For all $k\in\mathbb{N}$,
by (k - 15) and {\blue{by}} the fact that $a_i\perp a_j$ ($i\not=j$),}} we have
$\sum_{i=1}^k\gamma_{i}(1_{A_i})
\lesssim \sum_{i=1}^k a_k\lesssim a.$
{\blue{It follows that}}
\beq
{\blue{\iota_A(1_{A})}}-1_{\bar{A}}\lesssim_{l^{\infty}(A)/c_0(A)} {\blue{\iota_A(a)}}.
\eneq

\noindent
For $x\in \calF$ {\blue{and}} $ k\geq 2 \in\mathbb{N}$,
{\blue{using  {{(j - 13), (j - 12), (j - 17), {\blue{and}} (j - 1)}}
for $1\leq j\leq k\in\N,$ {\blue{repeatedly,}}}}  we have
\beq
\label{f8-9thm-1}
x1_{A_k}
&\approx_{\epsilon_1}&
(\gamma_1(x)+\phi_1(x))1_{A_k}
=\phi_1(x)1_{A_k}=\phi_{1,2}(x)1_{A_k}
\nonumber\\
&\approx_{\epsilon_2}&
(\gamma_2(\phi_{1,2}(x))+\phi_3(\phi_{1,2}(x)))1_{A_k}
=\phi_{1,3}(x)1_{A_k}
\nonumber\\
&\approx_{\epsilon_3}&
\cdots\,\,\,
\approx_{\epsilon_{k-1}}
\phi_{1,k}(x)1_{A_k}
=\phi_{1,k}(x).
\eneq
Similarly, we have
$1_{A_k}x\approx_{\sum_{i=1}^{k-1}\epsilon_i}\phi_{1,k}(x).$
Thus
$1_{A_k}x\approx_{2\sum_{i=1}^{k-1}\epsilon_i}x1_{A_k}.$
Note that $2\sum_{i=1}^{\infty}\epsilon_i<\epsilon.$
{\blue{Hence}}
\beq
\label{f8-9-thm-2}
1_{\bar{A}}{\blue{\iota_A(x)\approx_{\epsilon}\iota_A(x)}}1_{\bar{A}}
\quad\mbox{ for all }x\in\calF.
\eneq
By \eqref{f8-9-thm-2} and \eqref{f8-9thm-1}, we also have
\beq
1_{\bar{A}}{\blue{\iota_A(x)}}1_{\bar{A}}
\approx_{\epsilon}
{\blue{\iota_A(x)}}1_{\bar{A}}
\approx_{\epsilon}
\pi_\infty(\{\phi_{1,k}(x)\})
\in \iota(\bar A)
\quad
\mbox{ for all }x\in\calF.
\eneq
{\blue{This proves the first part of the theorem.}}
{\blue{If, in addition, $A$ is asymptotically tracially in ${\cal N},$ by the claim above, ${\bar A}\in {\blue{{\cal N}_{{\cal Z},s,s}}}.$
{{Since}} %
{\blue{\CA s in}}
${\blue{{\cal N}_{{\cal Z},s,s}}}$ have property (H)
({{see \cite[Corollary 3.1]{TomsW07}}}),}}
{\blue{by}} Proposition \ref{equivalent-definition-of-tracial-approximation},
$A$ is asymptotically tracially in  ${\blue{{\cal N}_{{\cal Z}, {{s,s}}}}}.$

\end{proof}


\section{Simple \CA s which are asymptotically tracially in
${\cal C}_{{{\cal Z},s}}$ or in ${\cal N}_n$}

\begin{thm}
\label{dichotomy-tracial-finite-dimnuc-finite-infinite}
Let $A$ be a simple separable unital $C^*$-algebra {\blue{which is
asymptotically tracially in ${\cal C}_{{{\cal Z},s}}.$}}
{\blue{Then, either $A$ has stable rank one, or $A$ is purely infinite.
Moreover, if $A$  {\blue{is}}
asymptotically tracially in ${\cal C}_{{{\cal Z},s}}$
and is not purely infinite,
$A$ has strict comparison {{for positive elements}}.}}
\end{thm}

\begin{proof}
{\blue{Suppose that $A$ is a unital  separable simple \CA\, which is
asymptotically tracially in ${\cal C}_{{{\cal Z},s}}.$
Let ${\cal P}_1$ be the class of unital {{separable}}
{{simple ${\cal Z}$-stable}} \CA s which are
{{purely infinite}}
and {\blue{let}} ${\cal P}_2$ be the class of unital
{{separable}} simple ${\cal Z}$-stable \CA s
which have stable rank one.
Then either (I) or (II) hold:}}

(I):  For any finite subset $\calF\subset A$,  any $\epsilon>0$, and any
$a \in A_+\backslash\{0\}$,
there exists a unital
$C^*$-subalgebra
$B\subset  l^\infty(A)/c_0(A)$ which is strictly embedded
{\blue{such that}} $B$ in $\mathcal{P}_1$, {\blue{and}}

\hspace{0.2in}(1) $1_B\iota_A(x)\approx_{\epsilon}\iota_A(x)1_B$\, {\blue{for all}} $ x\in\calF$,

\hspace{0.2in}(2) $1_B\iota_A(x)1_B\in_{\epsilon} B$ and $\|1_B\iota_A(x)1_B\|\geq \|x\|-\epsilon$\,
{\blue{for all}} $ x\in \calF$, {\blue{and}}

\hspace{0.2in}(3)
$\iota_A(1_A)-1_B\lesssim_{l^\infty(A)/c_0(A)}\iota_A(a).$

{\blue{(II):}} The same {\blue{statement holds}} as {\blue{in}} (I) but replacing ${\cal P}_1$ by ${\cal P}_2.$

\noindent
{{We}} may assume that $A$ is infinite dimensional.
By Lemma \ref{L1991},  there is a sequence of nonzero positive elements
{{$\{d_n\}\subset A_+$}}
{\blue{such that}} $d_{n+1}\lesssim d_n$ for all $n\in \N,$ and,
for any $x\in A_+\setminus \{0\},$ there exists $N$ such that $d_n\lesssim x$
{\blue{for all $n\ge N.$}}
Let ${\cal F}_n\subset A$ be an increasing  sequence of finite subsets of $A$ whose union is dense
in $A.$  Since $A$ is
asymptotically tracially in ${\cal C}_{{{\cal Z},s}},$
by {{Theorem}}
\ref{Thm-tracial-simple-nuclear-Z-stable},  there exists a sequence of
decreasing positive numbers $\{\ep_n\}$ with $\lim_{n\to\infty} \ep_n=0$
{\blue{and a sequence of unital}}  \CA s  $B_k\in {\cal C}_{{\cal Z},s}$ {\blue{such that}}

{{$(1')$}} $\|1_{B_k}\iota_A(x)-\iota_A(x)1_{B_k}\|<\ep_k\rforal x\in {\cal F}_k;$

{{$(2')$}} $1_{B_k}\iota_A(x)1_{B_k}\in_{\ep_k} B_k$ for all $x\in {\cal F}_k,$ and


{{$(3')$}}  $\iota_A(1_A)-1_{B_k}\lesssim_{l^\infty(A)/c_0(A)}  {{\iota_A(d_k)}}.$


\noindent
{{If}} there are infinitely many $B_k$ which are purely infinite, then, {\blue{since,}} for any $a\in A_+\setminus \{0\},$
there is $K$ such that $d_K\lesssim a,$      (I) holds.

Otherwise, by \cite[Theorem 6.7]{Rordam-2004-Jiang-Su-stable}, (II) holds.
It follows  {\blue{from}} the proof of
{{Proposition}} \ref{Ptrpi=pis} (see {\blue{also}}
{{Remark}} \ref{Rm45}) that, if (I) holds, $A$ is purely infinite.
On the other hand, if (II) holds, by
{{Theorem}} \ref{Tstablerank1},  $A$ has stable rank one.
This completes the proof of the first part of the theorem.

{\blue{For the last part, by \cite[Theorem 4.5]{Rordam-2004-Jiang-Su-stable} and by Theorem \ref{tracially-Z-stable-has-almost-unperforated-Cuntz-semigroup},
$W(A)$ is almost unperforated.  Then, by the proof of
\cite[Corollary 4.6]{Rordam-2004-Jiang-Su-stable}, $A$ has strict comparison.
Note that the proof of \cite[Corollary 4.6]{Rordam-2004-Jiang-Su-stable} refers  to the proof of \cite[Theorem 5.2]{Rordam-1992-UHF2},
where quasitraces are used (see also \cite[Theorem 4.3]{Rordam-1992-UHF2}
and \cite[Theorem II.2.2]{BH-trace1982},
{\blue{as well as \cite[Proposition 2.1]{OPR}}}).}}

\end{proof}

\begin{cor}\label{CNn-sC}
Let $A$ be a simple separable unital $C^*$-algebra {\blue{which is
asymptotically tracially in ${\cal N}_n$
for some integer $n\ge {{0}}.$}}
Then, either $A$ has stable rank one, or $A$ is purely infinite.
Moreover, if $A$   is not purely infinite,
$A$ has strict comparison
{{for positive elements}}.
\end{cor}

\begin{proof}
We note, {\blue{by {{Corollary}}
\ref{id-finite-tdimnuc-equal-finite-dimnuc},}} that $A$ is
asymptotically tracially in {\blue{${\cal N}_{n,s,s},$}} where
{\blue{${\cal N}_{n, s,s}$}}
is the class of unital separable simple \CA s with  nuclear dimension
{{at most}} $n.$
By {\blue{\cite{W-2012-pure algebras}}}, \CA s in {\blue{${\cal N}_{n,s,s}$}} are  nuclear  simple ${\cal Z}$-stable \CA s.
Thus {{Theorem}}
\ref{dichotomy-tracial-finite-dimnuc-finite-infinite} applies.

\end{proof}

\begin{thm}
\label{tdimnuc-finite-equal-t-nuclear-Z-stable}

Let $A$ be a simple separable infinite dimensional unital
$C^*$-{\blue{algebra.}}
{\blue{Then}} the following are equivalent:

(1) $A$ is asymptotically tracially  {\blue{in ${\cal N}_n$}}
for some $n\in\mathbb{N}\cup\{0\}$,

(2) $A$ is asymptotically tracially in ${\cal N}_{n,s,s}$
{{for some $n\in\mathbb{N}\cup\{0\}$}},

(3) $A$ is asymptotically tracially  in {\blue{${\cal N}_{{\cal Z},s,s},$}}

(4) $A$ is asymptotically tracially in ${\cal N}$ and is
asymptotically tracially in ${\cal C}_{{{{\cal Z},s}}}.$


\end{thm}

\begin{proof}

{{
(1) $\Rightarrow$ (2) follows from
Theorem \ref{simple-T-finite-dimnuc-implies-T-simple-finite-dimnuc},
(2) $\Rightarrow$ (3) follows from \cite[Corollary 7.3]{W-2012-pure algebras},
(3) $\Rightarrow$ (4) is trivial,
(4) $\Rightarrow$ (1) follows from
Theorem \ref{Thm-tracial-simple-nuclear-Z-stable}
and \cite[Theorem A]{CETWW-2019}{\green{.}}
}}
\end{proof}


\begin{lem}[{see \cite[Lemma 8.2]{LinLAH}}]
\label{82LAH}
{\blue{Let $A$ be a unital separable nuclear simple {\blue{\CA}} which is asymptotically tracially  ${\cal N}_{d,s}$
(for some integer $d\ge 0$).
Then, for any integer $k\ge 1,$ there is a sequence of order
zero   c.p.c.~maps $L_n: M_k\to A$ such
that $\{L_n(e)\}$ is a central sequence
of $A$ for a minimal projection $e\in M_k$ and such that,
for every integer $m\ge 1,$}}
\beq\label{82-n1}
\lim_{n\to\infty} \max_{\tau\in T(A)}\{|\tau(L_n(e)^m)-1/k|\}=0.
\eneq
\end{lem}
\begin{proof}
The proof follows the same lines of that of \cite[Lemma 8.2]{LinLAH}
with some minor modification{\green{s}}.
{{Fix $k\in\N.$}}
Fix a dense subset $\{x_1, x_2,{\green{\cdots}}\}$ of the unit ball of $A$ and let ${\cal F}_n=\{x_1,
{\green{x_2,\cdots}},x_n\}$ with
$1_A=x_1$ ($n\in \N$).
Let $\gamma_n>0$ be {\blue{in}}  the fifth line of the proof of \cite[Lemma 8.2]{LinLAH}.
By {{Lemma}} \ref{L1991}, there is a sequence $\{a_n\}$ of $A_+\setminus \{0\}$
{\blue{such that}} $0<d_\tau(a_n)<1/4n^2$
{{($n\in\N$).}}
By {{Corollary}} \ref{id-finite-tdimnuc-equal-finite-dimnuc},
A is asymptotically tracially in ${\cal N}_{d,s,s}.$
Therefore,  by {{Proposition}} \ref{hereditary-subalgebra-preserves-tracially-approximation},  for each $n\in \N,$ there exists a $C^*$-algebra $B_n\in {\cal N}_{d,s,s}$
and c.p.c maps
$\alpha_n: A\rightarrow B_n,$
$\beta_{n,j}: B_n\rightarrow A$, {{and}}
$\gamma_{n,j}: A\rightarrow A\cap \bt_{n,j}(B_n)^\perp$
(${{j}}
\in\mathbb{N}$){\green{, such}} that

(1) $\gamma_{n,j}(1_A)$ and $p_{n,j}:=\bt_{n,j}(1_{B_n})$ are projections, $1=\gamma_{n,j}(1_A)+\bt_{n,j}(1_{B_n}),$
and
$\af_n(1_A)=1_{B_n} ,$

(2) $x_i\approx_{\gamma_n/2^n}\gamma_{n,j}(x_i)+\beta_{n,j}\circ\alpha_n(x_i)$
{{for all}} $1\le i\le n$
and all $j\in\mathbb{N},$

(3) $\alpha_n$ is {\blue{an}} $(\calF_n, 1/2^n)$-{{approximate embedding}},

(4) $\lim_{j\rightarrow\infty}\|\beta_{n,j}(xy)-\beta_{n,j}(x)\beta_{n,j}(y)\|=0$
{{and}}
{$\lim_{j\rightarrow \infty}\|\beta_{n,j}(x)\|= \|x\|$}
for all $x,y\in {\blue{B_n,}}$ {{and}}

(5) $\gamma_{n,j}(1_A)\lesssim_A a_n$ for all $ j\in\mathbb{N}$.

\noindent
{{Note}} that
{\blue{one also {{has}}}}

(6) $\|p_{n,j}x-xp_{n,j}\|<1/2^{n-1}$ for all $x\in {\cal F}_n.$

\noindent
By \cite[Lemma 5.11]{W-2012-pure algebras}
(since $B\in {\cal N}_{d,s,s}$),
for each {{$n,$}} there 
{\blue{is an}} order zero c.p.c.~map {{$\Psi_{n}: M_k\to B_n$}}
such that
\beq\label{82-1}
&&
\|[\Psi_{n}(c), \alpha_n(x)]\|<1/2^n \rforal
c\in M_k^1\andeqn x\in\calF_n,\andeqn\\
\label{f9-4-2}
&&\inf\{\tau(\Psi_{n}(1_{M_k})):\tau\in T(B_n)\}>1-1/4n.
\eneq
Consider, for each $m,$ $\Psi_{n, m}=\bt_{n,m}\circ \Psi_{n}: M_k\to p_{n,m}Ap_{n,m}.$
Note that, by (4), for each $n\in\N,$ there exists $\bar m(n)\in\N$ such {\blue{that,}}
for all $m\geq \bar m(n),$
$\Psi_{n, m}$ is a $\Delta(M_k,\gamma_n/2^n)$-almost order zero map
(recall {{Definition}}
\ref{almost-n-dividable-error} for $\Delta(-,-)$),
and
\beq\label{f82-1}
&
\|[\beta_{n,m}\circ\Psi_{n}(c), \beta_{n,m}\circ\alpha_n(x)]\|<\gamma_n/2^n \rforal
c\in M_k^1\andeqn
x\in\calF_n.
\eneq

\noindent
{{Claim:}} For fixed $n,$ there is $m(n)>\bar m(n)$ such that,
for all $m>m(n),$
$$\inf\{\tau(\Psi_{n,m}(1_{M_k})):\tau\in T(A)\}\ge 1-1/2n.$$

\noindent
{{Otherwise,}} there would be a subsequence
$\{m(l)\}$ and $\tau_{m(l)}\in T(A)$ such {{that}}
$$\tau_{m(l)}\circ \bt_{n,m(l)}\circ \Psi_{n}(1_{M_k})<1-1/2n.$$
Let $t_0$ be a {{weak*-limit}} of
the sequence of
contractive positive linear functionals 
$\{\tau_{m(l)}\circ \bt_{n,m(l)}\}$ of $B_n.$  Then $t_0(\Psi_{n}(1_{M_k}))\le 1-1/2n.$
On the other hand,
by (5) and (1),  $t_0(1_{B_n})\ge 1-1/4n^2.$
Moreover, by (4), $t_0$ is a
{{positive tracial functional}}
with $\|t_0\|\ge 1-1/4n^2.$
It follows {\blue{from}} \eqref{f9-4-2} {\blue{that}}
$t_0(\Psi_{n}(1_{M_k}))\ge (1-1/4n^2)(1-1/4n)>1-1/2n.$
This proves the claim.


For all $c\in M_k^1$ and $x\in\calF_n$,
one has
\beq
\label{f9-4-1}
\hspace{-0.2in}
\|[\bt_{n,m(n)}\circ \Psi_{n}(c),x]\|
&\overset{\rm(by~(2))}{\approx_{\gamma_n/2^{n-1}}}&
\|[\bt_{n,m(n)}\circ \Psi_{n}(c),
\gamma_{n,m(n)}(x)+\beta_{n,m(n)}\circ\alpha_n(x)]\|
\nonumber\\
&\overset{\rm(by~(1))}{=}&
\|[\bt_{n,m(n)}\circ \Psi_{n}(c),\beta_{n,m(n)}\circ\alpha_n(x)]\|
\overset{\rm(by~\eqref{f82-1})}{<}\gamma_n/2^n.
\eneq

\noindent
{{Since}} $\Psi_{n, m(n)}$ is a $\Delta(M_k,\gamma_n/2^n)$-almost order zero map,
by the {\blue{choice}} of $\Delta(M_k,\gamma_n/2^n)$
(see {{Definition}}
\ref{almost-n-dividable-error} and {{Proposition}} \ref{n-almost-dividable-exact}),
one obtains a sequence of order zero c.p.c.~maps
$\Phi_{n}: M_k\to A$ such that
\beq
\|\Phi_{n}-\bt_{n,m(n)}\circ \Psi_{n}\|\leq \gamma_n/2^n
\rforal n\in\N.
\eneq
By \eqref{f9-4-1}, as well as the claim, for $n\geq 3$,
{{one has}}
\beq
&\|\Phi_n(c)x-x\Phi_n(c)\|<\min\{1/4n, \gamma_n\} \rforal c\in M_k^1\andeqn x\in {\cal F}_n,\andeqn\\
&\inf\{\tau(\Phi_n(1_{M_k})): \tau\in T(A)\}\ge 1-1/n.
\eneq
There is a \hm\, $\phi_n: C_0((0,1])\otimes M_k\to A$ such
that $\Phi_n(c)=\phi_n(\iota\otimes a)$ for all $c\in M_k,$ where
$\iota(t)=t$ for all $t\in (0,1].$
Let $c_n=\iota^{1/n}.$
Define $L_n(c)=\phi_n(c_n\otimes c)$ for all $c\in M_k.$
It is {\blue{an}} order zero  c.p.c.~map from $M_k$ to $A.$  Choose a minimal
projection $e_1\in M_k.$ Then
\beq
(L_n(e_1))^m=\phi_n(c_n^m\otimes e_1)=\phi_n(\iota\otimes e_1)^{m/n})=\Phi_n(e_1)^{m/n}.
\eneq
One then verifies that, for this $L_n,$ \eqref{82-n1} holds
exactly the same way as the proof of \cite[Lemma 8.2]{LinLAH}.
\end{proof}


\begin{thm}\label{Cnn=1}
{\blue{Every unital separable simple nuclear \CA\, which is
asymptotically tracially in ${\cal N}_n$
is ${\cal Z}$-stable and has nuclear dimension {{at most}} 1.}}

{\blue{On the other hand, every unital separable simple nuclear \CA\, which is asymptotically tracially in
${\cal C}_{{\cal Z},s}$ also has nuclear dimension at most 1.}}
\end{thm}

\begin{proof}
Let $A$ be a unital separable simple nuclear \CA\, which is
asymptotically tracially in ${\cal N}_n$
{\blue{for some non-negative integer $n.$}}
{\blue{By {{Theorem}}
\ref{tdimnuc-finite-equal-t-nuclear-Z-stable},}}
{{$A$ is asymptotically tracially in ${\cal N}_{n,s,s}.$}}
By {{Corollary}}
\ref{CNn-sC},  and by
{{\cite[Corollary 9.9]{BBSTWW},}}
we may assume that $A$ has stable rank one and has strict comparison.
We first prove that {\blue{$A$ is ${\cal Z}$-stable.}}
The proof of this  is exactly the same as that of
\cite[Theorem 8.3]{LinLAH} but using
{{Lemma}} \ref{82LAH}
(By the exactly the same argument for the proof of (ii) implies (iii) in \cite{MS2012_strict_comparison_and_z stability},
using {{Lemma}} \ref{82LAH} instead of
\cite[Lemma 3.3]{MS2012_strict_comparison_and_z stability},
 one concludes that any c.p.~map from $A$ to $A$ can be excised in small central sequence.
As in \cite{MS2012_strict_comparison_and_z stability}, this implies that $A$ has property (SI). Using
{{Lemma}} \ref{82LAH}, the same proof that (iv) implies (i) in \cite{MS2012_strict_comparison_and_z stability} shows
that $A$ is ${\cal Z}$-stable{{)}}.

Then, by {{\cite[Theorem A]{CETWW-2019}}},  
{{$A$}} has finite nuclear dimension.
It follows from {{\cite[Theorem B]{CETWW-2019}}} that 
{{$A$}} has in fact nuclear dimension at most 1.

{\blue{Finally, the last statement follows  the first part of the statement and part (4) of {\green{Theorem \ref{tdimnuc-finite-equal-t-nuclear-Z-stable}.}}}}

\end{proof}


\begin{cor}[cf. Appendix of \cite{GLII}]
\label{Cnrr0=Z}
{\blue{Every unital separable simple nuclear \CA\,
which has generalized tracial rank at most one
is ${\cal Z}$-stable. }}
\end{cor}

\begin{lem}
\label{tensor-product-of-strictly-embedding}

Let $A,B,C_i,D_i$ be \CAs\ ($i\in\N$),
and let $\alpha_i:A\rightarrow C_i$, $\beta_i:B\rightarrow D_i$
be c.p.c.~maps
{\blue{such}} that
\beq\nonumber
{\blue{\alpha: A
\rightarrow
\prod_{i=1}^{\infty}C_i/ \bigoplus_{i=1}^{\infty}C_i,\,\,\,\,
a\mapsto
\pi_\infty(\{\alpha_i(a)\}_i)\andeqn}}
{\blue{\beta: B
\rightarrow
\prod_{i=1}^{\infty}D_i/ \bigoplus_{i=1}^{\infty}D_i,\,\,\,\,
b\mapsto
\pi_\infty(\{\beta_i(b)\}_i)}}
\eneq
are *-homomorphisms. {\blue{Then}} the following map is
also a *-homomorphism:
\beq
{\blue{\gamma: A\otimes B
\rightarrow
\prod_{i=1}^{\infty}(C_i\otimes D_i)/ \bigoplus_{i=1}^{\infty}(C_i\otimes D_i),\,\,\,\,
a\otimes b\mapsto
\pi_\infty(\{\alpha_i(a)\otimes \beta_i(b)\}_i).}}
\eneq

\noindent
{{If,}} in addition, both $\alpha$ and $\beta$ are {\blue{strict}} embeddings,
{\blue{so is}} $\gamma.$

\end{lem}

\begin{proof}
Note that $\alpha_i\otimes \beta_i: A\otimes B\rightarrow C_i\otimes D_i,$
$a\otimes b\mapsto \alpha_i(a)\otimes \beta_i(b)$ are c.p.c.~maps.
{{Thus}} $\gamma$ is also a c.p.c.~map.
Fix $\check{a}\in A$, $\check{b}\in B.$ {\blue{Since}} $\alpha$ {\blue{and}} $\beta$ are *-homomorphisms,
we have
\beq
\label{f10-9lem-1}
\lim_{i\rightarrow\infty}
\|\alpha_i(\check{a}\check{a}^*)-\alpha_i(\check{a})\alpha_i(\check{a})^*\|
+\|\beta_i(\check{b}\check{b}^*)-\beta_i(\check{b})\beta_i(\check{b})^*\|=0.
\eneq
Then
\beq
\gamma(\check{a}\otimes \check{b})\cdot \gamma(\check{a}\otimes \check{b})^*
&=&
{\blue{\pi_\infty(\{\alpha_i(\check{a})\otimes \beta_i(\check{b})\}_i)
\cdot
\pi_\infty(\{\alpha_i(\check{a})\otimes \beta_i(\check{b})\}_i)^*}}
\\
&=&
{\blue{\pi_\infty(\{(\alpha_i(\check{a})\alpha_i(\check{a})^*)\otimes (\beta_i(\check{b})\beta_i(\check{b})^*)\}_i)}}
\\
\mbox{(by \eqref{f10-9lem-1})}&=&
{\blue{\pi_\infty(\{(\alpha_i(\check{a}\check{a})^*)\otimes (\beta_i(\check{b}\check{b})^*)\})}}
=\gamma((\check{a}\check{a}^*)\otimes (\check{b}\check{b}^*))
\\
&=&
{\blue{=\gamma((\check{a}\otimes \check{b})\cdot (\check{a}\otimes \check{b})^*).}}
\eneq

\noindent
{{Similarly,}} we have
$\gamma(\check{a}\otimes \check{b})^*\cdot \gamma(\check{a}\otimes \check{b})
=\gamma((\check{a}\otimes \check{b})^*\cdot (\check{a}\otimes \check{b}))$
{{(see, for example,
\cite[Proposition 1.5.7.(ii)]{BO 2008})}}.
Thus $\check{a}\otimes \check{b}$ lies in the multiplicative domain of
$\gamma$. Since the linear span of elementary tensor products
is dense in
$A\otimes B$, we see that
$A\otimes B$ lies in the multiplicative domain of $\gamma$.
{{In other words,}} $\gamma$ is a *-homomorphism.

Assume in addition both $\alpha$ and $\beta$ are
{{strict embeddings.}}
If $\gamma$ is not a
{{strict embedding,}}
then there {{exist}} $z_0\in A\otimes B,$  $\epsilon>0,$
and {\blue{a subsequence}} $\{m_i\}_{i\in\N}\subset\N$ such that
\beq
\label{f10-9lem-2}
\limsup_{i\rightarrow\infty}
\|\alpha_{m_i}\otimes \beta_{m_i}(z_0)\|\leq\|z_0\|-\epsilon.
\eneq
By what {\blue{has been}} proved,
the following map is
also a *-homomorphism:
\beq
\tilde{\gamma}: A\otimes B
\rightarrow
\prod_{i=1}^{\infty}(C_{m_i}\otimes D_{m_i})
/ \bigoplus_{i=1}^{\infty}(C_{m_i}\otimes D_{m_i}),\,\,\,
{\blue{a\otimes b \mapsto
\pi_\infty(\{\alpha_{m_i}(a)\otimes \beta_{m_i}(b)\}_i).}}
\eneq

\noindent
{{By}} \eqref{f10-9lem-2}, $\tilde\gamma$ is not an isometry.
{{Thus}} $\tilde \gamma$ could not be injective.
By \cite[Lemma 2.12(ii)]{BK2004},
{{$\ker \tilde\gamma$ (which is an ideal of $A\otimes B$)}}
contains a nonzero elementary tensor product
$a_0\otimes b_0$.
Then
\beq\nonumber
0=\|\tilde\gamma(a_0\otimes b_0)\|
&=&\limsup_{i\rightarrow \infty}
(\|\alpha_{m_i}(a_0)\otimes\beta_{m_i}(b_0)\|)
\\\nonumber&=&
\limsup_{i\rightarrow \infty}\|\alpha_{m_i}(a_0)\|\cdot\|\beta_{m_i}(b_0)\|
\\\nonumber
(\alpha,\  \beta\mbox{ are {{strict}} embeddings})&=&
\lim_{i\rightarrow \infty}\|\alpha_{m_i}(a_0)\|\cdot
\lim_{i\rightarrow \infty}\|\beta_{m_i}(b_0)\|
{\blue{=\|a_0\|\cdot\|b_0\|,}}
\eneq
which is contradict to {\blue{the assumption that}} $a_0\otimes b_0\neq 0$.
Hence $\gamma$ is a {{strict}} embedding.

\end{proof}

\begin{lem}
\label{tensor-product-of-appro-embed-lem}

Let $A_1$ {\blue{and}}  $A_2$ be \CAs\, {\blue{and let}}
$\calF\subset A_1\otimes A_2$ be a finite subset.
Then, for any $\epsilon>0$,
there {\blue{exist}} finite subsets $\mathcal{G}_i\subset A_i$ ($i=1,2$)
and $\delta>0$ such that,  {\blue{for any \CA s $B_1$ and $B_2,$ and,}}
for any {\blue{c.p.c.~maps}}
$\alpha_i:A_i\rightarrow B_i$
{\blue{which are}}
$(\mathcal{G}_i,\delta)$-{\blue{approximate embeddings,}}
{\blue{the map}} $\alpha_1\otimes \alpha_2: A_1\otimes A_2\rightarrow B_1\otimes B_2$
is {\blue{an}} $(\calF,\epsilon)$-{{approximate embedding}}.

\end{lem}

\begin{proof}

Without loss of generality, we may assume that $0<\epsilon<1$.
Let $M:=1+\max\{\|x\|: x\in\calF\}$.
Let $\calF_i\subset A_i$ ($i=1,2$) and  {\blue{$n$ be some integer}}
{\blue{such that}} $\calF_i=\calF_i^*$ ($i=1,2$) and
$\calF\subset_{\frac{\epsilon}{8M^2}}{\blue{ {\cal F}^{1,2}}},$
{\blue{where}}
${\blue{{\cal F}^{1,2}:=\{\sum_{i=1}^n x_i\otimes y_i: x_i\in {\cal F}_1\andeqn
y_i\in {\cal F}_2\}.}}$
Let $M_1:=1+\max\{\|x\|:x\in {\blue{{\cal{F}}^{1,2}}}\}.$

Keeping Lemma \ref{c.p.c.-almost-multiplicative-easy} in mind,
it is straightforward to see that
there exists $\delta_0>0$ such that,
for any c.p.c.~maps
$\alpha_i:A_i\rightarrow B_i$ ($i=1,2$, $B_i$ are \CAs),
if $\alpha_i$ is ($\calF_i,\delta_0$){{-multiplicative}} ($i=1,2$),
then $\alpha_1\otimes \alpha_2: A_1\otimes A_2\rightarrow B_1\otimes B_2$
is $
({\blue{\calF^{1,2}}},\frac{\epsilon}{8})${{-multiplicative}},
and, hence
$\alpha_1\otimes \alpha_2: A_1\otimes A_2\rightarrow B_1\otimes B_2$
is $(\calF, \epsilon)${{-multiplicative}}.
Let $\calF_i\subset \calF_{i,1}\subset \calF_{i,2}\subset\cdots $
be finite subsets of $C^*(\calF_i)$  such that  $\cup_{j\in\N}\calF_{i,j}$
is dense in $C^*(\calF_i)$ ($i=1,2$).

{\blue{Now let us}}  assume the lemma does not hold. {\blue{Then}} there exists
a sequence of
\CAs\ $B_{i,m}$ and c.p.c.~maps $\alpha_{i,m}: A_i\rightarrow B_{i,m}$
{{such}} that $\alpha_{i,m}$
{{is an}}
$(\calF_{i,m},\delta_0/m)$-{\blue{approximate embedding}}
({{$i=1,2$}}, {{$m\in\N$}}),
and
$\alpha_{1,m}\otimes \alpha_{2,m}: A_1\otimes A_2
\rightarrow B_{1,m}\otimes B_{2,m}$
{\blue{is}} not {\blue{an}}  $(\calF,\epsilon)$-{{approximate embedding}} ($m\in\N$).
%
{\blue{However, since
$\calF_i\subset\calF_{i,m},$ {\blue{by}}
the choice of $\delta_0$,
and
{\blue{by the fact that}}
$\alpha_{1,m}\otimes \alpha_{2,m}$ {is}
$(\calF,\epsilon)$-approximate multiplicative,}}
{\blue{for each $m,$ there must be some
$z_m\in\calF\subset \calF^{1,2}$
such that}}
%
\beq
\label{f10-10lem-1}
\|\alpha_{1,m}\otimes\alpha_{2,m}(z_m)\|<\|z_m\|-\epsilon.
\eneq

\noindent
{{Since}} {\blue{$\calF^{1,2}$}}
is a finite subset,
by \eqref{f10-10lem-1},
there exists $z_0\in {\blue{\calF^{1,2}}}$
and an increasing sequence $\{m_j\}\subset\N$
such that
\beq
\label{f10-10lem-2}
\|\alpha_{1,m_j}\otimes\alpha_{2,m_j}(z_0)\|<\|z_0\|-\epsilon
{\blue{\rforal}} j\in \N.
\eneq
Note that
{\blue{the map}}
{\blue{$\bar{\alpha}_1: A_1
\rightarrow
\prod_{m=1}^{\infty}B_{1,m}/ \bigoplus_{m=1}^{\infty}B_{1,m}$ defined by
$a\mapsto
\pi_\infty(\{\alpha_{1,m}(a)\})$ and {\blue{the map}}
$\bar{\alpha}_2: A_2
\rightarrow
\prod_{m=1}^{\infty}B_{2,m}/ \bigoplus_{m=1}^{\infty}B_{2,m}$
defined by
$a\mapsto
\pi_\infty(\{\alpha_{2,m}(a)\})$}}
are {\blue{strict}} embeddings.
Then,
by Lemma \ref{tensor-product-of-strictly-embedding},
the following
{\blue{is also a strict}} embedding:
{\blue{\beq\nonumber
\gamma: A_1\otimes A_2
\rightarrow
\prod_{m=1}^{\infty}(B_{1,m}\otimes B_{2,m})
/ \bigoplus_{m=1}^{\infty}(B_{1,m}\otimes B_{2,m}),\,\,\,
a\otimes b\mapsto
\pi_\infty(\{\alpha_{1,m}(a)\otimes \alpha_{2,m}(b)\}).
\eneq}}
{\blue{But this contradicts with \eqref{f10-10lem-2}.}}
{\blue{
The lemma  then follows.}}

\end{proof}

\begin{nota}
Let $\mathcal{X}_1$, $\mathcal{X}_2$ be two classes of \CAs.
{\blue{Denote}}
$\mathcal{X}_1\otimes \mathcal{X}_2:=
\{A\otimes B: A\in\mathcal{X}_1, B\in\mathcal{X}_2\}$,
where {\blue{each}} $A\otimes B$ is the spatial tensor product.
\end{nota}

Recall the following result
(see \cite[Lemma 2.15]{BK2004},
also see \cite[Lemma 4.1.9]{Rordam 2002}):

\begin{lem}
[Kirchberg's Slice Lemma]
\label{Kir-slice-lem}
Let $A$ {\blue{and}} $B$ be \CAs,
and let $D$ be a nonzero hereditary $C^*$-subalgebra of
the spatial tensor product $A\otimes B$.
Then there exists a nonzero element $z\in A\otimes B$
such that $z^*z=a\otimes b$ for some $a\in A$, $b\in B$,
and $zz^*\in D$.

\end{lem}

\begin{thm}
\label{tensor-of-tracial-approximation}
Let $\mathcal{X}_1$, $\mathcal{X}_2$ be two classes of \CAs.
Let $A$ {\blue{and}} $B$ be unital simple separable infinite dimensional \CAs.
Assume that $A$ is asymptotically tracially in $\mathcal{X}_1$
and $B$ is asymptotically tracially in $\mathcal{X}_2$.
Then the spatial tensor product $A\otimes B$
is asymptotically tracially in $\mathcal{X}_1\otimes \mathcal{X}_2$.

\end{thm}

\begin{proof}


Let $\calF\subset A\otimes B,$
{\blue{let}} $c\in (A\otimes B)_+\backslash\{0\}$,
and let ${\blue{1/4>}}\epsilon>0$.
By Kirchberg's Slice Lemma ({{see}} Lemma \ref{Kir-slice-lem}),
there exists $a\in A_+\backslash\{0\}$ and {\blue{there}} exists $b\in B_+\backslash\{0\}$,
such that
\beq
a\otimes b\lesssim_{A\otimes B} c.
\eneq

\noindent
{{Note that}} {\blue{$A$ and $B$ are {{non-elementary.}}}}
{{Then one}} 
{\blue{may}} choose $\check{a},\hat{a}\in\Her_A(a)_+\backslash\{0\}$
{\blue{and}} $\check{b},\hat{b}\in\Her_A(b)_+\backslash\{0\}$
such {\blue{that}} $\check{a}\bot\hat{a},$
 $\check{a}\sim_A\hat{a},$
 $\check{b}\bot\hat{b},$
and $\check{b}\sim_A\hat{b}$ {\blue{(see {{Lemma}} \ref{L1991}, for example).}}

Since $A$ {{and}} $B$ are simple and unital, there {\blue{exist}} $k\in \N,$
$r_1,{\green{r_2,}}\cdots,r_k\in A,$ {\blue{and}}
$s_1,{\green{s_2,}}\cdots,s_k\in B$  such that
$1_A=\sum_{i=1}^k r_i^* {{\hat{a}}} r_i$
and $1_B=\sum_{i=1}^k s_i^*{{\hat{b}}} s_i$.
Since $A$ {{and}} $B$ are simple {{and}} infinite dimensional,
so {\blue{are}} $\Her_A(\check{a})$ and $\Her_B(\check{b}).$ {\blue{Then}}
{\blue{(see {{Lemma}} \ref{L1991})}}
there exist mutually orthogonal positive elements
$a_1,{\green{a_2,}}\cdots, a_{k+1}\in \Her_A(\check{a})_+\backslash\{0\}$
and mutually orthogonal positive elements
${\blue{b_1,}} {\blue{b_2,}} \cdots, b_{k+1}
\in \Her_A(\check{b})_+\backslash\{0\}$
{\blue{such that}} $a_1\sim_A a_2\sim_A
{{\cdots}} \sim_A a_{k+1}$
and $b_1\sim_B b_2\sim_B
{{\cdots}} \sim_B b_{k+1}$
({\blue{recall}} Definition \ref{Dcuntz} for the definition of ``$\sim$'').

Let $N\in\N$, let
$\tilde\calF_1\subset A,$ and {\blue{let}} $\tilde\calF_2\subset B$ be finite subsets
such that
\beq
\calF\subset_{\epsilon/2}
\left\{\sum_{i=1}^N x_i\otimes y_i:x_i\in\tilde\calF_1,
y_i\in\tilde\calF_2\right\}.
\eneq

\noindent
{{By}} Lemma \ref{tensor-product-of-appro-embed-lem},
there exist finite subsets $\bar\calF_1\subset A$,
$\bar\calF_2\subset B,$ {\blue{and}} $\delta_0>0$ such that,
for any \CAs\ $D_1,D_2$ and any c.p.c.~maps
$\eta_1:A\rightarrow D_1$,  {\blue{and}} $\eta_2:B\rightarrow D_2,$
if $\eta_i$ is {\blue{an}} $(\bar\calF_1,\delta_0)$-{{approximate embedding}} ($i=1,2$),
then $\eta_1\otimes \eta_2:A\otimes B\rightarrow D_1\otimes D_2$
is  {\blue{an}} $(\calF,\epsilon)$-{{approximate embedding}}.

Let $\calF_i:=\tilde\calF_i\cup \bar\calF_i$, $i=1,2$.
Let $M:=1+\max\{\|x\|:x\in \calF_1\cup\calF_2\}$.
{\blue{Choose}} $\delta:=\min\{\delta_0,\frac{\epsilon}{3(N+1)(M+1)^2}\}$. {\blue{Note $2\dt^2<\dt<M/4.$}}

Since $A$ is asymptotically tracially in $\mathcal{X}_1$,
there {{exist}} a $C^*$-algebra $C_1$ in $\mathcal{X}_1$,
and c.p.c maps
$\alpha: A\rightarrow C_1,$
$\beta_i: C_1\rightarrow A$, {\blue{and}}
$\gamma_i: A\rightarrow A$
{\blue{($i\in\mathbb{N}$)}} such that


(1) $x\approx_{\delta}\gamma_i(x)+\beta_i\circ\alpha(x)$
{\blue{for all $x\in\calF_1$ and  for all}} $i\in\mathbb{N},$

(2) $\alpha$ is {\blue{an}} $(\calF_1, \delta)$-{{approximate embedding}},

(3) $\lim_{i\rightarrow\infty}\|\beta_i(xy)-\beta_i(x)\beta_i(y)\|=0$ {\blue{and}}
{$\lim_{n\rightarrow \infty}\|\beta_i(x)\|= \|x\|$}
{\blue{for all}} $x,y\in C_1,$ {{and}}

(4) $\gamma_i(1_A)\lesssim_A a_1{\blue{(\sim a_{k+1})}}$ {\blue{for all}} $i\in\mathbb{N}$.

\noindent
{{Since}} $B$ is asymptotically tracially in $\mathcal{X}_2$,
there {\blue{exist}} a $C^*$-algebra $C_2$ in $\mathcal{X}_2$,
and c.p.c maps
$\phi: B\rightarrow C_2,$
$\psi_i: C_2\rightarrow B$, {{and}}
$\theta_i: B\rightarrow B$
($i\in\mathbb{N}$) {\blue{such}} that


{{$(1')$}} $x\approx_{\delta}\theta_i(x)+\psi_i\circ\phi(x)$
{\blue{for all}} $x\in\calF_2$ {\blue{and for all}} $i\in\mathbb{N},$

{{$(2')$}} $\phi$ is {\blue{an}} $(\calF_2, \delta)$-{{approximate embedding}},

{{$(3')$}} $\lim_{i\rightarrow\infty}\|\psi_i(xy)-\psi_i(x)\psi_i(y)\|=0$ {\blue{and}}
{$\lim_{n\rightarrow \infty}\|\psi_i(x)\|= \|x\|$}
{\blue{for all}} $x,y\in C_2,$ {{and}}

{{$(4')$}} $\theta_i(1_B)\lesssim_B b_1$ {\blue{($\sim b_{k+1}$) for all}} $i\in\mathbb{N}$.

\noindent
{{Note}} that $C_1\otimes C_2$ is in $\mathcal{X}_1\otimes \mathcal{X}_2$.
Now define {\blue{a}} c.p.c.~map
\beq
\rho:=\alpha\otimes \phi: A\otimes B\rightarrow C_1\otimes C_2.
\eneq
By {\blue{(2),}} {{$(2')$}} and
{{by}} the choice of $\delta$, $\delta_0$,
$\calF_1$, $\bar\calF_1$, $\calF_2$ and $\bar\calF_2$,
{\blue{the map}} $\rho$ is {\blue{an}} $(\calF,\epsilon)$-{{approximate embedding}}.
{\blue{Hence}} (2) of Definition \ref{def-tracial_approximation} {\blue{holds.}}

For $i\in\N$, define {\blue{a}} c.p.c.~map
\beq
\omega_i:=\beta_i\otimes \psi_i: C_1\otimes C_2\rightarrow A\otimes B.
\eneq
{\blue{\noindent
{{Define}} c.p.c.~maps $\beta: C_1
\rightarrow
l^\infty(A)/c_0(A)$ by
$x\mapsto
\pi_\infty(\{\beta_1(x),\beta_2(x),\cdots\})$ and   $\psi: C_2
\rightarrow
l^\infty(B)/c_0(B)$ by
$x\mapsto
\pi_\infty(\{\psi_1(x),\psi_2(x),\cdots\}),$ respectively.}}
{\blue{Then,}} by (3) and {{$(3')$}},
$\beta$ and $\psi$
are {{strict}} embeddings.
{\blue{By Lemma \ref{tensor-product-of-strictly-embedding}, the map
$\omega: C_1\otimes C_2
\rightarrow
l^\infty(A\otimes B)/c_0(A\otimes B)$ defined by
$x\otimes y\mapsto
\pi_\infty(\{\beta_1(x)\otimes\psi_1(y),
\beta_2(x)\otimes\psi_2(y),\cdots\})$
is also a strict
embedding:}}
This is equivalent to say {{that}}
(3)  of Definition \ref{def-tracial_approximation} {\blue{holds.}}

Note that by (1) and {{$(1')$}} {\blue{above,}} for $i\in\N$,  {\blue{one has}}
\beq\nonumber
1_A\otimes1_B
&\approx_{{\blue{2\delta(1+\dt)}}}&
(\beta_i\circ\alpha(1_A)+\gamma_i(1_A))
\otimes
(\psi_i\circ\phi(1_B)+\theta_i(1_B))
\\\nonumber
&=&
\beta_i\circ\alpha(1_A)
\otimes
\psi_i\circ\phi(1_B)
\\\nonumber
&&
+
\gamma_i(1_A)
\otimes
\psi_i\circ\phi(1_B)
+
\beta_i\circ\alpha(1_A)
\otimes
\theta_i(1_B)
+
\gamma_i(1_A)
\otimes
\theta_i(1_B).
\eneq
{\blue{Thus}}
$$\|\gamma_i(1_A)
\otimes
\psi_i\circ\phi(1_B)
+\beta_i\circ\alpha(1_A)
\otimes
\theta_i(1_B)
+\gamma_i(1_A)
\otimes
\theta_i(1_B)
\|\leq  {\blue{1+2\dt+2\dt^2<1+3\dt.}}$$
{\blue{It follows that  the map defined below}}
\beq\nonumber
\sigma_i: A\otimes B&\rightarrow& A\otimes B, \\\nonumber
x\otimes y
&\mapsto&
\frac{1}{1+3\delta}
\left(\gamma_i(x)
\otimes
\psi_i\circ\phi(y)
+
\beta_i\circ\alpha(x)
\otimes
\theta_i(y)
+
\gamma_i(x)
\otimes
\theta_i(y)
\right)
\eneq
{\blue{is}}
{\blue{c.p.c.~{\blue{map}} ($i\in\N$).}}
By (1) and {{$(1')$}},
for $x\in\calF_1$ and $y\in\calF_2$, and for any $i\in \N$,
{\blue{one has}}
%
\beq
x\otimes y
&\approx_{\blue{2\dt(M+\dt)}}&
(\beta_i\circ\alpha(x)+\gamma_i(x))
\otimes
(\psi_i\circ\phi(y)+\theta_i(y))\\
&=&
\beta_i\circ\alpha(x)
\otimes
\psi_i\circ\phi(y)
+
\gamma_i(x)
\otimes
\psi_i\circ\phi(y)
\\
&&\quad+
\beta_i\circ\alpha(x)
\otimes
\theta_i(y)
+
\gamma_i(x)
\otimes
\theta_i(y)
\\&\approx_{3M^2\dt}
&
\omega_i\circ\rho(x\otimes y)+\sigma_i(x\otimes y).
\eneq
Then, for $\sum_{j=1}^N x_j\otimes y_j$ with $x_j\in\calF_1,$  $y_j\in \calF_2$
($j=1,{\green{2,}}\cdots, N$), and, for any $i\in\N$,   {\blue{one has}}
\beq
\sum_{j=1}^N x_j\otimes y_j
&\approx_{3NM(M+1)\delta}&
\sum_{j=1}^N \omega_i\circ\rho(x_j\otimes y_j)+\sigma_i(x_j\otimes y_j)
\\
&=&
\omega_i\circ\rho\left(\sum_{j=1}^Nx_j\otimes y_j\right)
+\sigma_i\left(\sum_{j=1}^Nx_j\otimes y_j\right).
\eneq
Thus, {\blue{by the choice of $\dt,$}}  (1) of Definition \ref{def-tracial_approximation} {\blue{holds.}}

{{Claim: For}} all $i\in\N$, $\sigma_i(1_A\otimes 1_B)\lesssim
c$  {\blue{in $A\otimes B.$}}
Indeed, {\blue{one has}}
\beq
\gamma_i(1_A)
\otimes
\psi_i\circ\phi(1_B)
\lesssim
a_1\otimes 1_B
=
a_1\otimes (\sum_{i=1}^k s_i^* {{\hat{b}}} s_i)
\lesssim
\sum_{i=1}^k a_i\otimes {{\hat{b}}},
\mbox{ and}\\
\beta_i\circ\alpha(1_A)
\otimes
\theta_i(1_B)
\lesssim
1_A
\otimes
b_1
=
(\sum_{i=1}^k r_i^* {{\hat{a}}} r_i)\otimes b_1
\lesssim
\sum_{i=1}^k{{\hat{a}}}\otimes b_i,
\mbox{ and }\\
\gamma_i(1_A)
\otimes
\theta_i(1_B)
\lesssim
a_1 \otimes b_1
\sim
a_{k+1}\otimes b_{k+1}.
\eneq

\noindent
{{Then}}

\beq\nonumber
\hspace{-0.5in}(1+3\delta)\sigma_i(1_A\otimes 1_B)
&=&
\gamma_i(1_A)
\otimes
\psi_i\circ\phi(1_B)
+
\beta_i\circ\alpha(1_A)
\otimes
\theta_i(1_B)
\\\nonumber
&&
\quad+
\gamma_i(1_A)
\otimes
\theta_i(1_B)
\\\nonumber
{\blue{\mbox{($\hat{a}\perp \check{a}, \hat{b}\perp \check{b},
i\not=j$)}}}&\lesssim
&
(\sum_{i=1}^ka_i\otimes {{\hat{b}}})
+
(\sum_{i=1}^k{{\hat{a}}}\otimes b_i)
+a_{k+1}\otimes b_{k+1}
\\\nonumber
&\lesssim
&
(\sum_{i=1}^{k}a_i\otimes {{\hat{b}}})
+
(\sum_{i=1}^{k+1}{{\hat{a}}}\otimes b_i)
\\\nonumber
&\lesssim
&
{\blue{\check{a}\otimes {{\hat{b}}}
+
{{\hat{a}}}\otimes \check{b}}}
{{\leq
(\check{a}+\hat a)\otimes (\check{b}+\hat b)}}
\lesssim
{\blue{a\otimes b
\lesssim
c.}}
\eneq
{\blue{This proves the claim.}}
Then (4) of Definition \ref{def-tracial_approximation} {\blue{holds.}}
{\blue{It follows that}} $A\otimes B$ is asymptotically tracially in
$\mathcal{X}_1\otimes \mathcal{X}_2$.

\end{proof}

\begin{cor}\label{CNntensorNn}
Let $A$ and $B$  be unital separable simple \CA s which are
asymptotically tracially in ${\cal N}_n.$
Then the {{spatial}} 
tensor product $A\otimes B$ is
asymptotically tracially in ${\cal N}_1.$
\end{cor}

\begin{proof}
{\blue{Note that ${\cal N}_n\otimes {\cal N}_n\subset {\cal N}_{2n+1}$
{{(see \cite[Proposition 2.3(ii)]{WZ2010})}}.
Therefore, by {{Theorem}}
\ref{tensor-of-tracial-approximation}, $A\otimes B$ is
asymptotically tracially in  ${\cal N}_{2n+1}.$  By {{Theorem}} \ref{tdimnuc-finite-equal-t-nuclear-Z-stable},
$A\otimes B$ is
asymptotically tracially in ${\cal N}_{2n+1,s,s}.$ It follows from
{{\cite[Corollary C]{CETWW-2019}}}
that $A\otimes B$ is
asymptotically tracially in ${\cal N}_{1,s,s}.$}}

\end{proof}

\begin{cor}\label{CCtZ}
Let $A$ be a unital separable simple \CA\, and let $B$ be a unital separable simple \CA\,
which is
asymptotically tracially in {{${\cal C}_{{\cal Z},s}.$}} Then the
{{spatial}} tensor product $A\otimes B$ is
asymptotically tracially in
{{${\cal C}_{{\cal Z},s}.$}}
\end{cor}

\begin{cor}\label{CTNtNn}
Let $A$ be a unital separable simple \CA\, which is
asymptotically tracially in ${\cal N}$ and let $B$ {\blue{be}} a unital separable
simple \CA\, which is
asymptotically tracially in ${\cal N}_{\cal Z}.$   Then the {{spatial}} 
tensor product $A\otimes B$
is asymptotically tracially in ${\cal N}_1.$
\end{cor}

\begin{cor}\label{CNtNn=N1}
Let $A$ be a unital separable  simple  nuclear \CA\, and $B$ be a unital separable simple \CA\, which is
asymptotically tracially in ${\cal N}_n.$ Then $A\otimes B$ is
asymptotically tracially in ${\cal N}_1.$
\end{cor}

\begin{rem}\label{NnotTN1}

(1) There are unital separable nuclear simple \CA s which are not
asymptotically tracially in ${\cal N}_n$ for any $n\ge {{0}}.$

Let $A$ be one of  Villadsen's examples of unital simple AH-algebras which has stable rank $r>1$ (see \cite{Vill99-stable-rank}).
Then $A$ is nuclear and  it is finite.
However, if $A$ were
asymptotically tracially in ${\cal N}_n$ for some integer $n\ge 0,$
then, by {{Theorem}}
\ref{dichotomy-tracial-finite-dimnuc-finite-infinite}, $A$ would  have stable rank one as it cannot be purely infinite.

(2) There are unital separable nuclear simple \CA s
which have stable rank one but
are not
asymptotically tracially in ${\cal N}_n$ for any $n\ge {{0}}.$

Let $A$ be another construction of Villadsen's AH-algebra
(see \cite{Vill99-perforation})
{\blue{which}} is {\blue{a}}  unital separable nuclear simple {\blue{\CA\,}} and  has stable rank one.
However, $A$ does not have strict comparison for projections,
this fact together with
{{Theorem}}
\ref{dichotomy-tracial-finite-dimnuc-finite-infinite}
and {{Theorem}}
\ref{Thm-tracial-simple-nuclear-Z-stable}
implies that
$A$ is not asymptotically tracially in ${\cal N}_n$ for any $n\ge 0.$

(3) There are unital separable nuclear simple \CA s
which have stable rank one and unperforated $K_0$ group,  but
are not
asymptotically tracially in ${\cal N}_n$ for any $n\ge {{0}}.$

Let $A$ be Toms's construction (see \cite[Corollary 1.1]{Toms08}).
Then $A$ is
{{a}}
unital separable nuclear simple {\blue{\CA\, with}} stable rank {\blue{one}}
which {\blue{has}} {\blue{unperforated}}  $K_0$ group,
but the Cuntz semigroup of $A$ is not almost unperforated.
{{Then, by {{Theorem}}
\ref{dichotomy-tracial-finite-dimnuc-finite-infinite}
and {{Theorem}}
\ref{Thm-tracial-simple-nuclear-Z-stable},
$A$ is not asymptotically tracially in ${\cal N}_n$ for any $n\ge 0.$}}

\end{rem}

\begin{exm}\label{EXLast}
Let $B$ be a unital separable simple \CA\, which has tracial rank zero {\blue{but not exact}}
(see {{\cite{D2000}}}, for example).
Let $C$ be any unital {{nuclear}} separable simple  \CA.
Consider $A=C\otimes B.$
Since $B$ is a non-exact $C^*$-subalgebra of $A$,
it follows that $A$ is not exact
{{(see \cite[6.1.10(i)]{Rordam 2002})}} (thus non-nuclear) {\blue{either.}}
By {{Theorem}}
\ref{tracial-finite-dimnuc-implies-tracial-Z-stable},
$B$ is asymptotically tracially in ${\cal N}_{{\cal Z},s,s}.$
By {{Corollary}} \ref{CCtZ},
$A$ is asymptotically tracially in ${\cal C}_{{\cal Z},s}$.
Since $C$ is nuclear and $B$ is asymptotically {\blue{tracially in ${\cal N}$}},
then, {\blue{by
{{Theorem}} \ref{tensor-of-tracial-approximation},}}
we have {\blue{that}} $A$ is asymptotically tracially in ${\cal N}$.
Then, {\blue{by}} {{Theorem}}
\ref{Thm-tracial-simple-nuclear-Z-stable},
$A$ is asymptotically tracially in
${\cal N}_{{\cal Z}, s,s}.$
This provides a great number of examples of unital separable simple
\CA s which are asymptotically tracially in {{${\cal N}_{{\cal Z}, s,s}$}} but not
{{exact}}.
{{For example, one may choose $C$
to be a unital simple AH-algebra.}}
Moreover,
though $C\otimes B$ are not exact, they are ``regular" in the sense that they
have almost unperforated Cuntz semigroups and has strict comparison.

In a subsequent paper, we will {{show}} {\blue{that}}
unital separable simple \CA s which are not exact but can exhaust all possible
Elliott invaraint{{s}}.

\end{exm}
%



\clearpage

\end{document}